\documentclass[english,10pt]{article}

\usepackage[english]{babel}
\usepackage{eurosym}
\usepackage{amsthm}
\usepackage{paralist}
\usepackage[intlimits]{amsmath}
\usepackage{selinput}
\usepackage{ulem}
\usepackage{amssymb}
\usepackage[linkcolor=black,citecolor=black]{hyperref}
\usepackage{amsfonts}
\usepackage{geometry}

\newtheorem{thm}{Theorem}[section]
\newtheorem{cor}[thm]{Corollary}

\newtheorem{lem}[thm]{Lemma}
\newtheorem{prop}[thm]{Proposition}

\newtheorem{defn}[thm]{Definition}
\theoremstyle{plain}

\theoremstyle{definition}

\newtheorem{rem}[thm]{Remark}

\numberwithin{equation}{section}

\newcommand{\DD}{\mathbb{D}}
\newcommand{\EE}{\mathbb{E}}

\newcommand{\NN}{\mathbb{N}}

\newcommand{\PP}{\mathbb{P}}

\newcommand{\RR}{\mathbb{R}}

\newcommand{\dd}{\mathrm{d}}

\newcommand{\aA}{\mathcal{A}}
\newcommand{\bB}{\mathcal{B}}
\newcommand{\cC}{\mathcal{C}}
\newcommand{\dD}{\mathcal{D}}
\newcommand{\eE}{\mathcal{E}}
\newcommand{\fF}{\mathcal{F}}
\newcommand{\gG}{\mathcal{G}}
\newcommand{\hH}{\mathcal{H}}
\newcommand{\iI}{\mathcal{I}}
\newcommand{\jJ}{\mathcal{J}}

\newcommand{\mM}{\mathcal{M}}

\newcommand{\pP}{\mathcal{P}}

\newcommand{\rR}{\mathcal{R}}
\newcommand{\sS}{\mathcal{S}}

\newcommand{\uU}{\mathcal{U}}
\newcommand{\vV}{\mathcal{V}}

\newcommand{\fS}{\mathfrak{s}}
\newcommand{\fK}{\mathfrak{K}}

\newcommand{\al}{\alpha}

\newcommand{\e}{\varepsilon}
\newcommand{\la}{\lambda}
\newcommand{\La}{\Lambda}
\newcommand{\si}{\sigma}
\newcommand{\om}{\omega}
\newcommand{\Om}{\Omega}
\newcommand{\ka}{\kappa}

\newcommand{\pd}{\partial}

\newcommand{\ra}{\rightarrow}

\newcommand{\lra}{\longrightarrow}
\newcommand{\ti}{\widetilde}
\newcommand{\ha}{\widehat}

\newcommand{\lgl}{\ensuremath{\langle}}
\newcommand{\rgl}{\ensuremath{\rangle}}

\newcommand{\ds}{\displaystyle}

\newcommand{\ind}{\mathbf{1}}

\newcommand{\lqq}{\leqslant}
\newcommand{\gqq}{\geqslant}

\DeclareMathOperator{\DEXP}{EXP}
\DeclareMathOperator{\DGEO}{GEO}

\usepackage[]{color}

\begin{document}
\thispagestyle{empty}

\title{The first exit problem of reaction-diffusion equations \\
for small multiplicative L\'evy noise}

\author{Michael Anton H\"ogele\footnote{Departamento de Matem\'aticas, 
Universidad de los Andes, Bogot\'a, Colombia; ma.hoegele@uniades.edu.co} }

\date{\today}

\maketitle

\begin{abstract}
This article studies the dynamics of a
nonlinear dissipative reaction-diffusion equation 
with well-separated stable states 
which is perturbed by infinite-dimensional 
multiplicative L\'evy noise with a regularly 
varying component at intensity $\epsilon>0$. 
The main results establish the precise asymptotics 
of the first exit times and locus of the solution $X^\epsilon$ 
from the domain of attraction 
of a deterministic stable state, in the limit as $\epsilon\ra 0$. 
In contrast to the exponential growth 
for respective Gaussian perturbations 
the exit times grow essentially as a power function 
of the noise intensity as $\epsilon \rightarrow 0$  
with the exponent given as the tail index $-\alpha$, $\alpha>0,$ of the L\'evy measure, 
analogously to the case of additive noise in Debussche et al \cite{DHI13}. 
In this article we substantially improve their quadratic estimate 
of the small jump dynamics and derive a new exponential estimate 
of the stochastic convolution 
for stochastic L\'evy integrals with bounded jumps 
based on the recent pathwise Burkholder-Davis-Gundy inequality by 
Siorpaes \cite{Siorpaes15}. 
This allows to cover perturbations with general tail index $\alpha>0$, 
multiplicative noise and perturbations of the linear heat equation. 
In addition, our convergence results are probabilistically strongest possible. 
Finally, we infer the metastable convergence 
of the process on the common time scale $t/\epsilon^\alpha$ to a Markov chain 
switching between the stable states of the deterministic dynamical system.
 
\end{abstract}


\noindent \textbf{Keywords:} 
first exit times; first exit locus; 
metastability; nonlinear reaction-diffusion equation; Morse-Smale property; 
small noise asymptotics; $\alpha$-stable L\'evy process in Hilbert space; 
multiplicative L\'evy noise; regularly varying noise; 
stochastic heat equation with additive and multiplicative $\alpha$-stable noise; 
stochastic Chafee-Infante equation with multiplicative noise; 

\bigskip

\noindent \textbf{2010 Mathematical Subject Classification: } 60H15; 60G51; 60G52; 60G55; 35K05; 35K91; 35K57; 35K55; 37D15; 37L55.


\section{Introduction}\label{chapter main results}
This article solves the asymptotic first exit problem  
from the domain of attraction 
of a stable state in a generic class of scalar dissipative reaction-diffusion equations 
subject to small multiplicative regularly varying L\'evy noise, 
such as small multiplicative $\alpha$-stable noise. 
More precisely, the asymptotic first exit time and locus, as the noise intensity $\e$ tends to $0$, 
are determined completely. 

The first exit problem of a randomly perturbed dynamical system from 
the domain of attraction of a stable fixed point in the limit of small noise intensity 
has a long history in finite dimensions for Gaussian perturbations 
going back to the works of Cram\'er and Lundberg 
and giving rise to the edifice of large deviations theory and the associated 
Freidlin-Wentzel theory. We refer the reader to the classical works 
\cite{BerglundG-04, BovierEGK-04, Day-83, Da96, DS01, DZ98, FW70, FreidlinW-98, Kramers-40} 
and the references therein. 
In infinite dimensions this problem was studied  
for the infinite dimensional Wiener processes 
for instance in \cite{BerGen-12,Br91,Br96, FJL82, Fr88}. 
It is a characteristic feature of small Gaussian perturbations 
that the first exit times grow exponentially 
as a function of the inverse of the noise intensity, 
with the prefactor in the exponent 
given as the solution of an optimization problem. 
The convergence of the suitably renormalized process to a Markov chain \cite{GalvesOV-87,KipnisN-85} 
and its connection between the metastability and the spectrum 
of the diffusion generator are treated in \cite{BerGen-10,BovierEGK-04,BovierGK-05,Kolokoltsov-00,KolokoltsovM-96}. 
In the context of regularly varying L\'evy jump noise perturbations, however, 
the perturbed process exhibits heavy tails and therefore 
lacks the necessary exponential moments for a large deviations principle 
(cf. \cite{Applebaum-09, Sato-99}). 

The first exit problem for dynamical systems perturbed by small 
$\alpha$-stable or more generally regularly varying L\'evy perturbations 
was addressed in different settings in a series of works. 
After the early work of \cite{Godovanchuk-82} on a large deviations principle in Skorohod space  
the first exit times problem is solved in one dimension for additive $\alpha$-stable noise in \cite{IP06}. 
The authors introduced the following purely probabilistic proof technique also used and extended in this article, 
which we sketch briefly: 

Given an $\alpha$-stable noise perturbation $\e d L$ 
the first step is the choice of an $\e$-dependent jump size threshold $\rho^\e$, 
which decomposes the driving noise into the sum $\e d \xi^\e + \e d \eta^\e$, $\xi^\e$ 
being an infinite intensity 
process with jumps bounded from above by the threshold $\rho^\e$ 
and thus exhibiting exponential moments and $\eta^\e$ the compound Poisson 
process of jumps bounded from below by the threshold. 
Assuming that $\e \rho^\e$ tends to $0$ as $\e$ tends to $0$ 
and taking into account that $\xi^\e$ has exponential moments, 
it does not come as a surprise that up 
to the first large compound Poisson jump of $\eta^\e$, 
the process $\e \xi^\e$ is very small. 
Hence the resulting flow decomposition 
of the strong Markov solution $X^\e$ yields 
that up to the first jump of $\e \eta^\e$, 
the process $X^\e$ remains 
close to the deterministic solution 
with overwhelming probability.  
Therefore in the vast majority of cases 
it cannot cause the exit from the domain of attraction. 
Due to the Morse-Smale property 
and the choice of the noise decomposition 
the convergence of the deterministic solution 
to a small ball centered in the stable state 
is faster than the first large jump time.  
As a consequence, the first large compound Poisson jump 
starts from close to the stable state and yields 
an exit probability of the first large jump 
in terms of the tail of the L\'evy measure. 
The strong Markov property propagates this exit scenario to all independent 
waiting time intervals between the large jumps. 
The exit is hence caused with very high probability 
by the first successful attempt of a large jump to exit. 
The resulting geometric exit structure of the exit times happens 
at a rate given by the tail decay of the L\'evy measure governing the large jumps 
which in the case of regular variation is of polynomial order. 
In \cite{Pa11} the author shows this result for gradient systems 
in any finite dimension and multiplicative noise; in particular, 
he derives an exponential estimate for small deviations from the deterministic system. 
He obtains exponential estimates for the small noise components, 
however his treatment of the small jump component depends on the dimension of the driving noise 
and is not suitable in infinite dimensions. 
In \cite{HP13} the results are generalized to the non-gradient case in finite dimensions, 
in addition the convergence in law of the first exit locus is proved. 
The well-posedness of reaction-diffusion equations in infinite dimensions 
in a generic setting is established in \cite{BHR18, MPR10, MR10, MR12, PZ07}. 
In the infinite dimensional situation 
summarized in \cite{DHI10} and explained in detail in \cite{DHI13} 
the authors consider the first exit times of the Chafee-Infante equation 
perturbed by small additive regularly varying noise.  
Their treatment of the small noise dynamics remains elementary 
with quadratic deviation estimates 
and precisely for this reason only allows 
for tail indices $-\al$ for $0 < \al <2$ there. 

This article provides a substantial extension of these results in several directions. 
We extend the scope of the deterministic forcing of \cite{DHI13} 
to a general class of weakly dissipative 
non-linear reaction terms over an interval with Dirichlet boundary conditions 
for which the system retains the Morse-Smale property 
of the deterministic system. The most important cases 
covered here are dissipative polynomials of odd order, 
such as for the Chafee-Infante equation, 
and the linear heat equation. 
Our results are stated for the Laplace operator with Dirichlet conditions on 
the Sobolev space $H^1_0$ over the standard interval $[0, 1]$. 
We expect the results to hold true for any unbounded operator with negative point spectrum~$A$ 
which generates a generalized contracting analytic semigroup. 
However, we use the Morse-Smale property of the deterministic dynamical system in $H^1_0$ 
as well as the smoothness of the separating manifold of the domains of attractions, 
and to our knowledge these results are not readily available 
in the literature for general spaces $D(A^\frac{1}{2})$. 

The generalizations of the type of stochastic perturbations are twofold. 
In the first place we study multiplicative noise coefficients as opposed to \cite{DHI13}. 
They are the original motivation of this article and 
make it necessary to consider the first exit problem localized on large balls. 
Consequently we get rid of the rather strong point dissipativity 
of the deterministic dynamical system, and also treat 
the important new example of the linear heat equation 
subject to additive and multiplicative $\alpha$-stable noise. 

Secondly, we lift the rather strong restriction of a tail index $0 < \al < 2$ 
in \cite{DHI13} with the help of an exponential estimate 
of the stochastic convolution for multiplicative 
Poisson random integrals with bounded jumps. 
It combines the recent pathwise estimate 
of the stochastic convolution in \cite{SZ16} 
and the pathwise Burkholder-Davis-Gundy inequality in \cite{Siorpaes15}. 
Since multiplicative noise necessarily leads to working 
with stopped processes, 
the non-pathwise estimates of the stochastic convolution 
available until then were rather difficult to implement. 
With these new powerful tools at hand an almost sure 
estimate in the exponent of an exponential moment yields a 
lift of the right side of the Burkholder-Davis-Gundy inequality to the exponent, 
which is estimated with the help of a Campbell type formula 
for the Laplace-transform of Poisson random integrals. 
Our estimates are rather direct and avoid the adaption 
of the technically charged large deviation theory introduced by Budhiraja 
and collaborators; see for instance \cite{BN15, BCD13, BDG16}. 
In comparison to those works we 
construct explicitly (on the same probability space as the driving noise) 
a completely understood model of the first exit times and locus respectively 
to which the original objects converge. 
Our convergence results are optimal in a probabilistic sense, 
in that we obtain exponential convergence up to all exponents strictly less than $1$, 
while the limiting object does not have exponential moments of order $1$. 
The same applies to the convergence of the first exit locus, 
which is essentially a geometric mixing of deformed 
large jump increments of the noise. 
Those increments with tail index $-\al$, $\al>0$, 
have moments of order $0< p<\al$ 
and we show convergence in any such $L^{p}$-sense towards the limiting object. 
Finally we infer metastability in the sense of \cite{HP15} and \cite{IP08} as a corollary. 

The article is organized as follows. In Section \ref{sec: preliminaries} 
we present the general setup, the specific hypotheses, the main results and examples. 
The proof relies on the mentioned $\e$-dependent distinction of large and small jump perturbations. 
In Section \ref{sec: small deviations} we prove 
an exponential error probability estimate on the smallness 
of the stochastic convolution between large jumps 
and its pushforward to the nonlinear equation. 
In Section \ref{sec: proofs} we use the preceding result  
which yields an asymptotic compound Poisson noise structure
that essentially contains only large jumps. 
With the help of the strong Markov property and 
tailor-made event estimates 
we identify the asymptotic first exit mechanism of the solution 
of the fully perturbed nonlinear equation. 

\section{The object of study and the main result \label{sec: preliminaries}}

\paragraph{Notation: } For $J = (0,1)$ we consider the Sobolev space $H:=H_0^1(J)$ 
equipped with 
the inner product $\lgl\!\lgl x, y\rgl\!\rgl = \lgl \nabla x, \nabla y\rgl$ 
for $x, y\in H$ and the norm $\|x\| = \lgl\!\lgl x, x\rgl\!\rgl^\frac{1}{2}$,  
where $\lgl \cdot, \cdot \rgl$ is the inner product in $L^2(J)$ with $|x| = \lgl x, x \rgl^\frac{1}{2}$. 
Let $\cC_0(\bar J)$ be the space of
continuous functions $x: \bar J\ra \RR$ with $x(0) = x(1) =0$ 
equipped with the supremum norm $|\cdot|_\infty$.
Since $|x| \lqq |x|_\infty \lqq \|x\|$ for $x\in H$
we have the embeddings $H \hookrightarrow \cC_0(\bar J) \hookrightarrow L^2(J)$, in particular, 
$|x| \lqq \La_0 \|x\|$ for all $x\in H$ and the Sobolev constant $\La_0>0$.

\subsection{The underlying deterministic dynamics}\label{subsec: deterministic dynamics}

\paragraph{The unperturbed PDE: } 
The object of study is the effect of random perturbations of the 
deterministic dynamical system given for any $t\gqq 0$ 
as the solution map $x \mapsto u(t;x)$ of the 
following nonlinear reaction-diffusion equation over the interval $J$ 
with Dirichlet boundary conditions. We consider 
\begin{align}\label{eq: deterministic equation}
\begin{split}
\frac{\pd}{\pd t} u(t, \zeta) &= \Delta u(t, \zeta)  + f(u(t,  \zeta))\quad 
\mbox{ with } \quad u(t, 0) =  u(t, 1) = 0 \quad \mbox{ and } \quad u(0, \zeta; x) =  x(\zeta),
\end{split}
\end{align}
for $t\gqq 0$, $x\in H$ and $\zeta \in J$, where the non-linearity $f \in \cC^2(\RR, \RR)$ satisfies the 
growth condition  
\begin{equation}
\limsup_{|r|\ra\infty} f'(r) < \La_0. \label{eq: growth rate}
\end{equation}
This equation has unique and well-posed 
weak and mild solutions in $L^2(J)$ and $H$ (cf. \cite{DHI13, Te97}). 
The solutions are most regular for any $t>0$ and $x\in L^2(J)$, that is, $u(t;x) \in \cC^\infty(J) \cap \cC_0(\bar J)$.  

\begin{rem}
In case of $f(r) = \sum_{j=0}^{2n-1} b_j r^j$ with $b_{2n-1} <0$ for $n\in \NN$
it is well-known in the literature \cite{Ha99, He85, Ro88} that 
for a generic choice of $(b_0, \dots, b_n)\in \RR^{n+1}$, 
that is, up to a nowhere dense set, 
the dynamical system generated by (\ref{eq: deterministic equation}) 
is of Morse-Smale type. In other words, there is a finite number 
of fixed points, all of which are hyperbolic and whose stable and unstable manifold 
intersect transversally. 
The paradigmatic example of the Chafee-Infante equation is studied in \cite{CI74, He85} 
where $f(r) = -a (r^3-r), r\in \RR$, whenever $a> \pi^2$ and $a \neq (\pi n)^2$, $n\in \NN$. 
Since all finitely many equilibria are elements of $H \subseteq L^\infty(J)$, 
the Morse-Smale property only involves $f$ on a bounded set. 
On bounded sets a general function $f \in \cC^2(\RR, \RR)$ 
is approximated in $\cC^2$-norm by polynomials 
and the generic Morse-Smale property is inherited by $f$. 
\end{rem}

\paragraph{The deterministic dynamics: } 
It is well-known that the solution of equation (\ref{eq: deterministic equation}) has 
the nonnegative potential function 
$\vV(x) = \int_J \big((\nabla x(\zeta))^2 + F(x(\zeta))\big) d\zeta$ on $H$ 
where $F(r) = \int_{r_0}^r f(s) ds$ for some $r_0$.  
Therefore, equation (\ref{eq: deterministic equation}) reads as the gradient system 
\begin{align*}
\frac{\pd}{\pd t} u(t, \zeta) &= - (D \vV)(u(t,\zeta))\qquad \mbox{ with } \quad u(0, \zeta; x) =  x(\zeta) \quad \mbox{ for } x\in H.
\end{align*}
The level sets of $\vV$ are bounded in $H$ 
and positive invariant under the system (\ref{eq: deterministic equation}). For $r>0$ set 
\begin{equation}
\uU^{r} := \{x \in H\,|~\vV(x) \lqq d_*(r)\},\quad d_*(r) := \inf\{s>0\,|~B_r(0) \subseteq \vV^{-1}[0, s]\},\quad d(r) := \sup_{x\in \uU^r}\|x\|. \label{eq: level sets}
\end{equation} 
As a consequence, $\vV$ serves as a Lyapunov function and yields 
the following result (cf. \cite{Ha99, He83}).

\begin{prop}\label{pointwise convergence}
Denote by $\pP\subseteq H$ the set of fixed points of (\ref{eq: deterministic equation}).
Then $0< |\pP|< \infty$ and for any $x \in H$ there exists a stationary state
$\phi \in \pP$ of the system (\ref{eq: deterministic equation}) such that
$\lim_{t\ra\infty} u(t;x) = \phi.$
\end{prop}

\noindent For $\phi\in \pP$ we define the domain of attraction $D(\phi) := \{x\in H~|~\lim_{t\ra \infty} u(t;x) = \phi\}.$ 
The set of stable states is the subset $\pP^-$ 
of all $\phi\in \pP$ such that $D(\phi)$ contains an open ball in $H$.  
For $\phi^\iota \in \pP^-$, $1\lqq \iota \lqq |\pP^-|$, 
we denote its domain of attraction $\dD^\iota = D(\phi^\iota)$ and the 
separating manifold between them by $\sS := H\setminus \bigcup_{\iota} \dD^\iota$. 
For a generic choice of coefficients 
the Morse-Smale property implies that $\sS$ is a closed
$\cC^1$-manifold without boundary in $H$ of codimension $1$ separating 
all elements of $(\dD^\iota)_{\phi^\iota \in \pP^-}$ and containing all unstable fixed points $\pP \setminus \pP^-$  
(cf. \cite{Ra01}). 

\paragraph{Reduced domains of attraction: } 
Note that $f: H\ra H$ is locally Lipschitz continuous.  
For any subset $D^\iota \subseteq \dD^\iota$ with $\cC^1$ boundary, 
such that $\Delta + f$ on $\pd D^\iota$ is uniformly inward pointing we have 
\begin{equation}\label{eq: uniformly inward pointing}
\kappa_1 := \inf_{v\in \pd D^\iota \cap \cC^3_0(I)} \lgl\!\lgl n^\iota(v), \frac{\Delta v + f(v)}{\|\Delta v + f(v)\|}\rgl\!\rgl  > 0,
\end{equation}
where $n^\iota(v)$ is the normalized inner normal at the foot point $v\in \pd D^\iota$ and $\cC^3_0(I) = \cC^3(I) \cap \cC_0(\bar I)$. In the sequel we 
define the following nested reduced domains of attraction of $D^\iota$ in order to formulate 
the nondegenericity of the noise perturbations in Hypothesis (S.4) below.   
Fix a radius $\rR_0>0$ such that $\pP \subseteq B_{\rR_0/2}(0)$ and $\uU^{\rR_0} \cap \pd D^\iota \neq \emptyset$ 
for all $\phi^\iota \in \pP^-$.  
We define for $\delta_i>0$, $i=1, \dots, 3$, $\rR\gqq \rR_0$ and $G$ the function appearing in \ref{main sys} below 
the following reductions of $D^\iota$: 
\begin{equation}\label{def: reduced domains}
\begin{split}
D_1^\iota(\rR) &:= D^\iota \cap \uU^{\rR}, \\
D_2^\iota(\delta_1, \rR) &:= \{x \in D_1^\iota(\rR)~|~B_{\delta_1}(x) \subseteq D_1^\iota(\rR)~\},\\
D_3^\iota(\delta_1, \delta_2, \rR) &:= \{x \in D_2^\iota(\delta_1, \rR)~|~\bigcap_{v\in B_{\delta_2}(x)} \{v+G(v, z) \} \subseteq D_2^\iota(\delta_1, \rR)~\}.
\end{split}
\end{equation}
For convenience we set $D_3^\iota(\delta_1, \rR) := D_3^\iota(\delta_1, \delta_1, \rR)$.  
The reduced domains of attraction 
are nested by construction and $D^\iota = \bigcup_{\rR\gqq \rR_0, \delta \in (0, 1]} D_3^\iota(\delta, \rR)$ (cf. \cite{DHI10}).  
For any $\rR>0$ and $\delta\in (0, \delta_0]$, $\delta_0\in (0, 1]$ sufficiently small, 
the reduced domains of attraction $D_3^\iota(\delta, \rR)$ (and $D_2^\iota(\delta, \rR)$) 
are positive invariant under the dynamical system (\ref{eq: deterministic equation}) 
due to the uniformly inward pointing property of $f$ on $\pd D^\iota$. 

\begin{prop}\label{prop: logarithmic convergence time}
For any choice of $f$ such that (\ref{eq: deterministic equation}) is Morse-Smale, 
$D^\iota\subseteq \dD^\iota$ with $\pd D^\iota \in \cC^1$ satisfying (\ref{eq: uniformly inward pointing}) 
and $\rR\gqq \rR_0$ there exists a constant $\kappa_0>0$ which satisfies the following. 
For any function $\gamma_\cdot: (0, 1] \ra (0,1)$ with $\lim_{\e\ra 0} \gamma_\e = 0$ 
there is a constant $\e_0 \in (0,1]$ such that 
for each $\e \in (0, \e_0]$ the conditions $t\gqq  \kappa_0 |\ln(\gamma_\e)|$ and
$x\in D_1^{\iota}(\rR)$ imply $\|u(t;x)-\phi^\iota\| < \frac{1}{4} \gamma_\e.$ 
In addition, (\ref{eq: deterministic equation}) is Morse-Smale if 
and only if the equilibrium points are hyperbolic. 
\end{prop}

\noindent This result is based on the existence of a Lyapunov function, 
the uniform inward pointing property of $f$ on $\pd D^\iota$, 
and the hyperbolicity of the fixed points. 
In \cite{DHI10} it is shown for a stronger form of approximation 
for the Chafee-Infante equation. Its generalization is straightforward. 
The second part of the statement is given by Theorem 2.2.1 in \cite{Ha99} 
and the references therein. 

\subsection{The stochastic reaction-diffusion equation}\label{subsec: SPDE}

\paragraph{The L\'evy driver: } Given a filtered probability space $\mathbf{\Om} = (\Om, \aA, \PP, (\fF_t)_{t\gqq 0})$ satisfying 
the usual conditions in the sense of Protter \cite{Pr04}, 
let $L = (L(t))_{t\gqq 0}$ be a c\`adl\`ag version of a L\'evy process in $(H, \bB(H))$. 
We denote by $\mM_0(H)$ the class of Radon measures $\nu$ on $\bB(H)$ satisfying
$\nu(A) <\infty$ for $A\in \bB(H)$ with $0 \notin \bar A$.
The L\'evy-Chinchine representation establishes a unique L\'evy triplet $(h, Q, \nu)$ 
with $h\in H$, a positive trace-class operator $Q\in L_1^+(H)$ and $\nu \in \mM_0(H)$ satisfying $\nu(\{0\}) = 0$ and $\int_H (1\wedge \|y\|^2)\nu(\dd y) <\infty$
such that the characteristic function $\phi_{L(t)}(u) := \EE\big[\exp(i\lgl\!\lgl u, L(t)\rgl\!\rgl)\big]$ has the exponent 
\[
t \Big(i \lgl\!\lgl h, u\rgl\!\rgl - \frac{1}{2} \lgl\!\lgl Q u, u\rgl\!\rgl 
+ \int_{H} \big(e^{i\lgl\!\lgl u, z\rgl\!\rgl} -1 - i \lgl\!\lgl z, u \rgl\!\rgl \ind\{\|z\| \lqq 1\}\big) \nu(dz)\Big),
\qquad u\in H, ~t\gqq 0. 
\]
By the L\'evy-It\={o} representation of $L$ there exist a $Q$-Wiener process $(B_Q(t))_{t\gqq 0}$ 
and a Poisson random measure $N$ on $\mathbf{\Om}$ with intensity measure $dt \otimes \nu(dz)$ on 
$[0, \infty)\times H$ 
such that $\PP$-a.s.  
\begin{align}\label{eq: Levy-Ito} 
L(t) = h t + B_Q(t) + \int_0^t \int_{\|z\| \lqq 1} z \ti N(dsdz) + \int_0^t \int_{\|z\|>1} z N(dsdz), \qquad \mbox{ for all } t\gqq 0, 
\end{align}
where $\ti N([a, b) \times A) := N([a, b) \times A) - (b-a) \nu(A)$ for $a\lqq b, A\in \bB(H)$, $0 \notin \bar A$ 
is the compensated Poisson random measure of $N$. 
For a comprehensive account of L\'evy processes in Hilbert spaces with refer to \cite{PZ07}. 
In this study we set $h = 0$ and $B_Q = 0$, 
since their exit contributions are asymptotically insignificant 
compared to the pure jump part, if $\nu\neq 0$. 

The multiplicative nonlinearity is given as a 
map $G: H \times H\ra H$ which satisfies the following standard boundedness and Lipschitz conditions. 
There are constants $K_1, K_2>0$ such that 
\begin{align}
&\int_{B_1(0)} \|G(x, z)\|^2 \nu(dz) \lqq K_1(1+ \|x\|^2) \qquad \mbox{ for all }x \in H,\label{eq: boundedness}\\
&\|G(x_1, z) - G(x_2, z)\|  \lqq K_2 \|x_1- x_2\|  \quad \mbox{ for all }x_1, x_2, z \in H. \label{eq: Lipschitz}
\end{align}

\paragraph{The perturbed equation: } We consider the formal stochastic reaction diffusion equation for $t>0, x\in H, \zeta \in J$ and $\e\in (0, 1]$ 
\begin{equation}\label{main sys}
\begin{split}
\ds dX^\e(t,\zeta) &= \big(\ds \Delta X^\e(t,\zeta) + f(X^\e(t,\zeta))\big) dt  + 
G(X^\e(t-, \zeta), \e dL(t, \zeta)) \\[3mm]
&\mbox{ with }\quad X^\e(t,0) = X^\e(t,1) = 0 \quad \mbox{ and } \quad X^\e(0,\zeta) = x(\zeta),
\end{split}
\end{equation}
where 
\[
G(X^\e(t, \zeta), \e dL(t, \zeta)) = \int_{|z|\lqq 1}  G(X^\e(t-, \zeta), \e z(\zeta)) \ti N(dtdz)+ \int_{|z| > 1}  G(X^\e(t-, \zeta), \e z(\zeta)) N(dtdz).
\]

\begin{prop} Assume the setting of Subsection \ref{subsec: deterministic dynamics}, in particular,
the growth rate (\ref{eq: growth rate}) for $f\in \cC^2(\RR,\RR)$ 
and the conditions (\ref{eq: boundedness}) and (\ref{eq: Lipschitz}) of $G$. 
Then for any mean zero c\`adl\`ag $L^2(\PP; H)$-martingale $\xi = (\xi(t))_{t\gqq 0}$ on $\mathbf{\Om}$, 
\mbox{$T>0$,} and initial value $x\in H$, equation (\ref{main sys}) driven
by $\e d\xi$ instead of $\e dL$ has a unique c\`adl\`ag mild solution
$(X^\e(t;x))_{t\in[0,T]}$. The transition kernels of the 
solution process $X^\e$ induce a
homogeneous Markov family satisfying the Feller property and hence the strong Markov property.
\end{prop}

\noindent The proof relies on the local Lipschitz continuity and the dissipativity 
of $f: H\ra H$. 
A proof for dissipative polynomials $f$ is given in \cite{PZ07}, Chapter 10, 
and for the Chafee-Infante equation in \cite{DHI10} and can be extended to our situation 
straightforwardly. 
By interlacing of large jumps, this notion of solution is extended to the heavy-tailed process $L$, 
as carried out in \cite{PZ07}, Subsec. 9.7, pp.170. 

\begin{cor}
For $x\in H$ equation (\ref{main sys}) has a global c\`adl\`ag mild solution $(X^\e(t;x))_{t\gqq 0}$,
which satisfies the strong Markov property.
\end{cor}

\subsection{The specific hypotheses and the main results} 

Under the standing assumptions of Subsection \ref{subsec: deterministic dynamics} and \ref{subsec: SPDE} 
we impose the following additional hypotheses. 

\paragraph{Hypotheses: } 

\begin{enumerate}
 \item[(D.1)]  
\textit{The function $f$ 
is generic in the sense that the solution flow of (\ref{eq: deterministic equation}) defines 
a Morse-Smale system. In addition, we assume $2\lqq |\pP^-|< \infty$.} 

 \item[(D.2)] \textit{The function $f$ satisfies the eventual monotonicity condition (\ref{eq: growth rate}).} 

 \item[(D.3)] \textit{Consider a subset $D^\iota \subseteq \dD^\iota$ with $\cC^1$-boundary such that 
 the operator $\Delta + f$ is uniformly inward pointing on $\pd D^\iota$ 
 in the sense of (\ref{eq: uniformly inward pointing}). 
 } 

 \item[(S.1) ] \textit{There is a globally Lipschitz continuous 
function $G_1: H\ra [0, \infty)$ such that 
\[
\|G(y, z) \| \lqq G_1(y) \|z\|, \qquad y, z\in H. 
\]
}
\end{enumerate}
\begin{enumerate}
 \item[(S.2)] \textit{The L\'evy measure $\nu \in \mM_0(H)$ of $L$ 
is regularly varying with index $-\al$, $\al>0$, and limit measure $\mu\in \mM_0(H)$.} 
\end{enumerate}
See for instance Def 3.44, p. 67, in \cite{DHI13}.  
By \cite{BGT87} and \cite{HL06} Hypothesis (S.2) is equivalent to the existence of 
measurable functions $h, \ell: (0, \infty) \ra (0, \infty)$ such that 
\begin{align*}
\lim_{r \ra \infty} \frac{\nu(r U)}{h(r) \mu(U)} = 1 \quad \mbox{ for any }U \in \bB(H) \mbox{ with } 0 \notin \bar U,
\end{align*}
where $h(r) = r^{-\al} \ell(r)$ and $\ell$ is a slowly varying function. 
In other words, $\lim_{r\ra \infty} \frac{\ell(a r)}{\ell(r)} = 1$ for any $a>0$. 
We define the set of increment vectors $z\in H$ sending $x\in H$ to the set $U\in \bB(H)$ as 
\[
\jJ^U(x) := \{z\in H~|~x+ G(x,z) \in U\}, \qquad x\in H.   
\]
For any $\phi^\iota \in \pP^-$ we denote the measure 
$m^\iota(U) := \mu(\jJ^U(\phi^\iota))$, $U\in \bB(H)$ 
with $0 \notin \bar U$, and the scale $h_\e := h(\frac{1}{\e})$ for $\e \in (0, 1]$.
\begin{enumerate}
 \item[(S.3) ] \textit{For all $\phi^\iota \in \pP^-$ and $\rR \gqq \rR_0$ we have $m^\iota((D^\iota \cap \uU^\rR)^\mathsf{c}) >0$.}
 \item[(S.4) ] \textit{For $D^\iota$ in (D.3) 
 and all $\eta>0$ there are $\delta>0$ and $\rR\gqq \rR_0$  
such that $m^\iota(D^\iota\setminus D_3^\iota(\delta,\rR))< \eta.$}
\end{enumerate}
Hypothesis (S.3) states that asymptotically there is some non-vanishing mass for large jumps to exit, 
while (S.4) codes the nondegeneracity of the limiting L\'evy measure on the boundary $\pd D^\iota$, 
in order to allow for the approximations of $D^\iota$ by $D^\iota_3(\delta, \rR)$ in terms of $m^\iota$. 

\begin{enumerate}
 \item[(S.5)] For any $\eta>0$ and $\iota$ 
 there are sets $D^\iota \subseteq \dD^\iota$ satisfying (D.3) 
 as well as $\delta>0$ and $\rR\gqq \rR_0>0$ such that 
 \begin{align*}
 m^\iota(H \setminus \bigcup_{\iota} D_3^\iota(\delta,\rR))< \eta.
 \end{align*}
\end{enumerate}
Hypothesis (S.5) is a sort of uniform version of (S.4) for all 
domains of attraction $\dD^\iota$.

\paragraph{The first exit time result: } 
For $\gamma, \e \in (0,1]$, $\rR\gqq \rR_0$, \mbox{$x\in D_2^\iota(\e^\gamma, \rR)$}
and the c\`adl\`ag mild solution $(X^\e(t;x))_{t\gqq 0}$  of (\ref{main sys}) 
we define the \textit{first exit time from the reduced
domain of $D^\iota$}
\[
\tau^\iota_x(\e, \rR) :=\inf\{t>0~|~X^\e(t;x)\notin D_2^\iota(\e^\gamma, \rR)\}.
\]
We define the characteristic exit rate $\la^\iota_\e$ of system (\ref{main sys})
from $D^\iota$ by
\begin{align}\label{eq: exit rate lambda}
\la^\iota_\e := \nu\Big(\frac{1}{\e} \jJ^{(D^{\iota})^\mathsf{c}}(\phi^\iota)\Big), \qquad \e \in (0, 1]. 
\end{align}
Then (S.2) implies 	
$
\frac{\la^\iota_\e}{h_\e} = \frac{\la^\iota_\e}{\e^\al \ell(\frac{1}{\e})} \stackrel{\e\ra 0+}{\lra}  m^\iota((D^\iota)^\mathsf{c}).
$
Our asymptotic exit time result reads as follows.

\begin{thm} \label{main result}
Let Hypotheses (D.1)-(D.3) and (S.1)-(S.4) be satisfied for some $\iota$.  
Then there is an $\DEXP(1)$-distributed family of 
random variables $(\fS^\iota(\e))_{\e \in (0, 1]}$ on $\mathbf{\Om}$ satisfying the following. 
For any $c>0$ and $\theta \in (0,1)$ there are $\rR\gqq \rR_0$ and $\e_0, \gamma \in (0, 1]$ 
such that $\e\in (0, \e_0]$ implies 
\begin{align*}
\sup_{x\in D_2^\iota(\e^\gamma, \rR)} \EE\Big[e^{\theta |\la^\iota_\e \tau_x^\iota(\e, \rR)- \fS^\iota(\e)|} 
\Big] \lqq 1+c. 
\end{align*}
Therefore, we have the convergence 
of all moments $\lim_{\e\ra 0} \EE[|\la^\iota_\e \tau_x^\iota(\e, \rR)|^n]\in [n!-c, n!+c]$ 
uniformly in $D_2^\iota(\e^\gamma, \rR)$ and the following polynomial behavior 
\[
\sup_{x\in D_2^\iota(\e^\gamma, \rR)} \EE\big[ \tau_x^\iota(\e, \rR)\big] 
\in [\frac{1-c}{\la^\iota_\e}, \frac{1+c}{\la^\iota_\e}] 
\subseteq [\frac{1-2c}{\e^\al \ell(\frac{1}{\e}) m^\iota((D^\iota)^\mathsf{c})}, \frac{1+2c}{\e^\al \ell(\frac{1}{\e}) 
m^\iota((D^\iota)^\mathsf{c})}],
\]
where the supremum can be changed to the infimum. 
\end{thm}
In terms of \cite{Br96} the memorylessness of $\fS^\iota(\e)$ describes  the ``unpredictability'' of the exit times, 
however, with a ``polynomial'' loss of memory as opposed to a ``exponential'' loss of memory in the case of Gaussian perturbations. 

\paragraph{The first exit locus result: }
For the statement of the main result about the exit locus we write  $\Delta_t L := L(t) - L(t-)$, $t\gqq 0$ 
for $L$ given by (\ref{eq: Levy-Ito}). For some $\rho \in (0,1)$ and $\e \in (0, 1]$ 
we define large jump arrival times of $L$ by 
\begin{equation}\label{def: large jumps} 
T_0(\e) :=0, \qquad T_{k}(\e) := \inf\left\{t>T_{k-1}(\e)~\big|~\|\Delta_t L\| > \e^{-\rho} \right\}, \quad k\gqq 1,
\end{equation}
and large jump increments by $W_k(\e) := \Delta_{T_k(\e)} L$, $k\in \NN$. 
The family $(W_k(\e))_{k\in \NN}$ is i.i.d. with 
\[
\PP(\e W_k(\e) \in U) = \frac{\nu(U \cap B_{\e^{-\rho}}^\mathsf{c}(0))}{\beta_\e} , \mbox{ where } \beta_\e := \nu(B_{\e^{-\rho}}^\mathsf{c}(0)) 
\mbox{ satisfies } \frac{\beta_\e}{\e^{\al\rho} \ell(\e^{-\rho})} \stackrel{\e\ra 0+}{\lra} \mu(B_1^\mathsf{c}(0)).
\]
We drop the $\e$-argument of $T_k$ and $W_k$. The asymptotic exit locus result theorem reads as follows.

\begin{thm} \label{main result 2}
Let Hypotheses (D.1)-(D.3) and (S.1)-(S.4) be satisfied for some $\iota$. Then there is a 
family of random variables $(\fK^\iota(\e))_{\e\in (0, 1]}$ on $\mathbf{\Om}$ with $\fK^\iota(\e)$ being $\DGEO(\la_\e^\iota / \beta_\e)$ 
distributed and satisfying the following. 
For any $c>0$ and $0 < p < \al$ there are $\rR\gqq \rR_0$ and $\gamma, \rho, \e_0  \in (0, 1]$ such that 
$\e\in (0, \e_0]$ implies
\begin{align*}
\sup_{x\in D_2^\iota(\e^\gamma, \rR)} \EE\Big[\big\|X^\e(\tau_x^\iota(\e, \rR); x)- (\phi^\iota + G(\phi^\iota, \e W_{\fK^\iota(\e)}))\big\|^{p} \Big] \lqq c.  
\end{align*}
For any $c>0$ there are $\rR\gqq \rR_0$ and $\gamma, \rho \in (0, 1]$ 
such that for any $U \in \bB(H)$ with $m^\iota(U)>0$ and $m^\iota(\pd U) = 0$ 
there is $\e_0 \in (0, 1]$ such that $\e \in (0, \e_0]$ yields 
\[
\sup_{x\in D_2^\iota(\e^\gamma, \rR)} \Big|\PP( X^\e(\tau_x^\iota(\e, \rR);x) \in U) - \frac{m^\iota(U \cap (D^\iota)^\mathsf{c})}{ṃ^\iota((D^\iota)^\mathsf{c})}\Big| \lqq c. 
\]
\end{thm}

\paragraph{The metastability result: } 
Under Hypothesis (S.5) we have a good approximation of $\dD^\iota$ by 
sets $D^\iota_3(\delta, \rR)$ with inward pointing $\Delta + f$ at its 
boundary in the sense of (\ref{eq: uniformly inward pointing}). 
Hence the exit rate $\e \mapsto \la^\iota_\e \approx_\e \e^\al \ell(\frac{1}{\e}) m^\iota((\dD^\iota)^\mathsf{c})$ 
asymptotically depends on $\iota$ only by a prefactor 
and the process $X^\e$ converges on the common (polynomial) 
time scale $t/\e^\al$ to a continuous time Markov chain whose transition probabilities 
from $\dD^\iota$ to $\dD^\kappa$ only depend 
on the values $(m^\iota(\dD^\kappa))$, $\kappa \in \{1,\dots, |\pP^-|\}\setminus \{\iota\}$. 
This behavior is typical for regularly varying L\'evy noises 
such as $\alpha$-stable noise \cite{DHI13, HP15, IP08} 
and differs strongly from the Gaussian case (see in particular \cite{GalvesOV-87}).  
In the introduction of \cite{HP15} it is explained in detail 
how this behavior corresponds to the degenerate Gaussian case 
where the time scales are comparable which occurs if and only if 
the potential barriers are all of exactly the same height. 
The asymptotic metastability result reads as follows. 

\begin{cor}\label{cor: metastability}
Let Hypotheses (D.1)-(D.3), (S.1)-(S.3) and (S.5) be satisfied.
Then there exists a continuous time Markov chain $M = (M_t)_{t \gqq 0}$ 
with values in $\pP^-$ and 
infinitesimal generator $\gG$ for $p = |\pP^-|$ given as the matrix  
\begin{align}
\label{eq:Q}
\gG 
= \begin{pmatrix} 
 -m^1\left(\left(D^1\right)^\mathsf{c}\right)   & m^1\left(D^2\right) & \dots & & m^1(D^{p}) 
\\ 
\vdots&&&&\vdots\\
m^{p} \left(D^1\right)  && \dots  & m^{\kappa}(D^{\kappa-1}) & -m^{p}\left(\left(D^{p}\right)^\mathsf{c}\right)   
  \end{pmatrix},
\end{align}
and a constant $\gamma>0$ such that for any $\rR\gqq R_0$, 
$T>0$ and $x\in D_2^\iota(\gamma, \rR)$, $1\lqq \iota \lqq |\pP^-|$, we have 
$(X^\e(\frac{t}{\e^\al}; x))_{t\in [0, T]} \stackrel{d}{\lra} (M_t)_{t\in [0, T]}$ in the 
sense of convergence of finite-dimensional distributions. 
\end{cor}

The proof of this result is an analogous construction as in \cite{HP15} and 
does not depend on the fact that the system is infinite-dimensional 
and follows in a fairly straightforward manner. 

\bigskip
\subsection{Examples} 
\subsubsection{The Chafee-Infante equation with multiplicative stable noise }

We consider equation (\ref{eq: deterministic equation}) 
for $f(r) = -a (r^3-r)$, $r \in \RR$, where $a> \pi^2$ and $a \neq (\pi n)^2$ for all $n\in \NN$. 
The respective deterministic dynamical system is a well understood Morse-Smale system 
and has two stable states $\phi^\iota, \iota \in \{+, -\}$ 
with domains of attraction $D^\pm$, see \cite{Ha99, He83}. 

The stochastic perturbation of interest is an $H$-valued symmetric $\alpha$-stable L\'evy process 
$L(t) = \int_0^t \int_{\|z\|\lqq 1} z \ti N(dsdz) + \int_0^t \int_{\|z\|> 1} z N(dsdz)$ 
with characteristic triplet $(0, 0, \nu)$, where $\nu(dz) = \frac{ \si(d \bar z)}{r^{\al+1}}$, $\al \in (0, 2)$, $r = \|z\|$, $\bar z = z / \|z\|$ 
and $\si$ is a symmetric Radon measure on $\pd B_1(0)\subseteq H$. 
The characteristic function has the special shape $\phi_{L(t)}(u) = \exp(-t c_\al \|u\|^\al)$ 
for some $c_\al >0$. 
This L\'evy measure is selfsimilar in the sense that $\nu(r A) = r^{-\al} \nu(A)$ for $r>0$ and $A\in \bB(H)$ with $0\notin \bar A$. 
In other words, it is regularly varying with index $-\al$ and has the limiting measure $\mu = \nu$. 
We take $G(x, z) := \|x\| z$.  

\noindent Then the $H$-valued mild solution of the Chafee-Infante equation with multiplicative $\al$-stable noise  
\begin{align*}
dX^\e(t) = \big(\Delta X^\e(t) - a (X^\e(t)^3- X^\e(t))\big) dt + \e \|X^\e(t)\| d L(t)  
\end{align*}
has the following first exit times and locus behavior from a 
set $D = D^\iota \subseteq \dD^\iota$ with inward pointing vector field $\Delta + f$ 
in the sense of (\ref{eq: uniformly inward pointing}). 
For any $c>0$ there are $\rR>1$ and $\e_0, \gamma \in (0, 1)$ such that 
$\e\in (0, \e_0]$ implies
\begin{align*}
&\sup_{x\in D_2^\iota (\e^\gamma, \rR)} \EE\big[\tau_x^\iota(\e, \rR)\big] \in \frac{1}{\e^\al \nu(\frac{1}{\|\phi\|} D^\mathsf{c}-\phi)} [1-c, 1+c]
\end{align*}
and for any $U \in \bB(H)$ with $\nu(\frac{1}{\|\phi\|} U-\phi)>0$ and $\nu(\frac{1}{\|\phi\|} \pd U-\phi)=0$ we have 
\begin{align*}
&\sup_{x\in D_2^\iota (\e^\gamma, \rR)} \PP(X^\e(\tau; x) \in U) \in \frac{\nu(\frac{1}{\|\phi\|} (U\cap D)^\mathsf{c}-\phi)}{\nu(\frac{1}{\|\phi\|} D^\mathsf{c}-\phi)}[1-c, 1+c], \quad \mbox{ where }\tau = \tau_x^\iota(\e, \rR).
\end{align*}
Our results also cover the exit times result of the 
additive case given in \cite{DHI13}. 
In addition, the continuous time Markov chain $M$ constructed in Corollary \ref{cor: metastability}
is switching trivially between the states $\{\phi^+, \phi^-\}$ and $(X^\e(\frac{t}{\e^\al}; x))_{t\in [0, T]} \stackrel{d}{\lra} (M_t)_{t\in [0, T]}$ 
in the sense of Corollary \ref{cor: metastability}. 

\subsubsection{The linear heat equation perturbed by additive and multiplicative stable noise}

We consider the linear heat equation on $H$ with Dirichlet conditions, that is, $f=0$, 
since the eigenvalues of the Dirichlet-Laplacian are strictly negative with upper bound $-\La_0$. 
The unit ball $D:= B_1(0)\subseteq H$ is obviously positive invariant. 
As in the previous example we treat perturbations by a symmetric 
$\alpha$-stable process $(L(t))_{t\gqq 0}$.  
We consider the multiplicative coefficient $G(x, z) = \lgl\!\lgl x- v, z\rgl\!\rgl v$ 
for some $0\neq v\in H$, with $\|v\| = 1$ fixed.  
For $\e>0$, $t\gqq 0$ and $x\in D$ the equation 
\begin{align*}
dX^\e(t) &= \Delta X^\e(t) dt + \e \lgl\!\lgl X^\e(t)-v,  d L(t)\rgl\!\rgl v \qquad \mbox{ with }\quad X^\e(0) = x, 
\end{align*}
has the following first exit times and locus behavior as $\e \ra 0$. 
We calculate the exit increments using the half space in $v$ direction, 
$\hH(v) = \{z\in ~H~|~\lgl z, v\rgl > 0\}$, as follows 
\[
\jJ^{D^\mathsf{c}}(0) = \{z\in H~|~|\lgl\!\lgl - v, z\rgl\!\rgl| >1\} 
= \big(\hH(v) + v\big) \cup \big(-\hH(v) - v\big).
\]
Theorem \ref{main result} states that 
for any $c>0$ we find $\e_0, \gamma \in (0, 1)$ such that $\e\in (0, \e_0]$ implies 
\begin{align*}
&\sup_{x\in B_{1-\e^\gamma}(0)} 
\EE\big[\tau_x(\e)\big] \in \frac{1}{\e^\al 2\nu\big(\hH(v) + v\big)} [1-c, 1+c]
\end{align*}
and Theorem \ref{main result 2} 
 yields for $U\in \bB(H)$ with $\nu(\jJ^{U\cap D^\mathsf{c}}(0))>0$ and $\nu(\jJ^{\pd U\cap D^\mathsf{c}}(0))=0$
\begin{align*}
&\sup_{x\in B_{1-\e^\gamma}(0)} \PP(X^\e(\tau; x) \in U) \in \frac{\nu\Big(U\cap \big(\big(\hH(v) + v\big) 
\cup \big(-\hH(v) - v\big)\big)\Big) }{2 \nu\big(\hH(v) + v\big)}[1-c, 1+c], 
\end{align*}
where $\tau = \tau_x(\e)$. Note that our first exit results also cover the additive case, 
which to our knowledge is also new in the literature for $G(x, z) = z$, with  
\begin{align*}
&\sup_{x\in B_{1-\e^\gamma}(0)} \EE\big[\tau_x(\e)\big] \in \frac{1}{\e^\al \nu\big(D^\mathsf{c}\big)} [1-c, 1+c],  \qquad \mbox{ and }\\
&\sup_{x\in B_{1-\e^\gamma}(0)} \PP(X^\e(\tau; x) \in U) \in \frac{\nu\big(U\cap D^\mathsf{c}\big) }{\nu\big(D^\mathsf{c}\big)}[1-c, 1+c], \mbox{ where }\tau = \tau_x(\e).
\end{align*}
\bigskip

\section{Exponentially small deviations of the small noise solution}\label{sec: small deviations}

\noindent This section is devoted to a large deviations type estimate 
for the stochastic convolution between consecutive large jumps. 
It quantifies the fact that in the time interval strictly 
between two adjacent large jumps the solution of (\ref{main sys}) 
is perturbed by only the
small noise component and deviates from the solution of the deterministic
equation by only a small $\e$-dependent quantity, with a probability
converging to $1$ exponentially fast as a function of the inverse noise intensity, $1/\e$, as $\e\to 0+$. 

\paragraph{Preliminaries and notation: }In this section we fix the 
domain of attraction $\dD = \dD^\iota$ of $\phi = \phi^\iota$ and  
the invariant subset $D = D^\iota$ with $\pd D^\iota \in \cC^1$ 
such that $f$ is inward pointing on $\pd D$.  
We drop all further dependencies on the index $\iota$.   
For a better understanding of the role of the different scales 
we formulate and prove our results for abstract scale functions 
\begin{equation}\label{eq: scales}
\begin{split}
&\rho^\cdot: (0, 1] \ra [1, \infty), \qquad \lim_{\e\ra 0+} \rho^\e = \infty, \qquad \lim_{\e \ra 0+} \e \rho^\e = 0, \\
&\gamma_\cdot: (0, 1] \ra (0, 1], \qquad ~\lim_{\e \ra 0+} \gamma_\e = 0, \\
&T^\cdot: (0,1] \ra [1, \infty), ~~\quad \lim_{\e \ra 0+} T^\e = \infty,
\end{split}
\end{equation}
before choosing them numerically in (\ref{eq: choice of scales})-(\ref{eq: large noise scales}) 
below the statement of Proposition \ref{cor: T1 Event E^c}. 
We fix the following notation for abstract scales 
\begin{equation}\label{def: large jumps 2}
T_0 :=0, \qquad T_{k} := \inf\left\{t>T_{k-1}~\big|~\|\Delta_t L\| > \rho^\e \right\}
\qquad \mbox{ and }\qquad W_k := \Delta_{T_k} L, \qquad k\gqq 1.
\end{equation}
We define the compound Poisson process $\eta^\e$ which consists only of 
large jumps of $L$ with intensity $\beta_\e := \nu\left( B_{\rho^\e}^\mathsf{c}(0)\right)$
and the jump probability distribution by 
$\PP(W_k \in U) = \ds\nu(U \cap B_{\rho^\e}^\mathsf{c}(0))/\beta_\e$. Then  
\begin{align*}
\eta^\e_t := \int_0^t \int_{\|z\|> \rho^\e} z N(dsdz) = \sum_{k=1}^\infty W_k \ind_{\{T_k < t\}}, \quad t\gqq 0. 
\end{align*}
The complementary small jumps process $\xi^\e_t := L_t - \eta^\e_t$ has the following shape 
\begin{align*}
\xi^\e_t &= \int_0^t \int_{\|z\|\lqq 1} z \ti N(dsdz) + \int_0^t \int_{1 < \|z\|\lqq \rho^\e} z N(dsdz) 
=\int_0^t \int_{\|z\|\lqq \rho^\e} z \ti N(dsdz) + \int_0^t \int_{1 < \|z\|\lqq \rho^\e} z \nu(dz) ds.
\end{align*}
The process $\xi^\e$ has for any $\e \in (0, 1]$ exponential moments of any order 
due to the uniformly bounded jump size and $\xi^\e - t \int_{1< \|z\| \lqq \rho^\e} z \nu(dz)$ 
is a mean zero $(\fF_t)_{t\gqq 0}$-martingale in $H$. 

Denote by $(S(t))_{t\gqq 0}$ the semigroup $(e^{t \Delta})_{t\gqq 0}$ on $H$. 
It is a generalized contraction $\cC_0$-semigroup satisfying several regularization properties. 
We refer for our special setup to \cite{DHI13} p.13-14 and for general cases to \cite{Pa83}. 

The i.i.d. family of $\DEXP(\beta_\e)$-distributed waiting times 
between successive large jumps of $\eta^\e$ is given by 
$t_0 = 0$ and \mbox{$t_k := T_k-T_{k-1}$} for $k\gqq 1.$ 
We denote the process $L$ between the waiting times by 
$\xi^{\e, k}(t):= L_{t+T_{k-1}} - L_{T_{k-1}}$ for $t\in [0, t_k)$. 
Then the i.i.d. families $(t_k)_{k\in \NN}$, $(W_k)_{k\in \NN}$, $(\xi^{\e, k}(t))_{t\in [0, t_k), k \in \NN}$ are independent. 
We call $Y^\e(t, \zeta;x)$ for all $t\gqq 0$, $\zeta \in J$ and $x\in H$ the mild solution of 
\begin{equation}\label{eq: small noise equation}
\begin{split}
d Y^\e (t, \zeta) &=  \Big(\Delta Y^\e(t, \zeta) + f(Y^\e(t, \zeta))\Big) dt + G\big(Y^\e(t-, \zeta), \e d\xi^\e(t)\big), \\
Y^\e(0, \zeta; x) &= x(\zeta), \\[1mm]
Y^\e(t, 0; x) &= Y^\e(t, 1; x) = 0.
\end{split}
\end{equation}

In the following two subsections we derive all results 
on the stochastic convolution w.r.t. to $\xi^\e$ 
up to the hitting time of $Y^\e$ leaving a large ball. 
We shall get rid of that artificial time horizon in the proof 
of Proposition \ref{prop: noise control implies remainder control}
by showing that the process $Y^\e$ at time $\si$ is strictly inside the large ball 
whenever the noise convolution is small. 
For $\rR\gqq\rR_0$, $\e \in (0, 1]$ and $x\in D_2(\gamma_\e, \rR)$ we define the $(\fF_t)_{t\gqq 0}$-stopping time 
\begin{equation}\label{def: exit time Y}
\si^1 := \si^1_{\rR, x}(\e) := \inf\{t>0~|~Y^\e(t;x) \notin \uU^{\rR}\}. 
\end{equation}

\bigskip
\subsection{Exponential estimate of stochastic convolutions with bounded jumps} 

In this subsection we show that the stochastic convolution 
with respect to $G(Y^\e(s-), \e d\xi^\e(s))$ for small $\e>0$ 
is very small, i.e. of order $\lqq \gamma_\e^q$ for some $q\gqq 1$, 
with a probability which tends to $1$ in terms of $\gamma_\e$ exponentially 
fast as long as $Y^\e$ and the stochastic convolution remain bounded by $\rR$. 
For the solution $Y^\e(t;x)$ of (\ref{eq: small noise equation}) with $\rR\gqq \rR_0$, $\e \in(0, 1]$ and 
$x \in D_2(\gamma_\e, \rR)$ we consider the multiplicative stochastic convolution process 
\begin{align*}
&\Psi^{\e, x}_{t} := \int_0^t S(t-s) G(Y^\e(s-;x), \e d\xi^\e(s))\\
&= \int_0^t \int_{0< \|z\| \lqq \rho^\e} S(t-s) G(Y^\e(s-;x),\e z) \ti N(dsdz) + \int_0^t \int_{1< \|z\| \lqq \rho^\e} S(t-s) G(Y^\e(s; x),\e z) \nu(dz) ds\\[2mm]
&=: \Phi^{\e, x}_{t} + b^{\e, x}_{t}. 
\end{align*}
The process $(\Psi^{\e, x}_t)_{t\gqq 0}$ is an $(\fF_t)_{t\gqq 0}$-adapted c\`adl\`ag process. 
For $\rR\gqq\rR_0$, $\e \in (0, 1]$ and $x\in D_2(\gamma_\e, \rR)$ we set 
\begin{align}\label{def: exit time Psi}
\si^2 := \si^2_{\rR, x}(\e) := \inf\{t>0~|~\Psi^{\e, x}_{t} \notin \uU^{\rR}\} \qquad \mbox{ and } \qquad \si := \si^1 \wedge \si^2. \\
\nonumber
\end{align}

\begin{prop}\label{prop: convolution estimate} 
Let the Hypotheses (D.1)-(D.3), (S.1) and (S.2) be satisfied and 
the functions $\rho^\cdot, \gamma_\cdot$ be given by (\ref{eq: scales}) 
and fulfill for some $q\gqq 1$ the limit relation 
  \begin{equation}
 \lim_{\e\ra 0} \Gamma(\e) = 0 \qquad \quad \mbox{ where } \qquad 
 \Gamma(\e) := \frac{(\e \rho^\e)^2 T^\e}{\gamma_\e^{4q +6}}. \label{eq: parameter limit conv} 
 \end{equation}
 Then for any $\rR\gqq \rR_0$ there is $\e_0\in (0, 1]$ such that $\e \in (0, \e_0]$ implies 
 \begin{equation}\label{eq: exponential martingale estimate} 
 \sup_{x\in D_2(\gamma_\e, \rR)} \PP\big(\sup_{s \in [0, \si\wedge T^\e]} \|\Psi^{\e, x}_{s} \| > \gamma_\e^q\big) \lqq 2 \exp(-(2 \gamma_\e)^{-1}).  
 \end{equation}
 \end{prop}

\begin{proof} 
The plan of the proof is as follows. First we get rid of the drift $b^{\e, x}$ (Step 0). 
In order to control $\Phi^{\e,x}$ we start with the exponential Kolmogorov inequality 
where we introduce the free parameters $\la$ and $\chi$. We estimate the stochastic convolution 
with the help of a result of Salavati and Zangeneh \cite{SZ16} 
and derive an exponential version of the Burkholder-Davis-Gundy inequality 
using the respective pathwise result by Siorpaes \cite{Siorpaes15} (Step 1).  
Then we optimize over the free parameters 
and use the Campbell representation of the Laplace transform of the quadratic 
variation of Poisson random integrals and a Campbell type estimate shown in Lemma \ref{lem: Campbell}. 
This allows for a comparison principle for the characteristic exponent of $\Psi^{\e, x}$ (Step 2) and allows to conclude (Step 3). 

\paragraph{Step 0: Drift estimate.} We show that for any $\rR\gqq \rR_0$ 
there is $\e_0 \in (0, 1]$ such that $\e\in (0, \e_0]$ 
implies \[
\sup_{t\in [0, \si\wedge T^\e]} \sup_{x \in D_2(\gamma_\e, \rR)} \|b^{\e, x}_{t}\| < \frac{1}{2} \gamma_\e^q.
\] 

\medskip
\noindent We write $Y(t) = Y^{\e}(t;x) $ and recall Hypothesis (S.1).  
For $g_1 := \sup_{x \in \uU^\rR} G_1(x)$ 
the triangular inequality and the norm estimate of the heat semigroup $S$ yield for $x\in D_2(\gamma_\e, \rR)$ 
\begin{align*}
\sup_{t\in [0, \si]} \|b^{\e, x}_{t}\|
&\lqq \sup_{t\in [0, \si]} \|\int_0^{t} \int_{1< \|z\|\lqq \rho^\e} S(t -s)G(Y(s-), \e z) \nu(dz) ds\| \\
&\lqq \sup_{t\in [0, \si]} \int_0^{t} \int_{1< \|z\|\lqq \rho^\e}  \|S(t -s)G(Y(s-), \e z) \| \nu(dz) ds \\
&\lqq g_1 \sup_{t\in [0, \si]} \int_0^{t}  \int_{1< \|z\|\lqq \rho^\e} e^{-\La_0(t -s)} \|\e z\| \nu(dz) ds \\
&\lqq \e g_1 \sup_{t\in [0, \infty)} \int_0^{t}  e^{-\La_0(t -s)} ds \int_{1< \|z\|\lqq \rho^\e}  \|z\| \nu(dz) 
\lqq \big(g_1 \frac{\nu(B_1^\mathsf{c}(0))}{\La_0}\big) ~\e\rho^\e \lqq C_0 \e\rho^\e. 
\end{align*}
Note that the right side is independent $\si$ and $x$. 
The limit (\ref{eq: parameter limit conv}) yields 
the existence of a constant $\e_0 \in (0, 1]$ such that for $\e\in (0, \e_0]$ we have   
$\e \rho^{\e} \lqq \gamma_{\e}^{q} /2 C_0$ 
and hence satisfies the claim. 

\medskip
\paragraph{Step 1: Exponential estimate of the stochastic convolution: } 
Note that for $\e_0$ of Step 0 and $\e \in (0, \e_0]$ we get  
\begin{align*}
\PP(\sup_{t\in [0,\si\wedge T^\e]} \|\Psi^{\e, x}_t\| > \gamma_\e^q) 
&\lqq \PP(\sup_{t\in [0,\si\wedge T^\e]} \|\Phi^{\e, x}_t\| > \frac{\gamma_\e^q}{2}) 
+ \PP(\sup_{ t\in [0,\si\wedge T^\e]}\|b^{\e, x}_t\| > \frac{\gamma_\e^q}{2}) \\
&= \PP(\sup_{t\in [0,\si\wedge T^\e]} \|\Phi^{\e, x}_t\| > \frac{\gamma_\e^q}{2}).
\end{align*}
Kolmogorov's exponential inequality yields with a free parameter $\vartheta>0$ reads 
\begin{align}
\PP(\sup_{t\in [0,\si\wedge T^\e]} \|\Phi^{\e, x}_t\| > \frac{\gamma_\e^q}{2} )
\lqq \exp(-\vartheta \frac{\gamma_\e^{q}}{2} ) \EE\Big[\exp(\vartheta \sup_{t\in [0,\si\wedge T^\e]} \|\Phi^{\e, x}_t\|) \Big].\label{eq: Kolmogorov}
\end{align}
For $M^{(1)}_t := \int_0^{t} \int_{0< \|z\| \lqq \rho^\e} G(Y(s-;x),\e z) \ti N(dsdz)$ the process $(M^{(1)}_{t\wedge \si})_{t\gqq 0}$
is an $(\fF_t)_{t\gqq 0}$-martingale. The pathwise estimate of the stochastic convolution in Salavati and Zangeneh \cite{SZ16} (Theorem 6) 
yields for any $t\gqq 0$ the following $\PP$-a.s. inequality ($\Phi^{\e, x}_0 = 0$)  
\begin{align*}
&\|\Phi^{\e, x}_t\|^{2} \lqq 2 \int_0^{t} e^{-2 \La_0 (t-s)} \lgl\!\lgl \Phi^{\e, x}_{s-}, dM^{(1)}_s\rgl\!\rgl 
+ \sum_{0< s\lqq t}e^{-2 \La_0 (t-s)} 
\big(\|\Phi^{\e, x}_{s}\|^2-\|\Phi^{\e, x}_{s-}\|^2 - 2 \lgl\!\lgl \Phi^{\e, x}_{s-}, \Delta_s M^{(1)}\rgl\!\rgl\big)\\
&= e^{-2 \La_0 t} \bigg(2  \int_0^{t}\int_{\|z\|\lqq \rho^\e} e^{2 \La_0 s} \lgl\!\lgl \Phi^{\e, x}_{s-}, G(Y(s-), z) \rgl\!\rgl \ti N(dsdz)
+  \int_0^t \int_{\|z\| \lqq \rho^\e} e^{2 \La_0 s} \|G(Y(s-), z)\|^2 N(dsdz) \bigg)\\[2mm]
&\lqq e^{-2 \La_0 t}  \big(|M_t^{(2)}|+ M_t^{(3)}\big), 
\end{align*}
where 
\begin{align*}
M_t^{(2)} &:=  2  \int_0^{t}\int_{\|z\|\lqq \rho^\e} e^{2 \La_0 s} \lgl\!\lgl \Phi^{\e, x}_{s-}, G(Y(s-), z) \rgl\!\rgl \ti N(dsdz),\\
M_t^{(3)} &:= \int_0^t \int_{\|z\| \lqq \rho^\e} e^{2 \La_0 s} \|G(Y(s-), z)\|^2 N(dsdz).
\end{align*}
Therefore 
\begin{align}\label{eq: pivot}
\sup_{t\in [0, \si\wedge T^\e]} \|\Phi^{\e, x}_t\|^{2} 
&\lqq \sup_{t\in [0, \si\wedge T^\e]} e^{-2 \La_0 t} |M^{(2)}_t|+ \sup_{t\in [0, \si\wedge T^\e]} e^{-2 \La_0 t} M^{(3)}_t.
\end{align}
\paragraph{Step 1a: } We start with the first term in (\ref{eq: pivot}). It\={o}'s formula 
applied to $e^{-2 \La_0 t} M^{(2)}_t$ gives 
\begin{align}
e^{-2 \La_0 t} M^{(2)}_t 
&= \int_0^t e^{- 2 \La_0 s} dM^{(2)}_s - 2 \La_0 \int_0^t \big(e^{-2 \La_0 s} M^{(2)}_s\big) ds.\label{eq: recursion}
\end{align}
The preceding identity (\ref{eq: recursion}) defines the following recursion. 
We replace the expression $ e^{-2 \La_0 s} M^{(2)}_s$ under the integral 
by the respective right-hand side 
\begin{align*}
&e^{-2 \La_0 t} M^{(2)}_t 
= \int_0^t e^{- 2 \La_0 s_1} dM^{(2)}_{s_1} - 2 \La_0 \int_0^t \big(e^{-2 \La_0 s_1} M^{(2)}_{s_1}\big) ds_1\\
&= \int_0^t e^{- 2 \La_0 s_1} dM^{(2)}_{s_1} + ( -2 \La_0) \int_0^t \Big(\int_0^{s_1} e^{- 2 \La_0 s_2} dM^{(2)}_{s_2}- 2 \La_0 \int_0^{s_1} e^{-2 \La_0 s_2} M^{(2)}_{s_2} ds_2\Big) ds_1\\
&= \int_0^t e^{- 2 \La_0 s_1} dM^{(2)}_{s_1} +  (- 2 \La_0) \int_0^t \int_0^{s_1} e^{- 2 \La_0 s_2} dM^{(2)}_{s_2} ds_1 
+ (- 2 \La_0)^2 \int_0^t \int_0^{s_1} e^{-2 \La_0 s_2} M^{(2)}_{s_2} ds_2 ds_1. 
\end{align*}
Repeating this procedure $k+1$ times 
an easy induction shows that for all $k\in \NN$ 
\begin{align*}
e^{-2 \La_0 t} M^{(2)}_t 
&= \sum_{\ell=0}^{k} \int_0^t \int_0^{s_1} \dots \int_0^{s_{\ell-1}} \Big(\int_0^{s_{\ell}} (- 2 \La_0)^{\ell} e^{-2 \La_0 s_\ell} d M^{(2)}_{s_{\ell+1}}\Big) ds_{\ell} \dots ds_1 \\ 
&\qquad + (- 2 \La_0)^{k+1} \int_0^t \int_0^{s_1} \dots \int_0^{s_{k}} e^{-2 \La_0 s_{k+1}} M^{(2)}_{s_{k+1}} ds_{k+1}  \dots ds_1.
\end{align*}
Cauchy's formula for repeated integrals applied to the second term provides  
\begin{align*}
\Big|(- 2 \La_0)^{k+1} \int_0^t \int_0^{s_1} \dots \int_0^{s_{k}} e^{-2 \La_0 s_{k+1}} M^{(2)}_{s_{k+1}} ds_{k+1} \dots ds_1\Big| 
&= \Big|(- 2 \La_0)\int_0^t  \frac{\big((-2\La_0)(t-s)\big)^{k}}{k!} e^{-2\La s} M^{(2)}_s ds\Big| \\
&\lqq  2\La_0 \int_0^t \frac{\big|2\La_0(t-s)\big|^{k}}{k!}  ds \sup_{s\in [0, t]} e^{-2\La s} |M^{(2)}_s|.
\end{align*}
Since $\sup_{s\in [0, t]} e^{-2\La s} |M^{(2)}_s|< \infty$ $\PP$-a.s. we may pass to the limit as $k\ra \infty$ and obtain 
with the help of the monotone convergence theorem that 
\begin{align}
&\limsup_{k \ra \infty} 
\Big|(- 2 \La_0)^{k+1} \int_0^t \int_0^{s_1} \dots \int_0^{s_{k}} e^{-2 \La_0 s_k} M^{(2)}_{s_{k+1}} ds_{k+1} \dots ds_1\Big| \nonumber\\
&\lqq 2\La_0 \limsup_{k \ra \infty} \int_0^t \frac{\big(2\La_0(t-s)\big)^{k-1}}{(k-1)!}  ds 
\sup_{s\in [0, t]} e^{-2\La s} |M^{(2)}_s| = 0, \quad \PP-\mbox{a.s.}\label{eq: remainder term vanishes}
\end{align}
Hence we have proved the represenation 
\begin{align*}
e^{-2 \La_0 t} M^{(2)}_t 
&=\sum_{\ell=1}^\infty  
\int_0^t \int_0^{s_1} \dots \int_0^{s_{\ell-1}} \int_0^{s_{\ell}} (- 2 \La_0)^{\ell} e^{-2 \La_0 s_{\ell+1}} d M^{(2)}_{s_{\ell+1}} ds_{\ell}  \dots ds_1, \qquad \PP-\mbox{a.s.}
\end{align*}
For each of the summands we apply Cauchy's formula for repeated integrals    
\begin{align*}
&\int_0^t \int_0^{s_1} \dots \int_0^{s_{\ell-1}} \int_0^{s_{\ell}} (- 2 \La_0)^{\ell} e^{-2 \La_0 s_{\ell+1}} d M^{(2)}_{s_{\ell+1}} ds_{\ell}  \dots ds_1\\
&=  (-2\La_0)\int_0^t \frac{(-2\La_0 (t-s))^{\ell-1}}{(\ell-1)!} \int_0^{s} e^{-2\La_0 r} dM^{(2)}_r ds.
\end{align*}
Fubini's theorem yields $\PP$-a.s. 
\begin{align*}
\sum_{\ell=1}^\infty \int_0^t \frac{(-2\La_0 (t-s))^{\ell-1}}{(\ell-1)!} \int_0^{s} e^{-2\La_0 r} dM^{(2)}_r ds
&=\int_0^t \sum_{\ell=1}^\infty \frac{(-2\La_0 (t-s))^{\ell-1}}{(\ell-1)!} \int_0^{s} e^{-2\La_0 r} dM^{(2)}_r ds\\
&=\int_0^t e^{-2\La_0 (t-s)} \int_0^{s} e^{-2\La_0 r} dM^{(2)}_r ds.
\end{align*}
We estimate 
\begin{align}
e^{-2 \La_0 t} |M^{(2)}_t| 
&\lqq \big|(-2\La_0) \int_0^t  e^{-2\La_0 (t-s)} \int_0^{s} e^{-2\La_0 r} dM^{(2)}_r ds\big| \nonumber\\
&\lqq \big|(-2\La_0) \int_0^t  e^{-2\La_0 (t-s)} ds\big|
\sup_{s\in [0, t]} |M^{(4)}_s| 
\lqq \sup_{s\in [0, t]} |M^{(4)}_s|,\label{eq: exp M2 leq M4 pointwise}
\end{align}
where
\begin{align*}
\int_0^{s} e^{-2\La_0 r} dM^{(2)}_r 
&= 2  \int_0^{t}\int_{\|z\|\lqq \rho^\e} \lgl\!\lgl \Phi^{\e, x}_{s-}, G(Y(s-), z) \rgl\!\rgl \ti N(dsdz) =: M^{(4)}_t.
\end{align*}
Since the right-hand side of (\ref{eq: exp M2 leq M4 pointwise}) is nondecreasing 
the first term in (\ref{eq: pivot}) has the upper bound 
\begin{align}\label{eq: exp M2 leq M4}
\sup_{t\in [0, \si\wedge T^\e]} e^{-2 \La_0 t} M^{(2)}_t \lqq \sup_{t\in [0, \si\wedge T^\e]}|M^{(4)}_t|. 
\end{align}
We continue with the pathwise Burkholder-Davis-Gundy inequality by Siorpaes \cite{Siorpaes15} 
\begin{align}\label{eq: Siorpaes}
\sup_{t\in [0, \si\wedge T^\e]} |M^{(4)}_t| 
&\lqq 6 \sqrt{[M^{(4)}]_{\si\wedge T^\e}} + 2 \int_0^{\si\wedge T^\e} H_{s-} dM_s^{(4)},
\end{align}
where
\begin{equation}\label{eq: H}
H_s := M_{s}^{(4)} / \sqrt{[M^{(4)}]_{s} + \sup_{r\in [0, s]} |M_r^{(4)}|^2~}.
\end{equation}
Note that $\sup_{s\gqq 0} |H_s|\lqq 1$ $\PP$-a.s. by construction. 

\paragraph{Step 1b: } The second term in (\ref{eq: pivot}) is easier since the nonnegative integrands 
allow for a $\PP$-a.s. monotonicity estimate 
\begin{align}
\sup_{t\in [0, \si\wedge T^\e]} e^{-2\La_0 t} M_t^{(3)} 
&\lqq \sup_{t\in [0, \si\wedge T^\e]} \int_0^t \int_{\|z\| \lqq \rho^\e} \|G(Y(s-), z)\|^2 N(dsdz)\nonumber\\
&= \int_0^{\si\wedge T^\e} \int_{\|z\| \lqq \rho^\e} \|G(Y(s-), z)\|^2 N(dsdz) = : M^{(5)}_{\si\wedge T^\e}.\label{eq: M3}
\end{align}
\paragraph{Step 1c: } We combine (\ref{eq: pivot}), (\ref{eq: exp M2 leq M4}), (\ref{eq: Siorpaes}) and (\ref{eq: M3}). 
The subadditivity of the square root and the estimate $\sqrt{r} \lqq r + \frac{1}{4}$ for 
$r\gqq 0$ yield
\begin{align}
&\EE\Big[\exp(\vartheta \sup_{t\in [0,\si\wedge T^\e]} \|\Phi^{\e, x}_t\|) \Big] \nonumber\\
&\lqq \EE\Big[\exp\Big(\vartheta \sqrt{6 \sqrt{[M^{(4)}]_{\si \wedge T^\e}} 
+ 2 \int_0^{\si \wedge T^\e} H_{s-} dM^{(4)}_s + M^{(5)}_{\si \wedge T^\e}}\Big)\Big]\nonumber\\
&\lqq e^\frac{1}{4} \EE\Big[\exp\Big(\vartheta^2\, \big(6 \sqrt{[M^{(4)}]_{\si \wedge T^\e}} 
+ 2 \int_0^{\si \wedge T^\e} H_{s-} dM^{(4)}_s+ M^{(5)}_{\si \wedge T^\e}\big)\Big)\Big].\label{eq: 1c}
\end{align}
Young's inequality with the additional free parameter $\chi>0$ reads 
\begin{align*}
\sqrt{[M^{(4)}]_{\si \wedge T^\e}}
&\lqq \frac{1}{2\chi^2} [M^{(4)}]_{\si \wedge T^\e} + \frac{ \chi^2}{2}. 
\end{align*}
Then the estimate $abcd \lqq (a^4 + b^4 + c^4 + d^4) /4$, $a, b, c, d>0$, provides 
for (\ref{eq: 1c}) the upper bound  
\begin{align}
&e^\frac{1}{4} \EE\Big[\exp\Big(\frac{3\vartheta^2}{\chi^2} [M^{(4)}]_{\si \wedge T^\e}
+ \frac{3 \vartheta^2 \chi^2}{2} + 2 \vartheta^2 \Big(\int_0^{\si \wedge T^\e} H_{s-} dM^{(4)}_s\Big) 
+\vartheta^2  M^{(5)}_{\si \wedge T^\e}\Big)\Big]\nonumber\\
&\lqq \frac{e^\frac{1}{4}}{4} \EE\Big[\exp\Big(\frac{12\vartheta^2}{\chi^2} [M^{(4)}]_{\si \wedge T^\e}\Big)\Big]
+ \frac{e^\frac{1}{4}}{4} \exp\Big(6  \vartheta^2 \chi^2\Big) \nonumber\\
&\qquad + \frac{e^\frac{1}{4}}{4} \EE\Big[\exp\Big( 8\vartheta^2 \int_0^{\si \wedge T^\e} H_{s-} dM^{(4)}_s\Big)\Big]
+ \frac{e^\frac{1}{4}}{4} \EE\Big[\exp\Big(4\vartheta^2  M^{(5)}_{\si\wedge m}\Big)\Big]\nonumber\\
&=: J_1(\vartheta, \chi) + J_2(\vartheta, \chi) + J_3(\vartheta) + J_4(\vartheta). \label{eq: upper bound}
\end{align}

\paragraph{Step 2: Campbell's formula and Optimization over the free parameters. }
We now choose the free parameters $\vartheta$ and $\chi$ as $\e$-dependent functions 
$\vartheta_\e := \frac{1}{\gamma_\e^{q+1}}$ and $\chi_\e = \gamma_\e^{q+2}$ 
such that the upper bound (\ref{eq: upper bound}) of 
the right side of (\ref{eq: Kolmogorov}) reads 
\[
\exp(-\vartheta_\e \frac{1}{2} \gamma_\e^{q}) \Big(J_1(\vartheta_\e, \chi_\e) + J_2(\vartheta_\e, \chi_\e) + J_3(\vartheta_\e) + J_4(\vartheta_\e)\Big).
\]
In the sequel we estimate the respective terms $J_1, \dots, J_4$ one by one. 

\paragraph{$\mathbf{J_1(\vartheta_\e, \chi_\e)}$: } By $\PP$-a.s. monotonicity we have 
\begin{align*}
[M^{(4)}]_{\si \wedge T^\e} 
&= \Big[4  \int_0^{\cdot}\int_{\|z\|\lqq \rho^\e} 
\lgl\!\lgl \Phi^{\e, x}_{s-}, G(Y(s-), \e z) \rgl\!\rgl \ti N(dsdz)\Big]_{\si \wedge T^\e}\\
&= 16 \int_0^{\si \wedge T^\e} \int_{\|z\|\lqq \rho^\e} 
\lgl\!\lgl \Phi^{\e, x}_{s-}, G(Y(s-), \e z) \rgl\!\rgl^2 N(dsdz)\\
&\lqq C_1 \int_0^{\si \wedge T^\e} \int_{\|z\|\lqq \rho^\e} \|\Phi^{\e, x}_{s-}\|^2 \|\e z\|^2 N(dsdz) \\
&\lqq C_2\int_0^{\si \wedge T^\e} \int_{\|z\|\lqq \rho^\e} \|\e z\|^2 N(dsdz) 
\lqq C_2\int_0^{T^\e} \int_{\|z\|\lqq \rho^\e} \|\e z\|^2 N(dsdz).
\end{align*}
The classical Campbell formula for the Poisson random measure $N$ has the shape 
\begin{align}
\EE\Big[\exp\Big(12 \big(\frac{\vartheta_\e}{\chi_\e}\big)^2  [M^{(4)}]_{\si \wedge T^\e}\Big)\Big] 
&\lqq \EE\Big[\exp\Big(\frac{C_3}{\gamma_\e^{4q+6}}  \int_0^{T^\e} 
\int_{\|z\|\lqq \rho^\e} \|\e z\|^2 N(dsdz)\Big)\Big] \nonumber\\
&=\EE\Big[\exp\Big( T^\e \int_{\|z\|\lqq \rho^\e} 
\big(\exp\big(\frac{C_3 \|\e z\|^2}{\gamma_\e^{4q+6}}\big) -1\big)  \nu(dz)\Big)\Big].\label{eq: Campell1}
\end{align}
The limit (\ref{eq: parameter limit conv}) implies for the exponent 
\[
\sup_{\|z\|\lqq \rho^\e} \frac{\|\e z\|^2}{\gamma_\e^{4q+6}}  
\lqq \frac{(\e \rho^\e)^2}{\gamma_\e^{4q+6}}\lra 0, \qquad \mbox{ as }\e \ra 0.
\]
In the sequel we 
use $(e^r -1) \lqq (e-1)r$ for all $r\in [0, 1]$ under the integral in the exponent of equation (\ref{eq: Campell1})
and choose $\e_0 \in (0, 1]$ small enough such that $\e \in (0, \e_0]$ implies $\frac{C_3 (\e \rho^\e)^2}{\gamma_\e^{4q+3}}\lqq 1$. 
If, in addition, $\rho^\e\gqq 1$ for any $\e \in (0, \e_0]$ 
we obtain the following upper bound of (\ref{eq: Campell1}) 
\begin{align*}
&\EE\Big[\exp\Big(C_3 (e-1) \frac{\e^2 T^\e}{\gamma_\e^{4q+6}} 
\int_{\|z\|\lqq \rho^\e} \|z\|^2 \nu(dz) \Big)\Big]\\
&\lqq \EE\Big[\exp\Big(C_3 (e-1) \frac{\e^2 T^\e}{\gamma_\e^{4q+6}} 
\big(\int_{\|z\|\lqq 1} \|z\|^2 \nu(dz) + (\rho^\e)^2 \nu(B_1^\mathsf{c}(0))\big)\Big)\Big] \lqq \exp\Big( C_4 \frac{(\e \rho^\e)^2 T^\e}{\gamma_\e^{4q+6}} \Big), 
\end{align*}
where $C_4 = C_3 (e-1) \big(\int_{\|z\|\lqq 1} \|z\|^2 \nu(dz) + \nu(B_1^\mathsf{c}(0))\big)$. 
Thus $\e\in (0, \e_0]$ yields 
\begin{align}
\exp(- \frac{\vartheta_\e}{2} \gamma_\e^{q}) J_1(\vartheta_\e, \chi_\e)
&= \exp(-\frac{1}{2\gamma_\e}) J_1(\vartheta_\e, \chi_\e) 
\lqq \frac{e^\frac{1}{4}}{4} \exp(-\frac{1}{2\gamma_\e}) \exp(C_4\Gamma(\e))\lqq \frac{1}{2} 
\exp(-\frac{1}{2\gamma_\e}).\label{eq: J1 fertig}
\end{align}

\paragraph{$\mathbf{J_2(\vartheta_\e, \chi_\e)}$: } There is $\e_0 \in (0, 1]$ such that for 
$\e \in (0, \e_0]$ we have the estimate 
\begin{align}
\exp(- \frac{\vartheta_\e}{2} \gamma_\e^{q}) J_2(\vartheta_\e, \chi_\e)
&\lqq \frac{e^\frac{1}{4}}{4} \exp(-\frac{1}{2\gamma_\e} + 6 (\vartheta_\e  \chi_\e)^2) 
= \frac{e^\frac{1}{4}}{4}\exp(-\frac{1}{2\gamma_\e} + 6 \gamma_\e^2)
\lqq \frac{1}{2} \exp(-\frac{1}{2\gamma_\e}).\label{eq: J2 fertig}
\end{align}

\paragraph{$\mathbf{J_3(\vartheta_\e)}$: } Recall that 
\begin{align*}
M^{(4)}_t = 2  \int_0^{t}\int_{\|z\|\lqq \rho^\e} \lgl\!\lgl \Phi^{\e, x}_{s-}, G(Y(s-), \e z) \rgl\!\rgl \ti N(dsdz) 
\end{align*}
such that for the function $h_y(s-, \e z):= 2 H_{s-} \lgl\!\lgl \Phi^{\e, x}_{s-}, G(Y(s-;x), \e z) \rgl\!\rgl$ we have the representation 
\begin{align}
Z_{t}^{\e, x} := \int_0^{t} \int_{\|z\|\lqq \rho^\e} h_x(s-, \e z) \ti N(dsdz) \Big(= \int_0^{t} H_{s-} dM^{(4)}_s\Big).\label{def: Z}
\end{align}
Lemma \ref{lem: Campbell} which is proved in Appendix \ref{app: Campbell} yields 
$\e_0\in (0, 1]$ such that $\e \in (0, \e_0]$ implies 
\[
\sup_{x\in D_2(\gamma_\e, \rR)} \EE\Big[\exp\Big(8\vartheta_\e^2\, \big|Z_{\si\wedge T^\e}^{\e,x}\big|\Big)\Big] \lqq 2.  
\]
This implies for $\e \in (0, \e_0]$ the estimate 
\begin{equation}
\exp(-\vartheta_\e \frac{1}{4} \gamma_\e^{2q}) J_3(\vartheta_\e) 
\lqq \frac{2}{4 e^{\frac{1}{4}}} \exp(-\frac{1}{2\gamma_\e}) 
\lqq \frac{1}{2} \exp(-\frac{1}{2\gamma_\e}) . \label{eq: J3 fertig}
\end{equation}

\paragraph{$\mathbf{J_4(\vartheta_\e)}$: } This case resembles the one of $J_1(\vartheta_\e, \chi_\e)$. 
Since we have only positive jumps we estimate $\PP$-a.s.  
\begin{align*}
M_{\si \wedge T^\e}^{(3)} 
&= \int_0^{\si \wedge T^\e} \int_{\|z\| \lqq \rho^\e}  \|G(Y(s-), \e z)\|^2 N(dsdz) 
\lqq \int_0^{\si \wedge T^\e} \int_{\|z\| \lqq \rho^\e} g_1^2 \|\e z\|^2 N(dsdz),
\end{align*}
leading to 
\begin{align*}
\EE\Big[\exp\Big(2 \vartheta_\e^2\,  M_{\si \wedge T^\e}^{(3)}\Big)\Big] 
\lqq \EE\Big[\exp\Big(T^\e \int_{\|z\|\lqq \rho^\e} \big( \exp(2 g_1 \vartheta_\e^2\,  \|\e z\|^2)-1\big)\nu(dz)\Big)\Big].
\end{align*}
Analogously to $J_1(\vartheta_\e, \chi_\e)$ we obtain with the help of the limit 
(\ref{eq: parameter limit conv}) that for $\e \ra 0+$  
\begin{align*}
\sup_{\|z\| \lqq \rho^\e} 2 g_1 \vartheta_\e^2\, \|\e z\|^2 
\lqq 2 g_1 \vartheta_\e^2\, (\e \rho^\e)^2 \lqq 2 g_1 \frac{(\e \rho^\e)^2}{\gamma_\e^{2q+2}} \ra 0.  
\end{align*}
The additional restriction of $\e_0\in (0, 1]$ such that $\e\in (0, \e_0]$ simultaneously implies 
$2 g_1 \frac{(\e \rho^\e)^2}{\gamma_\e^{2q+2}}\lqq 1$ and 
$\rho^\e \gqq \int_{\|z\| \lqq 1} \|z\|^2 \nu(dz) / \nu(B_1^\mathsf{c}(0))$
yields the estimate 
\begin{align*}
J_4(\vartheta_\e)
&\lqq \frac{e^\frac{1}{4}}{4} \EE\Big[\exp\Big(\int_0^{\si \wedge T^\e} \int_{\|z\|\lqq \rho^\e} 
\Big( \exp\big(2 \vartheta_\e^2\, g_1 \|\e z\|^2\big)-1\Big)\nu(dz) ds\Big)\Big]\\ 
&\lqq \frac{e^{\frac{1}{4}}}{4} \EE\Big[\exp\Big((e-1) 2 g_1 C_6 T^\e \vartheta_\e^2\, (\e \rho^\e)^2 \Big)\Big]\\ 
&\lqq \frac{1}{2} \exp\Big(C_7 \frac{(\e \rho^\e)^2 T^\e}{\gamma_\e^{2q +2}}\Big). 
\end{align*}
This implies for $\e_0 \in (0, 1]$ sufficiently small and any $\e \in (0, \e_0]$ the estimate 
\begin{equation}
\exp(-\vartheta_\e \frac{1}{2} \gamma_\e^{q}) J_4(\la_\e) \lqq \frac{1}{2} 
\exp(-\frac{1}{2\gamma_\e}). \label{eq: J4 fertig}
\end{equation}

\paragraph{Step 3: } For $\e_0 \in (0, 1]$ sufficiently small such that (\ref{eq: J1 fertig}) - (\ref{eq: J4 fertig}) are satisfied 
we conclude for $\e \in (0, \e_0]$ that  
\begin{align*}
\PP(\sup_{t\in [0, \si\wedge T^\e]} \|\Psi^{\e, x}_{t}\| > \gamma_\e^q) 
&\lqq \exp(-\vartheta_\e \frac{1}{2} \gamma_\e^{q}) \EE\Big[\exp(\vartheta_\e \sup_{t\lqq \si\wedge T^\e} \|\Phi^{\e, x}_{t}\|) \Big] 
\lqq 2 \exp(-\frac{1}{2 \gamma_\e}). 
\end{align*}
Note that our estimates are uniformly over all $x\in D_2(\gamma_\e, \rR)$. This finishes the proof.
\end{proof}

\bigskip

\begin{lem}[An asymptotic Campbell type estimate] \label{lem: Campbell} 
Under the hypothesis of Proposition \ref{prop: convolution estimate} 
and the notation of Step 2 of the proof of Proposition \ref{prop: convolution estimate} 
the process $Z^{\e, x} = (Z_t^{\e,x})_{t\gqq 0}$ given in (\ref{def: Z}) satisfies the following. 
There is $\e_0 \in (0, 1]$ such that $\e \in (0, \e_0]$ implies
\begin{align*}
\sup_{x\in D_2(\gamma_\e, \rR)} \EE\Big[\exp\Big(8\vartheta_\e^2 \big|Z_{\si\wedge T^\e}^{\e,x}\big|\Big)\Big] \lqq 2.\end{align*}
\end{lem}
\noindent The proof is found in Appendix \ref{app: Campbell}.

\bigskip
\subsection{Exponential estimates of the deviations of the small jump equation} 

\noindent For $\e, \gamma\in (0, 1]$, $x \in H$, $T\gqq 0$ we define the events 
\begin{align}
&\eE_{x, T}(\gamma, \e) := \{\sup_{s\in [0, T]} \|\Psi_s^{\e, x}\|\lqq \gamma\},\qquad 
\eE_{x, T}^\si(\gamma, \e) := \{\sup_{s\in [0, T\wedge \si]} \|\Psi_s^{\e, x}\|\lqq \gamma\},\label{def: curly E}\\
&\gG_{x, T}(\gamma, \e) := \{\sup_{s\in [0,T]} \|Y^\e(s;x)-u(s;x)\| \lqq \gamma \}, \\
&\gG_{x, T}^\si(\gamma, \e) := \{\sup_{s\in [0,T\wedge \si]} \|Y^\e(s;x)-u(s;x)\|\lqq \gamma \},\label{def: small deviation event T stopped} \\
&\gG_x(\gamma, \e) := \{\sup_{s\in [0,T_1]} \|Y^\e(s;x)-u(s;x)\| \lqq \gamma \}.\label{def: small deviation event} 
\end{align}
We suppress the overall dependence on $\e \in (0, 1]$. 
This subsection is dedicated to the proof of the following 
estimate used in the proof of the main result. 

\begin{prop}\label{cor: T1 Event E^c}
Let the Hypotheses (D.1)-(D.3), (S.1) and (S.2) be satisfied (for fixed $\iota$). 
Furthermore let the functions $\gamma_\cdot, \rho^\cdot, T^\cdot$ be given by (\ref{eq: scales}) 
and $\la_\cdot = \la^\iota_\cdot$ be defined in (\ref{eq: exit rate lambda}).    
Then there exists a constant $q\gqq 1$ such that if $\gamma_\cdot, \rho^\cdot, T^\cdot$ satisfy 
condition (\ref{eq: parameter limit conv}) for this value of $q$ and additionally 
\begin{equation}\lim_{\e \ra 0+} \beta_\e T^\e = \infty, \qquad \mbox{ and }\qquad \lim_{\e \ra 0+} \la_\e / \beta_\e = 0\label{eq: intensities}\end{equation} 
we have the following. 
For any $\rR\gqq \rR_0$ and $\theta\in (0,1)$ there is a constant $\e_0\in (0, 1]$ such that $\e\in (0, \e_0]$ implies  
\begin{equation}
\sup_{x\in D_2(\gamma_\e, \rR) } \EE\left[e^{\theta \la_\e T_1} \ind(\gG_x^\mathsf{c}(\frac{1}{2}\gamma_\e))\right]\lqq 
2 \exp(-\frac{1}{2 \gamma_\e}) +2 \exp(-\frac{\beta_\e T^\e}{2}).\label{eq: T1 Event E^c} 
\end{equation}
\end{prop}

\paragraph{(C) Choice of the scales: }
\begin{enumerate}
 \item For any $\alpha>0$ and $q\gqq 1$ fixed  
the scales satisfying (\ref{eq: scales}) and (\ref{eq: intensities}) are chosen as follows 
\begin{equation}\label{eq: choice of scales}
\begin{split}
\gamma_\e := \e^{\gamma^*}, \qquad \rho^\e := \e^{-\rho^*},\qquad 
\beta_\e = \nu(\rho^\e B_1^\mathsf{c}(0)) = O(\e^{\rho^*\al}\ell(\frac{1}{\e}))_{\e \ra 0}, \quad 
T^\e := \e^{-\theta^*}, 
\end{split}
\end{equation}
where $\gamma^*, \rho^* \in (0,1)$ satisfy 
\begin{align}
(2q +3) \gamma^* + (1+\al) \rho^* &< 1\label{eq: small noise scales}
\end{align}
and $\theta^* := 2 \al \rho^*$. 
Since both $q+2>0$ and $\al>0$ 
condition (\ref{eq: small noise scales}) 
is easily satisfied for sufficiently 
small positive values of $\gamma^*, \rho^*$. 
Condition (\ref{eq: small noise scales}) directly implies the 
limits (\ref{eq: parameter limit conv}) and (\ref{eq: intensities}). 

 \item For further use in Section \ref{sec: proofs} we additionally impose the conditions 
\begin{align}
\gamma^* &< \rho^*,\label{eq: intermediate noise scales}\\
\frac{\gamma*}{\al} + 3 \rho^* &< 1~, \label{eq: large noise scales} 
\end{align}
on $\gamma^*$ and $\rho^*$, which do not contradict (\ref{eq: small noise scales}) 
since (\ref{eq: large noise scales}) is of the same type and (\ref{eq: intermediate noise scales}) 
can be satisfied independently. 
Then condition (\ref{eq: intermediate noise scales}) yields 
$\lim_{\e \ra 0+} |\ln(\gamma_\e)| \frac{\e^\al}{\gamma_\e^\al} \frac{\beta_\e}{\la_\e} = 0$ 
and inequality (\ref{eq: large noise scales}) implies the limit $\lim_{\e \ra 0+} \gamma_\e \beta_\e \frac{\beta_\e}{\la_\e} = 0$. 
The latter two are used in the estimates 
(\ref{eq: small large jump estimate}) and (\ref{eq: small large jump time probability}) respectively 
of Step 3 in the proof of Proposition \ref{prop exit times} in Section \ref{sec: proofs}. 
\end{enumerate}

\paragraph{The proof of Proposition \ref{cor: T1 Event E^c}: } 
Our strategy consists of two steps. First we show in Proposition~\ref{prop: noise control implies remainder control}
that for some $q\gqq 1$ and small $\e$ the stopped small perturbation event 
$\eE_{y, T^\e}^\si(\gamma_\e^q)$ implies the stopped small deviation event $\gG_{x, T^\e}^\si(\frac{1}{2} \gamma_\e)$. 
In Corollary 
The second step relates $\eE_{y, T^\e}(\gamma_\e^q)$ to $\eE_{y, T^\e}^\si(\gamma_\e^q)$ before finally 
using the estimate of $(\eE^\si_{y, T^\e})^\mathsf{c}$ in Proposition \ref{prop: convolution estimate}.
 
\begin{prop}\label{prop: noise control implies remainder control}
Let the Hypotheses (D.1)-(D.3), (S.1) and (S.2) be satisfied and for some fixed $q\gqq 1$ 
the scales $\gamma_\cdot, \rho^\cdot, T^\cdot$ be chosen as in (C).  
Then for any $\rR\gqq \rR_0$ there exists a constant $\e_0\in (0, 1]$ such that $\e \in (0, \e_0]$ and $x\in D_2(\gamma_\e, \rR)$ imply 
\begin{equation}
\eE_{x, T^\e}^\si(\gamma_\e^q) 
\subseteq \gG_{x, T^\e}^\si(\frac{1}{2} \gamma_\e). \label{eq: pushforward inclusion}
\end{equation}
\end{prop}

\medskip 

\begin{cor}\label{cor: time inclusion}
Let the hypotheses of Proposition \ref{prop: noise control implies remainder control} be satisfied. 
Then for any $\rR\gqq \rR_0$ there exists a constant $\e_0\in (0, 1]$ such that $\e \in (0, \e_0]$ and $x\in D_2(\gamma_\e, \rR)$ imply 
\begin{equation}
\eE_{x, T^\e}^\si(\gamma_\e^q) 
\subseteq \{\si > T^\e\}. \label{eq: time inclusion}
\end{equation}
\end{cor}
\noindent The proof of Corollary \ref{cor: time inclusion} is given below the proof of Proposition \ref{prop: noise control implies remainder control} 
at the end of this subsection. 

\medskip 
\begin{cor}\label{cor: noise control implies remainder control}
Let the Hypotheses of Proposition \ref{prop: noise control implies remainder control} be satisfied. 
Then for any $\rR\gqq \rR_0$ there exists a constant $\e_0\in (0, 1]$ such that $\e \in (0, \e_0]$ and $x\in D_2(\gamma_\e, \rR)$ imply 
\begin{equation}
\eE_{x, T^\e}(\gamma_\e^q) 
\subseteq \gG_{x, T^\e}(\frac{1}{2} \gamma_\e). \label{eq: pushforward inclusion unstopped}
\end{equation}
\end{cor}
\noindent The proof of Corollary \ref{cor: noise control implies remainder control} is found at the end of the current subsection. 

\medskip

\begin{rem}\label{rem: monotonicity}
The proof of Proposition \ref{prop: noise control implies remainder control} given below 
is based on Gronwall estimates in Lemma \ref{lem: remainder short time scales} and 
\ref{lem: remainder large times}. They yield estimates with right-hand sides 
which are monotonically growing as a function of an ($\e$-independent) time argument $T$ 
and imply the inclusion (\ref{eq: pushforward inclusion}) for any fixed $T$ instead of $T^\e$ 
when $\e$ is sufficiently small. 
In Lemma \ref{lem: remainder large times} we show 
the stronger statement that (\ref{eq: pushforward inclusion}) 
is valid for the $\e$-dependent argument $T = T^\e$ which grows 
monotonically $T^\e \ra \infty$ as $\e \ra 0+$.  
We stress that by the mentioned monotonicity in $T$ 
that (\ref{eq: pushforward inclusion}) is also valid 
for $T^\e$ being replaced by any $s\in [0, T^\e]$ 
and can be verified below line by line without difficulty. 
The stopping procedure with $\sigma$ does not affect this reasoning. 
\end{rem}

\medskip

\noindent\textit{Proof of Proposition \ref{cor: T1 Event E^c}: } 
By Corollary \ref{cor: time inclusion} there is $q\gqq 1$ such that for any $\rR\gqq \rR_0$ there is some $\e_0\in (0, 1]$ such that 
$\e \in (0, \e_0]$ implies 
\begin{equation}
\eE_{x, T^\e}^\si(\gamma_\e^q) \subseteq \{\si > T^\e\}. \label{eq: second statement}
\end{equation}
This result yields  
\begin{align*}
\eE_{x, T^\e}^\si 
&=  \eE_{x, T^\e}^\si \cap \big(\{\si > T^\e\} \cup \{\si \lqq T^\e\}\big) 
=  \big(\eE_{x, T^\e}^\si \cap \{\si > T^\e\}\big) ~\cup ~\big(\eE_{x, T^\e}^\si \cap \{\si \lqq T^\e\}\big) \\
&=  \big(\eE_{x, T^\e} \cap \{\si > T^\e\}\big) ~\cup ~\big(\eE_{x, T^\e}^\si \cap \{\si \lqq T^\e\}\big) 
= \eE_{x, T^\e} \cap \{\si > T^\e\}.
\end{align*}
Hence $\e \in (0, \e_0]$ yields 
\begin{equation}\label{eq: getting rid of sigma}
(\eE_{x, T^\e}^\si)^\mathsf{c} =  (\eE_{x, T^\e} \cap \{\si > T^\e\})^\mathsf{c} = \eE_{x, T^\e}^\mathsf{c} \cup \{\si > T^\e\}^\mathsf{c}.  
\end{equation}
The identity (\ref{eq: getting rid of sigma}) puts us in the position to prove inequality (\ref{eq: T1 Event E^c}).

Due to the independence of $Y^\e$ and $T_1$ and the statement of 
Proposition~\ref{prop: noise control implies remainder control} 
there is a constant $\e_0\in (0, 1]$ such that $\e \in (0, \e_0]$ implies (\ref{eq: pushforward inclusion}) 
and additionally due to (\ref{eq: intensities}) the 
inequality 
\begin{equation}\label{eq: beta lambda estimate}
\ln(\frac{\beta_\e}{\beta_\e-\theta \la_\e}) + \theta \la_\e T^\e 
\lqq 2 \theta \frac{\la_\e}{\beta_\e} + \theta  \frac{\la_\e}{\beta_\e} \beta_\e T^\e \lqq \frac{1}{2} (1+ \beta_\e T^\e ) \lqq \ln(2) + \frac{1}{2} \beta_\e T^\e. 
\end{equation}
With the help of (\ref{eq: beta lambda estimate}) and Remark \ref{rem: monotonicity} $\e \in (0, \e_0]$ yields  
\begin{align} 
\sup_{x\in
D_2(\gamma_\e, \rR)} \EE\Big[e^{\theta \la_\e T_1} \ind(\gG^\mathsf{c}_x(\frac{1}{2}\gamma_\e))\Big] 
&= \sup_{x\in D_2(\gamma_\e, \rR)}  \int_0^{\infty }  
\PP(\gG^\mathsf{c}_{x, s}(\frac{1}{2}\gamma_\e)) \beta_\e e^{-(\beta_\e - \theta \la_\e) s} ds \nonumber\\
&\lqq \int_0^{T^\e } \sup_{x\in D_2(\gamma_\e, \rR)}  
\PP(\eE^\mathsf{c}_{x, s}(\gamma_\e^q)) ~\beta_\e e^{-(\beta_\e - \theta \la_\e) s} ds 
+ \frac{\beta_\e e^{-(\beta_\e -\theta \la_\e)) T^\e}}{\beta_\e - \theta\la_\e} \nonumber\\[2mm]
&\lqq  \sup_{x\in D_2(\gamma_\e, \rR)}   
\PP(\eE^\mathsf{c}_{x, T^\e}(\gamma_\e^q)) + 2 e^{-\frac{1}{2}\beta_\e T^\e}.\label{eq: first exp} 
\end{align}
Using (\ref{eq: getting rid of sigma}) we apply Proposition \ref{prop: convolution estimate} and obtain 
for $\e \in (0, \e_0]$ the inequality 
\begin{align}
\sup_{x\in D_2(\gamma_\e, \rR)}   
\PP(\eE^\mathsf{c}_{x, T^\e}(\gamma_\e^q)) 
&\lqq \sup_{x\in D_2(\gamma_\e, \rR)}\PP(\eE^\mathsf{c}_{x, T^\e}(\gamma_\e^q) \cup \{\si > T^\e\}^\mathsf{c}) \nonumber\\
&\lqq \sup_{x\in D_2(\gamma_\e, \rR)}\PP((\eE^\si_{x, T^\e}(\gamma_\e^q))^\mathsf{c}) \nonumber\\
&\lqq 2 e^{-\frac{1}{2 \gamma_\e}}.\label{eq: second exp} 
\end{align}
Combining (\ref{eq: first exp}) and (\ref{eq: second exp}) we obtain the desired result. 
\begin{flushright}$\square$ \end{flushright}

\medskip 

\paragraph{The proof of Proposition \ref{prop: noise control implies remainder control}: } 
We introduce the nonlinear residuum $R^\e$ of the randomness in $Y^\e$ 
\begin{equation}\label{def: remainder}
R^{\e, x}_t := Y^\e(t;x) - u(t;x) - \Psi^{\e, x}_t, \qquad t\gqq 0, ~x\in D_2(\gamma_\e, \rR),\quad  \e\in (0, 1].
\end{equation} 
The quantity we have to control in $\gG_{x, T^\e}$ has the shape $Y^\e - u = \Psi^{\e, x} + R^{\e, x}$. 
By Proposition \ref{prop: convolution estimate} 
we have a good estimate of $\Psi^{\e,x}$. 
It is therefore natural to control the remainder term 
$R^{\e, x}$ in terms of $\Psi^{\e, x}$, which is done first for large initial values $x$ (of $Y^\e$ and $u$) 
on small time scales in Lemma \ref{lem: remainder short time scales} 
and then for initial values $x$ (of $Y^\e$ and $u$) 
close to the stable state and large time scales in Lemma \ref{lem: remainder large times}. 
Lemma \ref{lem: remainder control complete} combines the previous two lemmas 
before concluding the statement of Proposition \ref{prop: noise control implies remainder control}.  

\begin{lem}\label{lem: remainder short time scales}
Let the Hypotheses (D.1)-(D.3), (S.1) and (S.2) be satisfied and 
the scales $\gamma_\cdot, \rho^\cdot, T^\cdot$ be chosen as in (\ref{eq: choice of scales}). 
We set $s^\e:= \kappa_0 |\ln(\gamma_\e)|$, $\e\in (0,1]$, where $\kappa_0>0$ be 
given by Proposition~\ref{prop: logarithmic convergence time}. 

Then for all $\rR\gqq \rR_0$ and $K>0$ there is a constant $q\gqq 1$ such that 
in case the scales $\gamma_\cdot, \rho^\cdot, T^\cdot$ satisfy 
(\ref{eq: small noise scales}), (\ref{eq: intermediate noise scales}) and (\ref{eq: large noise scales}) 
with respect to $q$ we have the following. 
There is $\e_0 \in (0,1]$ such that for 
$\e\in (0, \e_0]$, $x\in D_1(\rR)$ and $\om \in \eE^\si_{x, T^\e}(\gamma_\e^q)$ 
we have 
\begin{equation}\label{eq: remainder large values}
\sup_{t\in [0, s^\e \wedge T^\e\wedge \si(\om)]} \|R^{\e, x}_t(\om)\| \lqq K \gamma_\e.  
\end{equation}
\end{lem}

\begin{proof} Fix $\rR\gqq \rR_0$ and $\e \in (0, 1]$ and $x\in D_1(\rR)$. 
Recall that $f: H\ra H$ are locally Lipschitz continuous, that is, for $y, u \in H$ 
\begin{equation}\label{eq: f locally Lipschitz}
\|f(y) - f(u)\| \lqq \ell^*(y, u) \|y-u\|,  
\end{equation}
for some $\ell^*: H\times H \ra (0, \infty), (y, u) \mapsto \ell^*(y, u)$ jointly continuous and bounded on bounded sets. 
Consequently, it is globally Lipschitz continuous on any of the bounded level sets $\uU^\rR$. 
The process $R^{\e, x}_t$ satisfies formally for all $x\in H$ and $\e \in (0,1]$ 
\begin{equation}\label{eq: res}
\frac{d R^{\e, x}_t }{dt} = \Delta R^{\e, x}_t + f(Y^\e(t;x)) - f(u(t;x)), \qquad R^{\e, x}_0 = 0.  
\end{equation}
Then the mild formulation of (\ref{eq: res}), the triangular inequality, 
the identity $Y^\e(t;x) - u(t;x) = R^{\e, x}_t + \Psi^{\e, x}_t$ and (\ref{eq: f locally Lipschitz}) 
yield the estimate 
\begin{align*}
\|R^{\e, x}_t \| 
&\lqq \int_0^t e^{-\La_0(t-s)} \ell^*(Y^\e(s;x), u(s;x)) \,\|R^{\e, x}_s - \Psi^{\e, x}_s\| ds.
\end{align*}
Note that $x\in D_1(\rR) \subseteq \uU^\rR$ and the positive invariance of $\uU^\rR$ under the deterministic system $u$ yield 
\[
\sup_{x\in D_2(\gamma_\e, \rR)}\sup_{t\gqq 0} \|u(t;x)\| \lqq d(\rR)  < \infty. 
\]
We define the $(\fF_t)_{t\gqq 0}$ stopping time $\si_*(\e) := \si \wedge \inf\{t>0~|~ \|R^{\e, x}_t\|>1\}$. 
Then we obtain for $t\in [0, s^\e \wedge T^\e \wedge \si_*]$ on the event $\eE_{x, s^\e \wedge T^\e}^{\si_*}(\gamma_\e^q)$ 
for any arbitrary fixed $q\gqq 1$ 
\[
\|Y^\e(t;x)\|\lqq \|u(t;x)\| + \|\Psi^{\e, x}_t\| + \|R^{\e, x}_t\| \lqq d(\rR) + 2.
\]
As a consequence, $\ell_\rR := \sup_{(y, u) \in B_{d(\rR)+2}^2(0))} \ell^*(y, u) < \infty$ implies 
\begin{align*}
e^{\La_0 t} \|R^{\e, x}_t \| 
&\lqq \ell_\rR \Big(\int_0^t e^{\La_0 s} 
\|R^{\e, x}_s\| ds + \int_0^t e^{\La_0 s}\|\Psi^{\e,x}_s\|  ds\Big).
\end{align*}
The Gronwall-Bellmann inequality applied to $e^{\La_0 t} \|R^{\e, x}_t \| 
$ with $e^{\La_0 0} \|R^{\e, x}_0\| = 0$ yields 
\begin{align*}
e^{\La_0 t} \|R^{\e,x}_t \|  
&\lqq \int_0^t \int_0^s e^{\ell_\rR (t-s)} e^{\La_0 r}\|\Psi^{\e,x}_r\|  dr ds 
\lqq \sup_{r\in [0, t]} \|\Psi^{\e, x}_r\| \int_0^t \int_0^s e^{\ell_\rR (t-s)}  e^{\La_0 r}  dr ds.
\end{align*}
The elementary calculation of the factor 
\begin{align*}
\int_0^t \int_0^s e^{\ell_\rR (t-s)}  e^{\La_0 r}  dr ds 
&= \frac{e^{\ell_\rR t}}{\ell_\rR (\ell_\rR-\La_0)} + \frac{1}{\La_0 \ell_\rR} - \frac{e^{\La_0 t}}{\La_0 (\ell_\rR-\La_0)}
\end{align*}
shows for $\kappa := \ell_\rR - \La_0>0$ on the event $\eE_{x, s^\e \wedge T^\e}^{\si_*}(\gamma_\e^q)$ the estimate 
\begin{align*}
\|R^{\e, x}_t \| &\lqq \frac{e^{\kappa t}}{\kappa^2}\sup_{r\in [0, t]} \|\Psi^{\e,x}_r\|,
\end{align*}
where $t\in [0, s^\e\wedge T^\e \wedge \si_*]$. 
We set $q := \kappa_0 \kappa +3$ and obtain for any $K>0$ a value 
$\e_0\in (0, 1]$ sufficiently small such that $\e \in (0, \e_0]$ implies 
on the event $\eE_{x, s^\e\wedge T^\e}^{\si_*}(\gamma_\e^{q})$ the desired estimate 
\begin{equation}\label{eq: remainder control}
\sup_{t\in [0, s^\e\wedge T^\e \wedge \si_*]} \|R^{\e,x}_t \| \lqq 
\frac{e^{\kappa s^\e}}{\kappa^2}\gamma_\e^{q} \lqq K \gamma_\e. 
\end{equation}
If $\e _0 \in (0, 1]$ is additionally small enough 
such that $K \gamma_\e <1$ for $\e \in (0, \e_0]$ 
we have on the event $\eE_{x, s^\e \wedge T^\e}^{\si_*}(\gamma_\e^{q})$ 
\[\inf\{t>0~|~ \|R^{\e, x}_t\| >1\} > s^\e \wedge T^\e\wedge \si,\]
which proves (\ref{eq: remainder large values}). 
\end{proof}

\begin{lem}\label{lem: remainder large times}
Let the Hypotheses (D.1)-(D.3), (S.1) and (S.2) be satisfied and 
the scales $\gamma_\cdot, \rho^\cdot, T^\cdot$ given by (C) for some $q\gqq 1$. 
Then for all $\rR\gqq \rR_0$ there exist constants $\delta_0, \delta_1, K_0>0$ such that 
for all $x\in B_{\delta_0}(\phi)$, $\e\in (0, 1]$ and $\om \in \eE_{x, T^\e}^{\si}(\delta_1)$ we have 
\begin{equation}
\sup_{t\in [0, T^\e \wedge \si(\om)]} \|R^{\e, x}_t(\om)\|\lqq K_0 \sup_{r\in [0, T^\e \wedge \si(\om)]} \|\Psi^{\e, x}_r(\om)\|.  
\end{equation}
\end{lem}

\begin{proof} The stability of $\phi$ implies that 
the linearization $\Delta v + f'(\phi) v$ 
of $\Delta u+ f(u)$ centered in $\phi$ has strictly negative maximal eigenvalues 
with strictly negative upper bound, $-\La_1<0$, say, 
in that $\lgl \Delta v + f'(\phi) v, v\rgl \lqq - \La_1 |v|^2$ for $v\in H.$
We fix $\delta_0 \in (0,1)$ such that we have additionally 
\begin{align}
&\sup_{v \in B_{\delta_0}(\phi)}\|f'(v)\| \lqq 2 \|f'(\phi)\| =: C_0,\label{eq: local boundedness in H}\\
&\sup_{v, w\in B_{\delta_0}(\phi)}\|f'(v) - f'(w)\| \lqq \frac{\La_1}{4} \label{eq: local Lipschitz boundedness in H}.
\end{align}
The stability also implies the existence of $\delta_1 \in (0,1)$ 
such that for $x\in B_{\delta_1}(\phi)$ 
\begin{equation}\label{eq: det stability}
u(t;x) \in B_{\frac{\delta_0}{4}}(\phi) \qquad t\gqq 0. 
\end{equation}
Denote for $x\in B_{\delta_1}(\phi)$ the $(\fF_t)_{t\gqq 0}$-stopping time 
$\si_* := \si \wedge \inf\{t>0~|~\|R^{\e, x}_t\| >\frac{\delta_0}{4}\}$.
The decomposition (\ref{def: remainder}) and the mean value theorem applied to (\ref{eq: res}) read  
\begin{align*}
\frac{d R^{\e, x}_t}{dt} 
&= \Delta R^{\e, x}_t+ \int_0^1 f'(u(t;x) + \theta (R^{\e, x}_t + \Psi^{\e, x}_t)) d\theta (R^{\e, x}_t + \Psi^{\e, x}_t) \\
&= \Delta R^{\e, x}_t+ f'(\phi)R^{\e, x}_t +  \int_0^1 \big(f'(u(t;x)  + \theta (R^{\e, x}_t + \Psi^{\e, x}_t))-f'(\phi)\big) d\theta R^{\e, x}_t \\
&\qquad +\int_0^1 f'(u(t;x) + \theta (R^{\e, x}_t + \Psi^{\e, x}_t)) d\theta \Psi^{\e, x}_t. 
\end{align*}
We multiply with $R^{\e, x}_t$ in $L^2(J)$ and integrate by parts. 
Then for any $\delta_1< \frac{\delta_0}{4}$ 
the event $\eE_{x, T^\e}^{\si_*}(\delta_1)$ together with the embedding $|\Psi^{\e, x}_t|_\infty \lqq \|\Psi^{\e, x}_t\|$, 
the deterministic stability (\ref{eq: det stability}) and the definition of the stopping time $\si_*$ imply  
for $t\in [0, T^\e]$ and $\theta \in [0, 1]$ the estimate 
\[
\|u(t;x) + \theta (R^{\e, x}_t + \Psi^{\e, x}_t)\| \lqq \delta_0. 
\]
Hence the inequalities (\ref{eq: local boundedness in H}) and (\ref{eq: local Lipschitz boundedness in H}) 
and the embedding $H \subseteq L^2(J)$ give on $\eE_{x, T^\e}^{\si_*}(\delta_1)$ for any $t \in [0, T^\e \wedge \si_*]$ 
the estimate 
\[
\frac{1}{2} \frac{d}{dt} |R^{\e, x}_t|^2 + \La_1 |R^{\e, x}_t|^2 \lqq \frac{\La_1}{4} |R^{\e, x}_t|^2 + C_0 |R^{\e, x}_t| |\Psi^{\e, x}_t| 
\lqq \frac{\La_1}{2} |R^{\e, x}_t|^2 + \frac{(C_0)^2}{\La_0}  |\Psi^{\e, x}_t|^2,
\]
such that we have after rearrangement  
\[
\frac{d}{dt} |R^{\e, x}_t|^2 + \La_1 |R^{\e,x}_t|^2 \lqq \frac{2(C_0)^2}{\La_1}  |\Psi^{\e, x}_t|^2.
\]
Gronwall's lemma applied to $|R^{\e, x}_t|^2$ with initial condition $|R^{\e, x}_0|^2 = 0$ yields on the event $\eE_{x, T^\e}^{\si_*}(\delta_1)$ 
for $t \in [0, T^\e \wedge \si_*]$ the estimate  
\begin{equation}\label{eq: L2 remainder estimate} 
|R^{\e, x}_t|^2 \lqq \frac{2(C_0)^2}{\La_1} |\Psi^{\e, x}_t|^2. 
\end{equation}
In order to obtain an estimate in $H$ we use the smoothing property of the heat semigroup $S$ 
and the mean value theorem as well as (\ref{eq: L2 remainder estimate}) on $\eE_{x, T^\e}^{\si_*}(\delta_1)$ 
for $t\in [0,  T^\e \wedge \si_*]$ as follows 
\begin{align*}
\|R^{\e, x}_t\| 
&\lqq C_1 \int_0^t \frac{e^{-\La_0(t-s)}}{\sqrt{t-s}} |f(Y^{\e,x}_s) - f(u(s;x))| ds  \\
&\lqq C_1 (C_0 + \frac{\La_0}{4}) \int_0^t \frac{e^{-\La_0(t-s)}}{\sqrt{t-s}} (|R^{\e, x}_s| + |\Psi^{\e, x}_s|) ds\\
&\lqq C_1 (C_0 + \frac{\La_0}{4})\Big(2\frac{(C_0)^2}{\La_0} +1\Big) 
\int_0^t  \frac{e^{-\La_0(t-s)}}{\sqrt{t-s}}  ds \sup_{r\in [0, t]} |\Psi^{\e, x}_r| ~\lqq C_2 \sup_{r\in [0, t]} \|\Psi^{\e, x}_r\|, 
\end{align*}
where $C_2 = C_1 (C_0 + \frac{\La_0}{4})\Big(\frac{2 (C_*)^2}{\La_0} +1\Big)  \int_0^\infty  \frac{\exp(-\La_0 r)}{\sqrt{r}}  dr < \infty$.  
If, in addition, $\delta_1 < \frac{1}{K_1}$ we obtain on $\eE_{x, T^\e}^{\si}(\delta_1)$ 
\[
\inf\{t>0~|~\|R^{\e, x}_t\| >\frac{\delta_0}{4}\} > T^\e\wedge \si.
\] 
Note that we have not used any specific property of $T^\e$. 
This finishes the proof. 
\end{proof}

We combine Lemma \ref{lem: remainder short time scales} and Lemma \ref{lem: remainder large times}. 
For this purpose we assume without loss of generality 
that the limit 
\begin{equation}\label{eq: short or large times}\lim_{\e \ra \infty} \frac{s^\e}{T^\e} \in \{0, \infty\}.\end{equation} 
This is justified by the choice of scales in (\ref{eq: choice of scales}) 
and Lemma \ref{lem: remainder short time scales}. 

\begin{lem}\label{lem: remainder control complete}
Let the Hypotheses (D.1) - (D.3), (S.1) and (S.2) be satisfied and 
the scales $\gamma_\cdot, \rho^\cdot, T^\cdot$ given by (\ref{eq: choice of scales}).
For the constant $q\gqq 1$ obtained in Lemma \ref{lem: remainder short time scales} 
let $\gamma_\cdot, \rho^\cdot, T^\cdot$ additionally satisfy conditions (\ref{eq: small noise scales}), (\ref{eq: intermediate noise scales}) 
and (\ref{eq: large noise scales}). 
Furthermore, we assume (\ref{eq: short or large times}).

Then for any $\rR\gqq \rR_0$ there is a constant 
$\e_0\in (0, 1]$ such that for any $\e \in (0, \e_0]$, 
$x\in D_2(\gamma_\e, \rR)$ and $\om \in \eE_{x, T^\e}^{\si}(\gamma_\e^{q})$ we have 
\begin{equation}
\sup_{t\in [0, T^\e \wedge \si(\om)]} \|R^{\e, x}_t(\om)\|\lqq \frac{1}{4} \gamma_\e. 
\end{equation}
\end{lem}

\begin{proof} 
Recall the notation $s^\e :=  \kappa_0 |\ln(\gamma_\e)|$ with $\kappa_0$ from the 
statement of Proposition \ref{prop: logarithmic convergence time}. 
Assume $\e_0 \in (0, 1]$ is sufficiently small such that $\gamma_\e \lqq \delta_0$ 
and $\gamma_\e^{q} < \delta_1$ given in Lemma \ref{lem: remainder large times}. 
In the first case $\lim_{\e \ra 0+} \frac{T^\e}{s^\e} = 0$ the result follows immediately by Lemma \ref{lem: remainder short time scales} 
for $K=\frac{1}{4}$. 

In the second case $\lim_{\e \ra 0+}\frac{s^\e}{T^\e}= 0$ 
there is $\e_0\in (0, 1]$ such that $T^\e > s^\e$ for all $\e \in (0, \e_0]$. 
Note that if $\si < s^\e$ then $s^\e \wedge \si\lqq s^\e$ and we are back in the first case. 
Thus we only have to consider the case $\si \gqq s^\e$. 
Using Lemma \ref{lem: remainder short time scales} and the stability of $\phi$ 
we fix additionally $\e_0\in (0, 1]$ small enough 
such that $\e \in (0, \e_0]$ implies for $x\in D_2(\gamma_\e, \rR)$ on $\eE_{x, T^\e}^\si(\gamma_\e^q)$ and $\{\si > s^\e\}$ 
the estimates 
\begin{align}
&\sup_{t\in [0, s^\e\wedge \si]} \|R^{\e, x}_t\| \lqq \frac{1}{9} \gamma_\e, \label{eq: short small remainder}\\
&\|u(t;x) - \phi\| \lqq \frac{1}{4} \gamma_\e \qquad \mbox{ for }t\gqq s^\e. \label{eq: det stable invariance}
\end{align}
Then (\ref{eq: short small remainder}) and (\ref{eq: det stable invariance}) give for 
$x\in D_2(\gamma_\e, \rR)$ on $\eE_{x, T^\e}^\si(\gamma_\e^q)$ and $\{\si > s^\e\}$
\begin{equation}
\|Y^\e(s^\e;x) - \phi\| 
\lqq \|u(s^\e; x) - \phi\| + \|R^{\e, x}_{s^\e}\| + \|\Psi^{\e, x}_{s^\e}\| 
\lqq \Big(\frac{1}{4} + \frac{1}{9}\Big) \gamma_\e + \gamma_\e^q \lqq \frac{1}{2} \gamma_\e. \label{eq: Y at relaxation time} 
\end{equation}
As in the proof of  Lemma \ref{lem: remainder large times} the stability of $\phi$  
implies for all $x\in B_{\delta_0}(\phi)$ that $u(t;x) \in B_{\delta_1}(\phi)$ for all $t\gqq 0$. 
In addition, the linear stability of $\phi$ gives a constant $\ell_0 \in (0, 1]$ such that 
\[
\|u(t;x)- u(t;y)\| \lqq \ell_0 \|x-y\|, \qquad \mbox{ for all } x, y\in B_{\delta_0}(\phi), t\gqq 0.   
\]
Hence we have for $\e\in (0, \e_0]$ and $x\in D_2(\gamma_\e, \rR)$
on the event $\eE_{x, T^\e}^\si(\gamma_\e^{q})\cap \{\si > T^\e\}$ 
\begin{align}
\|u(t;x) - u(t-s^\e; Y^\e(s^\e;x))\| 
&\lqq \ell_0 \|u(s^\e; x) - Y^\e(s^\e;x)\| \lqq \|R^{\e, x}_{s^\e}\| + \|\Psi^{\e, x}_{s^\e}\| 
\lqq  \frac{1}{9} \gamma_\e + \gamma_\e^q.\label{eq: flow lipschitzianity}
\end{align}
Estimate (\ref{eq: flow lipschitzianity}) 
provides for $x\in D_2(\gamma_\e, \rR)$ on $\eE_{x, T^\e}^\si(\gamma_\e^{q})\cap \{\si > s^\e\}$ 
and additionally $s^\e < t \lqq T^\e \wedge \si^\e$ the inequality 
\begin{align*}
\|R^{\e, x}_t\| 
&= \| Y^\e(t- s^\e, s^\e, Y^\e(s^\e;x)) - u(t- s^\e; u(s^\e;x))-\Psi^{\e, x}_t\|\\[2mm]
&\lqq \| Y^\e(t- s^\e, s^\e, Y^\e(s^\e;x)) - u(t- s^\e; Y^\e(s^\e;x))-\Psi^{\e, Y^\e(s^\e;x)}_{t-s^\e, s^\e}\|\\[2mm]
&\quad ~+ \|u(t- s^\e; u(s^\e;x)) - u(t- s^\e; Y^\e(s^\e;x))\| + \|\Psi^{\e, x}_{t}\| + \|\Psi^{\e, Y^\e(s^\e;x)}_{t-s^\e, s^\e}\| \\
&\lqq \sup_{s \in [s^\e, T^\e\wedge \si]} \|R^{\e, Y^\e(s^\e;x)}_{s, s^\e}\| +\frac{1}{9} \gamma_\e + \gamma_\e^q + 
\sup_{s \in [s^\e, T^\e\wedge \si]} \|\Psi^{\e, x}_{s}\| + \sup_{s \in [s^\e, T^\e\wedge \si]} \|\Psi^{\e, Y^\e(s^\e;x)}_{s-s^\e, s^\e}\|.
\end{align*}
On the other hand estimate (\ref{eq: Y at relaxation time}), the Markov property of $Y^\e$ for time $s^\e$ and  
Lemma \ref{lem: remainder large times} guarantee for $x\in D_2(\gamma_\e, \rR)$ 
on $\eE_{x, T^\e}^\si(\gamma_\e^{q})$, $\{\si > s^\e\}$ and $s^\e < t \lqq T^\e \wedge \si^\e$ 
the inequality 
\begin{align*}
\|R^{\e, x}_t\| 
&\lqq \sup_{y\in B_{\frac{1}{2} \gamma_\e}(\phi)} \sup_{s \in [0, T^\e\wedge \si-s^\e]} \|R^{\e, y}_{s}\| + \frac{1}{9} \gamma_\e + 
\gamma_\e^q + \sup_{s \in [0, T^\e\wedge \si]} \|\Psi^{\e, x}_{s}\| \\
&\qquad + \sup_{y\in B_{\frac{1}{2} \gamma_\e}(\phi)} \sup_{s \in [0, T^\e\wedge \si-s^\e]} \|\Psi^{\e, y}_{s}\|\\
&\lqq \sup_{y\in B_{\frac{1}{2} \gamma_\e}(\phi)} \sup_{s \in [0, T^\e\wedge \si]} \|R^{\e, y}_{s}\| + \frac{1}{9} \gamma_\e + 
\gamma_\e^q + 2 \sup_{y\in D_2(\gamma_\e, \rR)} \sup_{s \in [0, T^\e\wedge \si]} \|\Psi^{\e, x}_{s}\| \\
&\lqq K_0 \sup_{y\in B_{\frac{1}{2} \gamma_\e}(\phi)} \sup_{s \in [0, T^\e\wedge \si]} \|\Psi^{\e, y}_{s}\| + \frac{1}{9} \gamma_\e + 
\gamma_\e^q + 2 \sup_{y\in D_2(\gamma_\e, \rR)} \sup_{s \in [0, T^\e\wedge \si]} \|\Psi^{\e, x}_{s}\| \\
&\lqq (K_0 +2) \gamma_\e^q+ \frac{1}{9} \gamma_\e. 
\end{align*}
We note that the preceding expression is less than $\frac{1}{4} \gamma_\e$ for all $\e \in (0, \e_0]$ if 
$\e_0 \in (0, 1]$ is chosen sufficiently small. This finishes the proof. 
\end{proof}

\bigskip

\noindent \textit{Proof of Proposition \ref{prop: noise control implies remainder control}: }  
Without loss of generality we assume in the sequel $T^\e \gqq s^\e$ for all $\e \in (0, \e_0]$ for some $\e_0\in (0, 1]$.
Let the assumptions of Lemma \ref{lem: remainder control complete} be satisfied for some $\rR\gqq \rR_0$ 
and $q\gqq 1$ be given by Lemma \ref{lem: remainder short time scales}. 
 By Lemma \ref{lem: remainder control complete} 
there exists $\e_0\in (0, 1]$ such that for all $\e\in (0, \e_0]$ and $x\in D_2(\gamma_\e, \rR)$ 
we have $\PP$-a.s. 
\begin{align}
(\gG_{x, T^\e}^\si(\frac{1}{2} \gamma_\e))^\mathsf{c} 
&= \{\sup_{t\in [0, T^\e\wedge \si]} \|Y^\e(t;x)- u(t;x)\| \gqq \frac{\gamma_\e}{2} \}\nonumber\\
&= \{\sup_{t\in [0, T^\e\wedge \si]} \|R^{\e, x}_t + \Psi^{\e, x}_t\| \gqq \frac{\gamma_\e}{2}\}\nonumber\\
&\subseteq \{\sup_{t\in [0, T^\e\wedge \si]} \|R^{\e, x}_t\| \gqq \frac{\gamma_\e}{4}\} 
\cup \{\sup_{t\in [0, T^\e\wedge \si]} \|\Psi^{\e, x}_t\| \gqq \frac{\gamma_\e}{4}\}\nonumber\\
&\subseteq \{\sup_{t\in [0, T^\e\wedge \si]} \|R^{\e, x}_t\| \gqq \frac{\gamma_\e}{4}\} \cup (\eE^\si_{x, T^\e}(\gamma_\e^q))^\mathsf{c} 
\subseteq (\eE^\si_{x, T^\e}(\gamma_\e^q))^\mathsf{c}.\label{eq: sigma inclusion}
\end{align}
This finishes the proof of Proposition \ref{prop: noise control implies remainder control}. \begin{flushright}$\square$\end{flushright}

\noindent \textit{Proof of Corollary \ref{cor: time inclusion}: }  
Proposition \ref{prop: noise control implies remainder control} states the existence of $q\gqq 1$ 
such that for all $\rR\gqq \rR_0$ there is $\e_0\in (0, 1]$ such that $\e\in (0, \e_0]$ implies for $x\in D_2(\gamma_\e, \rR)$ 
\begin{align*}
\eE_{x, T^\e}^\si(\gamma_\e^q) \cap \{\si < T^\e\}
&= \eE_{x, T^\e}^\si(\gamma_\e^q) \cap \{\si < T^\e\}
\cap \{\sup_{t\in [0, T^\e \wedge \si]} \|Y^\e(t;x)- u(t;x)\| \lqq (1/2) \gamma_\e\}\\
&\subseteq \{\sup_{t\in [0, \si]} \|Y^\e(t;x)- u(t;x)\| \lqq (1/2) \gamma_\e\}\\
&= \{Y^\e(t;x) \in B_{\frac{1}{2} \gamma_\e}(u(t;x)) \mbox{ for all } t\in [0, \si]\}\\
&\subseteq \{Y^\e(t;x) \in \bigcup_{t\gqq 0} \in B_{\frac{1}{2} \gamma_\e}(u(s;x)) \mbox{ for all } t\in [0, \si]\}.
\end{align*}
By construction, we have that 
\[
\bigcup_{x\in D_2(\gamma_\e, \rR)} \bigcup_{t\gqq 0} B_{\gamma_\e}(u(t;x)) 
\subseteq D_1(\rR)\subseteq \uU^\rR \setminus \bigcup_{v\in \pd \uU^{\rR}} B_{\gamma_\e}(v). 
\]
In particular, we obtain 
\begin{align*}
\eE_{x, T^\e}^\si(\gamma_\e^q) \cap \{\si < T^\e\}
&\subseteq \{Y^\e(\si;x) \in \uU^{\rR} \setminus \bigcup_{v\in \pd \uU^{\rR}}B_{\gamma_\e}(v)\}.
\end{align*}
However, by definition of $\si$ it is clear that $Y^\e(\si;x) \in (\uU^\rR)^\mathsf{c}$. 
Therefore for $\e_0 \in (0, 1]$ sufficiently small, $\e\in (0, \e_0]$ implies 
the desired result 
\begin{align}\label{eq: emptyset}
\eE_{x, T^\e}^\si(\gamma_\e^q) \cap \{\si < T^\e\} = \emptyset. 
\end{align}
\begin{flushright}$\square$\end{flushright}

\noindent \textit{Proof of Corollary \ref{cor: time inclusion}: }  
Combining (\ref{eq: sigma inclusion}) and (\ref{eq: emptyset}) ensures a constant $q\gqq 1$ such that for any $\rR\gqq \rR_0$ there is 
$\e_0 \in (0, 1]$ such that $\e \in (0, \e_0]$ yields 
\[
\eE_{x, T^\e}(\gamma_\e^q) \subseteq \gG_{x, T^\e}(\frac{1}{2} \gamma_\e). 
\]
\begin{flushright}$\square$\end{flushright}

\bigskip

\section{The geometric structure of the large jumps dynamics}\label{sec: proofs}
\subsection{The models of the exit times and exit locus}\label{sec: models}

We now construct on $\mathbf{\Om}:= (\Om, \aA, \PP, (\fF_t)_{t\gqq 0})$ 
the random variables $(\fS^\iota(\e))_{\e\in (0, 1]}$ of Theorem \ref{main result} 
and $(\fK^\iota(\e))_{\e \in (0, 1]}$ of Theorem \ref{main result 2}. 
\begin{defn} For given scales $\rho^\cdot$ and $\gamma_\cdot$ in (C), 
$B_j^\diamond(\e) := \{\e W_j \in \jJ^{(D^\iota)^\mathsf{c}}(\phi^\iota)\}$, and the arrival times $T_k$  of $W_k$
given in (\ref{def: large jumps 2}), we define for $\e \in (0, 1]$ 
\begin{align*}
\bar \fS^\iota(\e) &:= \sum_{k=1}^\infty T_k \prod_{j=1}^{k-1} (1-\ind(B^\diamond_j)) \ind(B^\diamond_k),\\[2mm]
\fK^\iota(\e) &:= \sum_{k=1}^\infty k \prod_{j=1}^{k-1} (1-\ind(B^\diamond_j)) \ind(B^\diamond_k).  
\end{align*}

\end{defn}

\begin{lem}\label{lem: candidate is exponential}
For given scale $\rho^\cdot$ in (\ref{eq: scales}) and any $\e \in (0,1]$ 
the random variable $\bar \fS^\iota(\e)$ is exponentially distributed with rate $\la^\iota_\e$ 
and the random variable $\fK^\iota(\e)$ is geometrically distributed with rate $\PP(B^\diamond) = \la^\iota_\e /\beta_\e$. 
In particular $\fS^\iota (\e) := \la^\iota_\e \,\bar \fS^\iota(\e)$ is exponentially distributed with rate~$1$. 
\end{lem}

The proof is elementary and provided in Appendix \ref{app: precise models}. 	

\subsection{Exit events and their estimates} \label{sec: events}

Recall the arrival times $T_k = t_1 + \dots + t_k$ of $W_k$ from (\ref{def: large jumps 2}). 
The following events are the building blocks of the first exit events. For $x\in H$, $\rR\gqq\rR_0$ and a given 
rate $\gamma: (0, 1) \ra (0,1)$ with $\gamma_\e \ra 0$ as $\e \ra 0$ we define for $j\in \NN$ 
\begin{align*}
A_x^{(j)}&:= \{Y^\e(t; x) \in D_2^\iota(\gamma_\e, \rR)~\mbox{ for } t\in [0, t_j] 
\mbox{ and }Y^\e(t_j; x) + G(Y^\e(t_j; x), \e \Delta_{t_j} L) \in D_2^\iota(\gamma_\e, \rR) \},\\[2mm]
B_x^{(j)} &:= \{Y^\e(t; x) \in D_2^\iota(\gamma_\e, \rR)~\mbox{ for } t\in [0, t_j] 
\mbox{ and }Y^\e(t_j; x) + G(Y^\e(t_j; x), \e \Delta_{t_j} L) \notin D_2^\iota(\gamma_\e, \rR) \},\\[2mm]
C_x^{(j)} &:= \{Y^\e(t; x) \notin D_2^\iota(\gamma_\e,\rR)~\mbox{ for some } t\in [0, t_j) \}.
\end{align*}
In order to use the (strong) Markov property in Subsection \ref{subsec: main proof} 
we identify $\Om$ with the canonical probability space given as the path space 
of the driving noise $\DD([0, \infty), H)$. 
 The shift operator $\Theta_s: \DD([0, \infty), H) \ra \DD([0, \infty), H)$ by $s>0$ 
is defined on this space by $\Theta_s \circ \om(\cdot) := \om(s+\cdot)$ for $s>0$. 
It is applied to the event $A_x^{(j)}$ by 
\begin{align*}
\Theta_s \circ A_x^{(j)} 
&=  \{Y^\e(t+s; Y^\e(s; x)) \in D_2^\iota(\gamma_\e,\rR)~\mbox{ for all } t\in (s, t_j+s) \quad \mbox{ and } \\[2mm]
&~~\quad ~Y^\e(t_j+s; Y^\e(s; x)) + G(Y^\e(t_j+s; Y^\e(s; x)), \e \Delta_{t_j+s} L) \in D_2^\iota(\gamma_\e, \rR) \}.
\end{align*}
In particular, since $t_j + T_{j-1}= T_j$ we obtain 
\begin{align}
A^j_x := \Theta_{T_{j-1}} \circ A_x^{(j)} 
&=  \{Y^\e(t, Y^\e(T_{j-1}; x)) \in D_2^\iota(\gamma_\e, \rR)~\mbox{ for all } t\in (T_{j-1}, T_j) \quad \mbox{ and }\nonumber\\[2mm]
&~~\quad Y^\e(T_j; x) + G(Y^\e(T_j; x), \e \Delta_{T_j} L) \in D_2^\iota(\gamma_\e, \rR) \}\label{eq: shifted event}
\end{align}
and define the analogous expressions $B_x^j := \Theta_{T_{j-1}} \circ B_x^{(j)}$ 
and $C_x^j := \Theta_{T_{j-1}} \circ C_x^{(j)}$ in the sense of (\ref{eq: shifted event}).  
We further define $\gG_x^{(j)}(\e, \gamma)$ by $\gG_x(\e, \gamma)$ where $T_1$ in (\ref{def: small deviation event}) 
is replaced by $t_j$ and analogously as above $\gG_x^{j}(\e, \gamma) := \Theta_{T_{j-1}} \circ \gG_x^{(j)}(\gamma, \e)$. 
By construction we have the representations
\begin{align}
&\{\tau_x = T_k\} = \bigcap_{j=1}^{k-1} A_{x}^j \cap B_{x}^k \qquad \mbox{ and }\qquad 
\{\tau_x \in (T_{k-1}, T_k)\} = \bigcap_{j=1}^{k-1} A_{x}^j \cap C_{x}^k. \label{eq: exit event estimate 2}
\end{align}

\bigskip
 \subsection{Proof of Theorem \ref{main result} and Theorem \ref{main result 2}}\label{subsec: main proof}

In this section we prove two results which are not congruent to Theorem \ref{main result} and Theorem \ref{main result 2}. 
In Proposition \ref{prop exit times} we show the statement of Theorem \ref{main result} and additionally the convergence in probability 
of the first exit locus of Theorem \ref{main result 2}. 
We apply this strategy for the sake of efficiency in order to avoid the repetition of arguments. 
Proposition \ref{prop exit locus} sharpens this result to the convergence in $L^{p}$ for $p \in (0, \al)$ 
of Theorem \ref{main result 2} by showing the respective uniform integrability. 

\begin{prop}\label{prop exit times}
Let the assumptions of Theorem \ref{main result 2} be satisfied. 
Then for any $\theta \in (0,1)$ and $c>0$ there are $\e_0, \gamma\in (0, 1]$ and $\rR\gqq \rR_0$ such that 
$\e\in (0, \e_0]$ implies for any $U \in \bB(H)$ with $m^\iota(\pd U) = 0$ that 
\begin{align*}
\sup_{x\in D_3^\iota(\e^\gamma, \rR)} \EE\Big[e^{\theta |\la^\iota_\e \tau_x^\iota(\e, \rR)- \fS^\iota(\e)|} 
\big(1+ |\ind\{X^\e(\tau;  x) \in U\} - \ind\{W_{\fK^\iota(\e)} \in \frac{1}{\e} \jJ^{U\cap (D^\iota)^\mathsf{c}}(\phi^\iota)\}|\big)
\Big] \lqq 1+c. 
\end{align*} 
\end{prop}

\noindent The statement of Proposition \ref{prop exit times} 
directly implies the statement of Theorem \ref{main result}. 

\begin{proof} The proof is organized in four consecutive steps. 
First, the strong Markov property reduces the main expression 
to four geometric sums, whose limit consists of event  
involving certain events, which are estimated 
in Step 2. In Step 3 we estimate the resulting event probabilities  
using all the previous results available and apply these results 
in Step 4 to the four sums mentioned above and conclude.  

\paragraph{Step 0: Conventions and assumptions. } 
We choose the scales $\gamma_\cdot$, $\rho^\cdot$, $T^\cdot$ 
according to (C) for $q\gqq 1$ given in Lemma \ref{lem: remainder short time scales}. 
Without loss of generality we set $\theta \in (\frac{1}{2}, 1)$. 
We use Hypothesis (S.4) and fix $c\in (0, \frac{1}{2} (1-\theta))$,  
$\rR \gqq \rR_0$ large enough and $\delta \in (0, 1]$ sufficiently small such that 
\begin{equation}\label{eq: choice of chi}
\frac{m^\iota\big(D^\iota \setminus D^\iota_3(\delta, \rR)\big)}{\mu^\iota(D^\iota)} < c.
\end{equation}
In addition, we assume $\e_0 \in (0, 1]$ is sufficiently small such that $\gamma_\e \lqq \delta$. 
Due to the ubiquitous dependence of all quantities of $\e$, $\rR$ and $\iota$ we drop these dependencies. 
For convenience we write $D_i = D^\iota_{i}(\gamma_\e, \rR)$, $i=2,3$.  

\paragraph{Step 1: Reduction to expressions based on events on $(0, T_1]$. } We start with the estimate 
\begin{align*}
&\sup_{x\in D_2} \EE\Big[e^{\theta |\la_\e \tau_x- \bar s|}
\big(1+ |\ind\{X^\e(\tau_x;  x) \in U\} - \ind\{W_{\fK^\iota} \in \frac{1}{\e} \jJ^{U \cap D^\mathsf{c}}(\phi)\}|\big)
\Big] \lqq S_{11} + S_{12} + S_{2} + S_3, 
\end{align*}
where
\begin{align*}
&S_{11} := \sum_{k=1}^{\infty} \sup_{y\in D_2} \EE\Big[e^{\theta \la_\e |\tau_y- T_k|} \\
&\qquad \qquad \qquad \qquad\quad 
\cdot\ind\{\tau_y = T_k\}\cap \{s = T_k\} \big(1+ |\ind\{X^\e(\tau_y;  y) \in U\} - \ind\{W_{\fK^\iota} \in \frac{1}{\e} \jJ^{U \cap D^\mathsf{c}}(\phi)\}|\big)\Big],\\
&S_{12} := 2 \sum_{k=1}^{\infty} \sup_{y\in D_2} \EE\Big[e^{\theta \la_\e |\tau_y- T_k|} 
\ind\{\tau_y \in (T_{k-1}, T_k)\}\cap \{s = T_k\}\Big],\\ 
&S_{2} := 2 \sum_{k=1}^{\infty} \sum_{\ell=1}^{k-1} \sup_{y\in D_2} \EE\Big[e^{\theta \la_\e |\tau_y- T_k|} 
\ind\{\tau_y \in (T_{\ell-1},  T_\ell]\}\cap \{s = T_k\}\Big],\\
&S_{3} := 2 \sum_{k=1}^{\infty} \sum_{\ell=k+1}^{\infty} \sup_{y\in D_2} \EE\Big[e^{\theta \la_\e 
|\tau_y- T_k|} 
\ind\{\tau_y \in (T_{\ell-1},  T_\ell]\}\cap \{s = T_k\}
\Big].
\end{align*}
In the sequel we estimate the preceding expressions 
using the representations in (\ref{eq: exit event estimate 2}) 
and the strong Markov property with respect to the $(\fF_t)_{t\gqq 0}$-stopping times $T_k$. 

\paragraph{$\mathbf{S_{11}}$: } The term $S_{11}$ is treated first since it is the only one of order $O(1)_{\e \ra 0}$, 
while all other expressions are $o(1)_{\e \ra 0}$. We denote the symmetric difference 
$E_1 \bigtriangleup E_2 := (E_1 \setminus E_2) \cup (E_2 \setminus E_1)$ for events $E_1, E_2$. 
In the sequel we repeatedly use strong Markov estimates of the following type 
\begin{align*}
&\EE\Big[\ind\big(\{\tau_y = T_k\}\cap \bigcap_{j=1}^{k-1} A^\diamond_j \cap B^\diamond_k\big) 
\big(1+ |\ind\{X^\e(T_k;  y) \in U\} - \ind\{W_{k} \in \frac{1}{\e} \jJ^{U \cap D^\mathsf{c}}(\phi)\}|\big) \Big]\\ 
&\lqq \EE\Big[
\ind\Big(\bigcap_{j=1}^{k-1} A^j_y \cap A^\diamond_j\Big) 
\ind\Big( B^k_{y} \cap B^\diamond_k\Big) 
\Big(1+ \ind\{X^\e(T_k;  y) \in U\} \bigtriangleup \{W_{k} \in \frac{1}{\e} \jJ^{U \cap D^\mathsf{c}}(\phi)\}\Big)\Big]\\
&= \EE\Big[
\ind\Big(\bigcap_{j=1}^{k-1} A^j_{y} \cap A^\diamond_j\Big) 
\EE\Big[ \ind\Big( B^k_{y} \cap B^\diamond_k\Big) 
\Big(1+ \ind\{X^\e(T_k;  y) \in U\} \bigtriangleup \{W_{k} \in \frac{1}{\e} \jJ^{U \cap D^\mathsf{c}}(\phi)\}\Big)
~|~\fF_{T_{k-1}}\Big] \Big]\\
&= \EE\Big[
\ind\Big(\bigcap_{j=1}^{k-1} A^j_{y} \cap A^\diamond_j\Big) 
\EE_{X^\e(T_{k-1};y)}\Big[\ind\big( B^k_{\cdot} \cap B^\diamond_k\big) \big(1+ \ind\{X^\e(T_k;  y) \in U\} \bigtriangleup \{W_{k} \in \frac{1}{\e} \jJ^{U \cap D^\mathsf{c}}(\phi)\}\big)
\Big] \Big]\\
&\lqq \EE\Big[ 
\ind\Big(\bigcap_{j=1}^{k-1} A^j_{y} \cap A^\diamond_j\Big)\Big] 
\sup_{y\in D_2} \EE\Big[\ind\big(B_{y} \cap B^\diamond\big)
\big(1+ \ind\{X^\e(T_k;  y) \in U\} \bigtriangleup \{W_{k} \in \frac{1}{\e} \jJ^{U \cap D^\mathsf{c}}(\phi)\}\big)\Big].
\end{align*}
The $(k-1)$-fold iteration of this argument yields  
\begin{equation}
\begin{split}\label{eq: S21}
S_{11} 
&\lqq \sum_{k=1}^\infty \sup_{y\in D_2} \PP( A_{y} \cap A^\diamond)^{k-1} 
\sup_{y\in D_2} \EE\Big[\ind\big(B_{y} \cap B^\diamond\big)
\big(1+ \ind\{X^\e(T_k;  y) \in U\} \bigtriangleup \{W_{k} \in \frac{1}{\e} \jJ^{U \cap D^\mathsf{c}}(\phi)\}\big)\Big]\\
&= \frac{\sup_{y\in D_2} \EE\Big[\ind\big(B_{y} \cap B^\diamond\big)
\big(1+ \ind\{X^\e(T_k;  y) \in U\} \bigtriangleup \{W_{k} \in \frac{1}{\e} \jJ^{U \cap D^\mathsf{c}}(\phi)\}\big)\Big]}
{1- \sup_{y\in D_2} \PP( A_{y} \cap A^\diamond)}.
\end{split}
\end{equation}
\paragraph{$\mathbf{S_{12}}$: } The remaining diagonal term is estimated as follows   
\begin{align*}
S_{12} 
&\lqq 2 \sum_{k=1}^{\infty}  \sup_{y\in D_2} \EE\Big[e^{\theta \la_\e t_k} 
\ind\Big(\bigcap_{j=1}^{k-1} \big(A^j_{y}\cap A^\diamond_j\big) \cap 
\big(C^k_{y}\cap B_k^\diamond\Big)\Big]. 
\end{align*}
For $k\gqq 1$ we obtain by the analogous strong Markov arguments as for the term $S_{11}$   
\begin{align*}
\sup_{y\in D_2} \EE\Big[e^{\theta \la_\e t_k} 
\ind\Big(\bigcap_{j=1}^{k-1} \big(A^j_{y}\cap A^\diamond_j\big) \cap 
\lqq \sup_{y\in D_2} \PP(A_y \cap A^\diamond)^{k-1} 
\sup_{y\in D_2} \EE\Big[e^{\theta \la_\e T_1} \ind \big(C_{y}\cap B^\diamond\big)\Big],
\end{align*}
such that 
\begin{equation} \label{eq: S22}
\begin{split}
S_{12} &\lqq 2 \sup_{D_2} \EE\Big[e^{\theta \la_\e T_1} \ind\big(C_{y}\cap B^\diamond\big)\Big] \sum_{k=1}^{\infty} \sup_{D_2} \PP(A_y \cap A^\diamond)^{k-1} 
\lqq \frac{2 \sup_{y\in D_2} \EE\Big[e^{\theta \la_\e T_1} \ind\big(C_{y}\big)\Big] }{1-\sup_{y\in D_2} \PP(A_y \cap A^\diamond)}.
\end{split}
\end{equation}
\paragraph{$\mathbf{S_2}$: } The estimate of $\{\tau_y \in (T_{\ell-1}, T_{\ell}]\}$ and the representation of $\{s = T_k\}$ yield
\begin{align*}
S_{2} &\lqq 2 \sum_{k=1}^{\infty} \sum_{\ell=1}^{k-1}\sup_{y\in D_2} \EE\Big[
e^{\theta \la_\e (t_{\ell} + \dots + t_k)} 
\ind\Big(\bigcap_{j=\ell+1}^{k-1} A^\diamond_j \cap B^\diamond_k\Big) 
\ind\Big(\bigcap_{j=1}^{\ell-1} \big(A^j_{y} \cap A^\diamond_{j}\big) 
\cap \big((B^\ell_{y}\cup C^\ell_{y}) \cap A^\diamond_\ell\big) \Big) \Big].
\end{align*}
For each of the summands $k\in \NN$ and $k-1 \gqq \ell \gqq 1$ 
we combine the mutual independence 
of the families $(T_k)_{k\in \NN}$ and $(W_k)_{k\in \NN}$ 
with the analogous strong Markov estimate and obtain 
\begin{align*}
&\sup_{y\in D_2} \EE\Big[
e^{\theta \la_\e (t_{\ell} + \dots + t_k)} \ind\Big(\bigcap_{j=\ell+1}^{k-1} A^\diamond_j \cap B^\diamond_k\Big)
\ind\Big(\bigcap_{j=1}^{\ell-1} \big(A^j_{y} \cap A^\diamond_{j}\big) 
\cap \big(B^\ell_{y}\cup C_{y}^\ell\big) \cap A^\diamond_\ell \Big) \Big] \\
&\lqq (1- \PP(A^\diamond)) \EE\big[e^{\theta \la_\e T_1}\big] \sup_{y\in D_2} \PP\big(B_{y}\cup C_{y}\big)
~\EE\big[e^{\theta \la_\e T_1}\big]^{k-1} \PP(A^\diamond)^{k-1} \\
&\qquad\qquad \qquad \qquad \qquad\qquad \qquad \qquad  \sup_{y\in D_2} \PP(A_y \cap A^\diamond)^{\ell-1} 
\Big(\EE\big[e^{\theta \la_\e T_1}\big] \PP(A^\diamond)\Big)^{-(\ell-1)}.
\end{align*}
Obviously we have 
\begin{equation*}
\sup_{y\in D_2}  \PP(A_y \cap A^\diamond) \lqq \EE\big[e^{\theta \la_\e T_1}\big] \PP(A^\diamond), 
\end{equation*} 
such that for any $k\gqq 1$ 
\begin{align*}
\sum_{\ell=1}^{k-1} \sup_{y\in D_2} \PP(A_y \cap A^\diamond)^{\ell-1} 
\Big(\EE\big[e^{\theta \la_\e T_1}\big] \PP(A^\diamond)\Big)^{-(\ell-1)} \lqq k-1, 
\end{align*}
and hence 
\begin{align}
S_{2} 
&\lqq 2\PP(B^\diamond)\sup_{y\in D_2} \PP\big((B_{y}\cup C_{y} )\cap A^\diamond\big)
~\EE\big[e^{\theta \la_\e T_1}\big] \sum_{k=1}^\infty (k-1) \EE\big[e^{\theta \la_\e T_1}\big]^{k-1} \PP(A^\diamond)^{k-1} \nonumber\\
&= 2 \PP(B^\diamond) \EE\big[e^{\theta \la_\e T_1}\big] \frac{\sup_{y\in D_2} \PP\big((B_{y}\cup C_{y} )\cap A^\diamond\big)
}{\big(1-\EE\big[e^{\theta \la_\e T_1}\big] \PP(A^\diamond)\big)^2}. 
\label{eq: S1}
\end{align}

\paragraph{$\mathbf{S_3}$: } Due to the doubly infinite summation $S_3$ turns out to be the cumbersome case here. 
We rewrite $S_3$ in terms of the events 
\begin{align*}
S_{3} 
&= 2 \sum_{k=1}^{\infty} \sum_{\ell=k+1}^{\infty} \sup_{y\in D_2}\EE\Big[e^{\theta \la_\e (t_{k}+\dots + t_\ell)} 
 \ind\Big( \bigcap_{j=1}^{k-1} \big(A^j_{y}\cap A^\diamond_j\big)\Big) \ind\Big(A^k_{y} \cap B^\diamond_k\Big) 
 \ind\Big(\bigcap_{j=k+1}^{\ell-1} A^j_{y}  
\cap \big(B^\ell_{y}\cup C^\ell_{y}\big)\Big) \Big].
\end{align*}
The strong Markov estimates as in $S_{11}$ yield for all $\ell \gqq k+1$ for each summand the following upper bound   
\begin{multline*} 
\sup_{y\in D_2}\EE\Big[e^{\theta \la_\e (t_{k}+\dots + t_\ell)} 
 \ind\Big( \bigcap_{j=1}^{k-1} \big(A^j_{y} \cap A^\diamond_j\big)\Big) 
 \ind\Big(A^k_{y} \cap B^\diamond_k\Big) \ind\Big(\bigcap_{j=k+1}^{\ell-1} A^j_{y}  
\cap \big(B^{\ell}_{y}\cup C^\ell_{y}\big)\Big) \Big] \\
\lqq \sup_{y\in D_2}\PP\big(A_{y} \cap A^\diamond\big)^{k-1} \sup_{y\in D_2} \PP\big(A_{y} \cap B^\diamond\big)
\sup_{y\in D_2} \EE\Big[e^{\theta \la_\e T_1}  \ind(A_y)\Big]^{\ell -k} 
\sup_{y\in D_2} \EE\Big[e^{\theta \la_\e T_1} \ind(B_{y}\cup C_{y}) \Big].
\end{multline*}
Assuming that $ \sup_{y\in D_2} \EE\big[e^{\theta \la_\e T_1} \ind(A_y)\big]<1$ 
for $\e \in (0, \e_0]$ for $\e_0 \in (0, 1]$ sufficiently small,  
which we verify in estimate (\ref{eq: eTA}) of Step 3, we obtain 
\begin{align}
&S_3/2\nonumber\\
&\lqq \sum_{k=1}^{\infty} \sup_{D_2} \PP\big(A_{y} \cap A^\diamond\big)^{k-1} 
\bigg(\sup_{D_2} \PP\big(A_{y} \cap B^\diamond\big) 
\EE\Big[e^{\theta \la_\e T_1 }\ind\Big(B_{y}\cup C_{y}\Big)\Big] 
\sum_{\ell=k+1}^{\infty} \sup_{D_2} \EE\Big[e^{\theta \la_\e T_1}  \ind(A_y)\Big]^{\ell -k} \bigg)\nonumber\\
&= \frac{\sup_{D_2} \PP\big(A_{y} \cap B^\diamond\big)}{1- \sup_{D_2} \PP(A_{y} \cap A^\diamond)} 
\bigg(\frac{\sup_{D_2}\EE\big[e^{\theta \la_\e T_1}  \ind(A_y)\big] 
 \sup_{y\in D_2} \EE\big[e^{\theta \la_\e T_1}\ind\big(B_{y}\cup C_{y}\big)\big]}
{1-\sup_{D_2}\EE\big[e^{\theta \la_\e T_1} \ind(A_y)\big]}\bigg)\nonumber\\
&\lqq \EE\big[e^{\theta\la_\e T_1}\big] \bigg(\frac{\sup_{D_2} \PP\big(A_{y} \cap B^\diamond\big)  
 \sup_{y\in D_2} \EE\big[e^{\theta \la_\e T_1}\ind\big(B_{y}\cup C_{y}\big)\big]}
{(1-\sup_{D_2}\EE\big[e^{\theta \la_\e T_1} \ind(A_y)\big])^2}\bigg).
\label{eq: S3}
\end{align}

\paragraph{Step 2: Precise estimates of the events on $(0, T_1]$. } 
\paragraph{Claim 1: } For $y\in D_2$ it follows that 
\begin{align}
\ind(A_y) 
&\lqq \ind\{\e W_1 \in \jJ^{D}(\phi)\} + 
\ind\{\| \e W_1\| >\frac{\gamma_\e}{2}\} \ind\{T_1 <  \kappa_0 |\ln(\gamma_\e)|\} + \ind( \gG_y^\mathsf{c}),\label{eq: estimate A}\\
\ind(B_y) 
&\lqq \ind\{\e W_1 \in \jJ^{D^\mathsf{c} }(\phi)\} 
+\ind\{\e W_1 \in \jJ^{D \setminus D_3}(\phi)\} + \ind\{\| \e W_1\| >\frac{\gamma_\e}{2}\} \ind\{T_1 <  \kappa_0 |\ln(\gamma_\e)|\} \nonumber\\
&\qquad + \ind\{T_1 < \kappa_1 \gamma_\e\}+ \ind( \gG_y^\mathsf{c}),\label{eq: estimate B}\\
\ind(C_y) &\lqq \ind\{T_1 < \kappa_1 \gamma_\e\} + \ind( \gG_y^\mathsf{c}).\label{eq: estimate C}
\end{align}
\paragraph{Proof of Claim 1: } 
We prove (\ref{eq: estimate A}): For $y\in D_2$ we have by construction for $a = 5 \vee g_1(\rR)$
\begin{align}
\ind(A_y) 
&\lqq \ind(A_y) \ind(\gG_y) \ind\{\| \e W_\e\| > \frac{\gamma_\e}{a}\} 
+ \ind\{\|\e W_\e\|\lqq \gamma_\e\}+ \ind(\gG_y^\mathsf{c}) \nonumber\\
&\lqq \ind(A_y) \ind(\gG_y) \ind\{\| \e W_\e\| > \frac{\gamma_\e}{a}\} \{T_1 \gqq \kappa_0 |\ln(\gamma_\e)|\} \nonumber\\
&\quad + \ind(A_y) \ind(\gG_y) \ind\{\| \e W_\e\| > \frac{\gamma_\e}{a}\} \{T_1 < \kappa_0 |\ln(\gamma_\e)|\}
+ \ind\{\|\e W_\e\|\lqq \frac{\gamma_\e}{a}\}+ \ind(\gG_y^\mathsf{c}) \nonumber\\
&\lqq \ind\{\| \e W_\e\| > \frac{\gamma_\e}{a}\} \ind\{Y^\e(T_1; y)\in B_{\frac{3}{4} \gamma_\e}(\phi)\} 
\ind\{\e W_1 \in \jJ^{D_2}(Y^\e(T_1;y))\} \nonumber\\
&\qquad + \ind\{\| \e W_\e\| > \frac{\gamma_\e}{a}\} \{T_1 < \kappa_0 |\ln(\gamma_\e)|\} 
+ \ind\{\|\e W_\e\|\lqq \frac{\gamma_\e}{a}\} + \ind(\gG_y^\mathsf{c})\nonumber\\
&\lqq \ind\{\| \e W_\e\| > \frac{\gamma_\e}{a}\}\ind\big(\bigcap_{y\in B_{\frac{3}{4}\gamma_\e}(\phi)}\{\e W_1 \in \jJ^{D_2}(y)\}\big) \nonumber\\
&\qquad + \ind\{\|\e W_\e\|\lqq \frac{\gamma_\e}{a}\} + \ind\{\| \e W_\e\| > \frac{\gamma_\e}{a}\} \{T_1 < \kappa_0 |\ln(\gamma_\e)|\} + \ind(\gG_y^\mathsf{c})\nonumber.
\end{align}
We use that by definition 
\begin{align*}
\bigcap_{y\in B_{\frac{3}{4}\gamma_\e}(\phi)}\{\e W_1 \in \jJ^{D_2}(y)\} 
= \bigcap_{y\in B_{\frac{3}{4}\gamma_\e}(\phi)} \{y + G(y, \e W_1) \in D_2\}.
\end{align*}
Then for $y\in B_{\frac{3}{4}\gamma_\e}(\phi)$ on $\{\|\e W_1\| \lqq \frac{\gamma_\e}{a}\}$ we obtain for 
$\e \in (0, \e_0]$ with $\e_0\in (0, 1]$ sufficiently small the estimate 
\[
\|y + G(y, \e W_1) -\phi\|\lqq \frac{3}{4}\gamma_\e + \frac{\gamma_\e}{a}  < \gamma_\e.
\]
The obvious inclusion $B_{\gamma_\e}(\phi) \in D_2$ for $\e$ sufficiently small yields 
\[
\ind\{\|\e W_\e\|\lqq \frac{\gamma_\e}{a}\} =  \ind\{\|\e W_\e\|\lqq \frac{\gamma_\e}{a}\} \ind\big(\bigcap_{y\in B_{\frac{3}{4}\gamma_\e}(\phi)}\{\e W_1 \in \jJ^{D_2}(y)\}\big).
\]
Hence the inclusion $\jJ^{D_2}(B_{\gamma_\e} (\phi))\subseteq \jJ^{D}(\phi)$ provides the desired result (\ref{eq: estimate A}) 
\begin{align*}
\ind(A_y) 
&\lqq \ind\{\e W_1 \in \jJ^{D}(\phi)\}  + \ind\{\| \e W_\e\| > \frac{\gamma_\e}{a}\}  
\ind\{T_1 <  \kappa_0 |\ln(\gamma_\e)|\} + \ind(\gG_y^\mathsf{c}).
\end{align*}
We prove (\ref{eq: estimate B}): Hypothesis (D.3) 
implies for $y \in D_2$ and $t \gqq \kappa_1 \gamma_\e$ that $u(t;y) \in D_3$. 
Hence on $\gG_y \cap \{\|\e W_1\| \lqq \frac{\gamma_\e}{a}\}$ it follows 
$Y^\e(T_1;x) + G( Y^\e(T_1;x), \e W_1) \in D_2$, which implies 
$B_y \cap \gG_y \cap \{\|\e W_1\| \lqq \frac{\gamma_\e}{a}\}\cap \{T_1 >\kappa_1 \gamma_\e\} = \emptyset$. 
Therefore, we obtain the estimate 
\begin{align*}
&\ind(B_y) \lqq \ind(B_y) \ind(\gG_y) +\ind(\gG_y^\mathsf{c})\nonumber\\
&\lqq \ind(B_y) \ind(\gG_y) \ind\{\| \e W_1\| > \frac{\gamma_\e}{a}\} \ind\{T_1 \gqq \kappa_0 |\ln(\gamma_\e)|\} 
+ \ind(B_y) \ind(\gG_y)\ind\{\| \e W_1\| > \frac{\gamma_\e}{a}\} \ind\{T_1 < \kappa_0 |\ln(\gamma_\e)|\} \nonumber\\
&~~ + \ind(B_y) \ind(\gG_y) \ind\{\| \e W_1\| \lqq \frac{\gamma_\e}{a}\} \ind\{T_1 \gqq \kappa_1 \gamma_\e\} 
+ \ind(B_y) \ind(\gG_y) \ind\{\| \e W_1\| \lqq \frac{\gamma_\e}{a}\} \ind\{T_1 <\kappa_1 \gamma_\e\}
+ \ind(\gG_y^\mathsf{c})\nonumber\\
&\lqq \ind\{Y^\e(T_1; y)\in B_{\frac{3}{4}\gamma_\e}(\phi)\} \ind\{\e W_1 \in \jJ^{D_2^\mathsf{c}}(Y^\e(T_1;y))\} 
+ \ind\{\|\e W_1\| > \frac{\gamma_\e}{a}\} \ind\{T_1 <  \kappa_0 |\ln(\gamma_\e)|\}+ 0\nonumber\\ 
&\qquad + \ind\{T_1 <\kappa_1 \gamma_\e\} + \ind(\gG_y^\mathsf{c})\nonumber\\
&\lqq \ind\{\e W_1 \in \jJ^{D_2^\mathsf{c}}(B_{\frac{3}{4}\gamma_\e}(\phi))\} + \ind\{\|\e W_1\| > \frac{\gamma_\e}{a}\}\ind\{T_1 <  \kappa_0 |\ln(\gamma_\e)|\} 
+ \ind\{T_1 <\kappa_1 \gamma_\e\} + \ind(\gG_y^\mathsf{c}).
\end{align*}
We conclude (\ref{eq: estimate B}) by the obvious inclusions 
\[
\jJ^{D_2^\mathsf{c}}(B_{\gamma_\e} (\phi))\subseteq  \jJ^{D_3^\mathsf{c}}(\phi), \quad \mbox{ and }\quad 
D_3^\mathsf{c} 
\subseteq D^\mathsf{c} \cup (D \setminus D_3).
\]
We prove (\ref{eq: estimate C}): By Hypothesis (D.3) 
$y\in D_2$ and $t\gqq \kappa_1 \gamma_\e$ imply $u(t;y) \in D_3$. 
Hence the event $\gG_y \cap \{T_1 \gqq \kappa_1 \gamma_\e\}$ 
implies that $Y^\e(t;y) \in B_{\frac{1}{2} \gamma_\e}(u(t;y)) \subseteq D_2(\gamma_\e, \rR)$ for all $t\in [\kappa_1 \gamma_\e, T_1]$ 
and $C_y\cap \gG_y \cap \{T_1 \gqq \kappa_1 \gamma_\e\} = \emptyset $. This implies the desired result 
\begin{align*}
\ind(C_y) &\lqq \ind(C_y) \ind(\gG_y) \ind\{T_1 \gqq \kappa_1 \gamma_\e\} + \ind\{T_1 < \kappa_1 \gamma_\e\} + \ind(\gG_y^\mathsf{c}) 
= \ind\{T_1 < \kappa_1 \gamma_\e\} + \ind(\gG_y^\mathsf{c}),
\end{align*}
and finishes the proof of Claim 1. 

\medskip

We recall the Lipschitz constant $K_2$ of $G$ given in (\ref{eq: Lipschitz}). 
\paragraph{Claim 2: } For $y\in D_2$ and $U\in \bB(H)$ it follows that 
\begin{align}
 &\ind(A_y \cap B^\diamond) \lqq  \ind\{\|\e W_1\| > \frac{\gamma_\e}{a}\} \ind\{T_1 <  \kappa_0 |\ln(\gamma_\e)|\} + \ind( \gG_y^\mathsf{c}),  \label{eq: A c B}\\
 &\ind(B_y \cap A^\diamond) \lqq \ind\{\e W_1 \in \jJ^{D \setminus D_3(\delta, \rR)}(\phi)\} +\ind\{\|\e W_1\| > \frac{\gamma_\e}{a}\}  
 \ind\{T_1 <  \kappa_0 |\ln(\gamma_\e)|\} \nonumber\\
 &\qquad \qquad \qquad + \ind\{T_1 < \kappa_1 \gamma_\e\} + \ind( \gG_y^\mathsf{c}), \label{eq: B c A}\\
 &\ind(B_y \cap B^\diamond) \ind\big(\{X(T_1;y) \in U\} \bigtriangleup \{\e W_1 \in \jJ^{U}(\phi)\}\big)
 \lqq \ind\{\e W_1 \in \jJ^{B_{(K_2 +1) \gamma_\e}(\pd U) \cap D^\mathsf{c}}(\phi)\} \nonumber\\
 &\qquad\qquad \qquad \qquad \qquad \quad+ \ind\{\|\e W_1\| > \frac{\gamma_\e}{a}\} \ind\{T_1 <  \kappa_0 |\ln(\gamma_\e)|\} + \ind\{T_1 < \kappa_1 \gamma_\e\}+ \ind( \gG_y^\mathsf{c}).\label{eq: B c B extra}
\end{align}
\paragraph{Proof of Claim 2: }
Estimate (\ref{eq: A c B}) is a direct consequence of (\ref{eq: estimate A}) in Claim 1.  
With the help of (\ref{eq: estimate B}) the proof of (\ref{eq: B c A}) is straightforward.   
For the proof of (\ref{eq: B c B extra}) we use the inclusion $\jJ^U(B_{\gamma_\e}(\phi))\subseteq \jJ^U(\phi)$ and  
the global Lipschitz continuity of $y \mapsto y + G(y, z)$ with Lipschitz constant $1+K_2$ as follows  
\begin{align*}
&\ind(B_y \cap B^\diamond) \ind\{X(T_1;y) \in U\} \bigtriangleup \{\e W_1 \in \jJ^{U}\} \nonumber\\
&\lqq \ind\{\e W_1 \in \jJ^{D^\mathsf{c}}(\phi) \cap \big (\jJ^{U}(B_{\gamma_\e}(\phi)) \bigtriangleup \jJ^{U}(\phi\big)\} 
 + \ind\{\|\e W_1\| > \frac{\gamma_\e}{a}\}\ind\{T_1 <  \kappa_0 |\ln(\gamma_\e)|\} \nonumber\\
&\qquad + \ind\{T_1 < \kappa_1 \gamma_\e\} + \ind( \gG_y^\mathsf{c}).
\end{align*}
Finally we see for the first term the inclusions 
\begin{align*}
\{\e W_1 \in \jJ^{D^\mathsf{c}}(\phi) \cap \big (\jJ^{U}(B_{\gamma_\e}(\phi)) \bigtriangleup \jJ^{U}(\phi\big)\} 
&\subseteq \{\e W_1 \in \jJ^{D^\mathsf{c}}(\phi) \cap \big (\jJ^{U}(\phi\big) \setminus \jJ^{U}(B_{\gamma_\e}(\phi))\}\\
&\subseteq \{\e W_1 \in \jJ^{B_{(K_2 +1) \gamma_\e}(\pd U) \cap D^\mathsf{c}}(\phi)\}.
\end{align*}
This finishes the proof of (\ref{eq: B c B extra}) and of Claim 2. 

\paragraph{Step 3: Estimates of the resulting expressions: } 

Step 2 provides the estimates to dominate respectively the term $S_{11}$ by (\ref{eq: S21}), 
$S_{12}$ by (\ref{eq: S22}), $S_{2}$ by (\ref{eq: S1}) and $S_3$ by (\ref{eq: S3}). 
In the sequel we estimate the probabilities of the events contained in these expressions. 

\paragraph{Event $A_y$: } Note that due to Hypothesis (S.2) and the choice 
$\gamma^* < \rho^*$ in (\ref{eq: small noise scales}) we have  
\begin{equation}\label{eq: small large jump estimate}
\lim_{\e \ra 0+} \PP(\|\e W_1\| > \frac{\gamma_\e}{a}) \Big(\frac{\e^\al}{(a \gamma_\e)^{\al} \beta_\e}\Big)^{-1} = 1\quad \mbox{ and } 
\lim_{\e \ra 0+} |\ln(\gamma_\e)| \Big(\frac{\e^\al}{(a \gamma_\e)^{\al}}\Big) \frac{\beta_\e}{\la_\e} = 0. 
\end{equation}
Together with the estimate (\ref{eq: estimate A}) there is a constant $\e_0 \in (0, 1]$ such that $\e \in (0, \e_0]$ implies 
\begin{align}
\sup_{y\in D_2} \PP(A_y \cap A^\diamond) &\lqq \sup_{y\in D_2} \EE\big[e^{\theta \la_\e T_1}  \ind(A_y)\big] \nonumber\\
&\lqq 1 - \frac{(1-\theta) \la_\e}{\beta_\e - \theta \la_\e} +   (1+c)  \kappa_0 |\ln(\gamma_\e)| \beta_\e
\Big(\frac{\e^\al}{(a \gamma_\e)^{\al} \beta_\e}\Big)
+ \frac{\beta_\e}{\beta_\e - \theta \la_\e} \big(e^{-\frac{1}{3 \gamma_\e}} +e^{-\frac{\beta_\e T^\e}{2}}\big) \nonumber\\
&\lqq 1-  \Big(\frac{(1-\theta)}{1 - \theta \frac{\la_\e}{\beta_\e}} - 2 c\Big)\frac{\la_\e}{\beta_\e} 
\lqq 1-  \frac{1-\theta}{2} \frac{\la_\e}{\beta_\e} \lqq 1-  (1-c) \frac{\la_\e}{\beta_\e} < 1.
\label{eq: eTA}
\end{align}

\paragraph{Event $B_y$: } Using that $\nu$ is regularly varying and the initial choice of $\rR\gqq \rR_0$ in (\ref{eq: choice of chi}) 
we obtain 
\begin{align}
\lim_{\e \ra 0} \frac{\nu( \frac{1}{\e} \jJ^{D \setminus D_3(\gamma_\e, \rR)}(\phi))}{\beta_\e} \Big(\frac{\la_\e}{\beta_\e}\Big)^{-1}
\lqq \lim_{\e \ra 0} \frac{\nu( \frac{1}{\e} \jJ^{D \setminus D_3(\gamma_{\e_0}, \rR)}(\phi))}{\nu( \frac{1}{\e} \jJ^{D}(\phi))} 
 = \frac{\mu(\jJ^{D \setminus D_3(\gamma_{\e_0}, \rR)}(\phi))}{\mu(\jJ^{D}(\phi))} 
 \lqq c.\label{eq: limit measure error}
\end{align}
In addition, by the choice of scales (\ref{eq: large noise scales}) there is $\e_0 \in (0, 1]$ such that for $\e \in (0, \e_0]$ 
we have 
\begin{equation}\label{eq: small large jump time probability}
\PP(T_1 < \kappa_1 \gamma_\e) = 1- e^{-\kappa_1 \gamma_\e \beta_\e} \lqq c \kappa_1 \gamma_\e \beta_\e \lqq c \frac{\la_\e}{\beta_\e}.  
\end{equation}
Together with (\ref{eq: small large jump estimate}) we 
apply estimate (\ref{eq: estimate B}) which gives $\e_0 \in (0, 1]$ 
such that $\e \in (0, \e_0]$ implies 
\begin{align}
&\sup_{y\in D_2} \PP(B_y \cap B^\diamond) 
\lqq \sup_{y\in D_2} \EE\big[e^{\theta \la_\e T_1} \ind(B_{y}) \big] \nonumber\\
&\lqq \EE\big[e^{\theta \la_\e T_1}\big] \PP(W_1 \in \frac{1}{\e} \jJ^{D^\mathsf{c}}(\phi)) +
\EE\big[e^{\theta \la_\e T_1} \ind\{T_1 <  \kappa_0 |\ln(\gamma_\e)|\}\big] \PP(\|\e W_1\|> \gamma_\e)\nonumber\\
&\qquad + \EE\big[e^{\theta \la_\e T_1}] \PP(T_1 < \kappa_1 \gamma_\e\} 
+ \sup_{y\in D_2} \EE\big[ e^{\theta \la_\e T_1} \ind( \gG_y^\mathsf{c})\big] \nonumber\\
&\lqq \frac{\la_\e}{\beta_\e - \theta \la_\e} 
+  \frac{(1+c) \beta_\e}{\beta_\e - \theta \la_\e}  \big(\frac{\e^\al(1- e^{-(\beta_\e -\theta \la_\e) ( \kappa_0 |\ln(\gamma_\e)|)}) }{\gamma_\e^{\al} \beta_\e}\big)
+ \frac{c\la_\e }{\beta_\e - \theta \la_\e}  + \frac{\beta_\e }{\beta_\e - \theta \la_\e}\big(e^{-\frac{1}{3 \gamma_\e}} +e^{-\frac{\beta_\e T^\e}{2}}\big) \Big) \nonumber\\
&\lqq (1+ 5c) \frac{\la_\e}{\beta_\e}. \label{eq: eTB}
\end{align}
\paragraph{Event $C_y$: } By estimate (\ref{eq: estimate C}) 
we have a constant $\e_0\in (0, 1]$ such that for $\e \in (0, \e_0]$ it holds 
\begin{align}
\sup_{y\in D_2} \PP(C_y) 
&\lqq \sup_{y\in D_2} \EE\big[ e^{\theta \la_\e T_1} \ind\big(C_y\big)\big] 
\lqq 
\sup_{y\in D_2} \EE\big[ e^{\theta \la_\e T_1} (\ind(T_1 < \kappa_1 \gamma_\e\} + \ind( \gG_y^\mathsf{c}))\big] 
 \lqq 3 c \frac{\la_\e}{\beta_\e}.\label{eq: eTC}
\end{align}

\paragraph{Events $A_y \cap B^\diamond$ and $B_y \cap A^\diamond$: } 
By (\ref{eq: A c B}) there is $\e_0 \in (0, 1]$ such 
that for $\e \in (0, \e_0]$ we obtain with the analogous calculations 
\begin{align}
\sup_{y\in D_2} \PP(A_y \cap B^\diamond)
&\lqq \sup_{y\in D_2} \EE\big[e^{\theta \la_\e T_1} \ind(A_y \cap B^\diamond)\big] 
\lqq c \frac{\la_\e}{\beta_\e}.\label{eq: Ax Bd}
\end{align}
With the help of (\ref{eq: B c A}), the regular variation of $\nu$ 
and (\ref{eq: choice of chi}) there is a constant $\e_0$ 
such that for $\e\in (0, \e_0]$ it follows that 
\begin{align}
\sup_{y\in D_2} \PP(B_y \cap A^\diamond)
&\lqq \sup_{y\in D_2} \EE\big[e^{\theta \la_\e T_1} \ind(B_y \cap A^\diamond)\big] 
\lqq c \frac{\la_\e}{\beta_\e}.\label{eq: Bx Ad}
\end{align}

\paragraph{Step 4: Concluding estimates of the sums of (\ref{eq: S21}): } 

\paragraph{Estimate $\mathbf{S_{11}}$: } 
Since $m^\iota(\pd U) = \mu(\jJ^{\pd U}(\phi^\iota)) = 0$ by assumption, the regular variation 
of $\nu$ by Hypothesis (S.2) and (\ref{eq: choice of chi})
we have $\e_0\in (0,1]$ such that $\e\in (0, \e_0]$ yields 
\begin{align}
\lim_{\e\ra 0} \PP\big(\e W_1 \in \jJ^{B_{(K_2 +1) \gamma_\e}(\pd U) \cap D^\mathsf{c}}(\phi)\big) \Big(\frac{\la_\e}{\beta_\e}\Big)^{-1} 
&= \lim_{\e\ra 0} \frac{\nu\big(\frac{1}{\e} \jJ^{B_{(K_2 +1) \gamma_\e}(\pd U)\cap D^\mathsf{c}}(\phi)\big)}{\nu(\rho^\e B_1^\mathsf{c}(0))} 
\frac{\nu(\rho^\e B_1^\mathsf{c}(0))}{\nu\big(\frac{1}{\e} \jJ^{D^\mathsf{c}}(\phi)\big)}\nonumber\\
&\lqq \frac{\mu\big(\jJ^{B_{(K_2 +1) \delta}(\pd U)\cap D^\mathsf{c}}(\phi)\big)}{\mu(\jJ^{D^\mathsf{c}}(\phi))} \lqq c. \label{eq: kleiner fehler}
\end{align}
Hence  (\ref{eq: B c B extra}),  (\ref{eq: eTB}) and (\ref{eq: kleiner fehler}) combined yield 
\begin{align*}
&\sup_{y\in D_2} \EE\Big[\ind\big(B_{y} \cap B^\diamond\big)\big(1 + \ind\{X^\e(T_1;y) \in U\}
\bigtriangleup \{\e W_1 \in \jJ^{U \cap D^\mathsf{c}}\}\big) \Big] \\
&\lqq \PP(\e W_1 \in \jJ^{D^\mathsf{c}} (\phi)) + \PP(\e W_1 \in \jJ^{B_{(K_2 +1) \gamma_\e}(\pd U) \cap D^\mathsf{c}}(\phi))+ 
\PP(T_1 <  \kappa_0 |\ln(\gamma_\e)|) \PP(\|\e W_1\| > \frac{\gamma_\e}{a})\\
&\quad +\sup_{y\in D_2} \PP\big( \gG_y^\mathsf{c}\big)\\
&\lqq (1+3c) \frac{\la_\e}{\beta_\e}.
\end{align*}
Finally, for $\e \in (0, \e_0]$ the sum $S_{11}$ given in (\ref{eq: S21}) 
satisfies due to $c \lqq \frac{1}{4}$ 
\begin{equation}\label{eq: S11 final} S_{11} \lqq \frac{1+3c}{1-c}\lqq 1+ 6c.\end{equation}

\paragraph{$S_{12}$ given by (\ref{eq: S22}): } By (\ref{eq: eTA}) and (\ref{eq: eTC})  
the sum $S_{12}$ given in (\ref{eq: S22}) satisfies for $\e\in (0, \e_0]$ 
\begin{equation}\label{eq: S12 final} 
S_{12} \lqq \frac{6 c}{1-c} \lqq 8 c.
\end{equation} 

\paragraph{$S_2$ given by (\ref{eq: S1}): } 
Using the estimates (\ref{eq: eTC}), (\ref{eq: Bx Ad}) and the choice $c\in (0, \frac{1-\theta}{2})$ 
the sum $S_2$ given in (\ref{eq: S1}) satisfies for $\e \in (0, \e_0]$ the estimate 
\begin{align}
S_2 
&\lqq \frac{8c (1+5c)}{\big((1-c) \frac{\la_\e}{\beta_\e})\big)^2}\Big(\frac{\la_\e}{\beta_\e}\Big)^2 \lqq 48c.\label{eq: S2 final}
\end{align}

\paragraph{$S_3$ given by (\ref{eq: S3}): } 
Using (\ref{eq: eTA}) 
and (\ref{eq: Ax Bd}) 
we obtain $\e_0 \in (0, 1]$ such that $\e \in (0,\e_0]$ implies 
\begin{align}
S_{3} 
&\lqq 4  \frac{c\frac{\la_\e}{\beta_\e} 4 c \frac{\la_\e}{\beta_\e}}{\big((1-c)\frac{\la_\e}{\beta_\e}\big)} 
\lqq \frac{16 c^2}{(1-c)^2} \lqq 4c.\label{eq: S3 final}
\end{align}
We finally collect (\ref{eq: S11 final}) - (\ref{eq: S3 final}) and infer the existence of 
$\e_0\in (0,1]$ such that $\e\in (0, \e_0]$ yields   
\begin{align*}
\sup_{x\in D_2} \EE\Big[e^{\theta \la_\e |\tau_x- s(\e)|} 
\big(1+ |\ind\{X^\e(\tau_x;  x) \in U\} - \ind\{W_{\fK^\iota(\e)} \in \frac{1}{\e} \jJ^{U\cap D^\mathsf{c}}(\phi)\}|\big)
\Big] \lqq 1+ 66 c.
\end{align*}
Since $c \in (0, \frac{1-\theta}{2})$ was chosen arbitrary this finishes the proof. 
\end{proof}

\bigskip

Having established the convergence in probability of the exit locus it is sufficient 
to establish the uniform integrability. We keep all the notation and the scales of the proof of Proposition \ref{prop exit times}. 

\begin{prop}\label{prop exit locus}
Under the assumptions of Proposition \ref{prop exit times} 
for any $0 < p < \al$ and  $\rR\gqq \rR_0$ 
there are $\e_0, \gamma\in (0, 1]$ and such that 
\begin{align}
\sup_{\e \in (0, \e_0]} \sup_{x\in D_2(\e^\gamma, \rR)} \EE\Big[\|X^\e(\tau; x)- (\phi + G(\phi, \e W_{\fK^\iota(\e)}))\|^{p} \Big] < \infty. \label{eq: uniform boundedness}
\end{align}
\end{prop}
\noindent The proof of Proposition \ref{prop exit locus} is given in Subsection \ref{subsec: uniform integrability} of the appendix. 
\bigskip

\textit{Proof of Theorem \ref{main result 2}: }
The convergence $\|X^\e(\tau; y) - \phi - G(\phi, \e W_{\fK^\iota(\e)})\| \ra 0$ in probability as $\e \ra 0$ 
is established in Proposition \ref{prop exit times}. 
In addition it holds true uniformly for all $y\in D_2(\e^\gamma, \rR)$.  
The uniform boundedness of Proposition \ref{prop exit locus} 
implies the uniform integrability of the family of random variables $(\|X^\e(\tau; y) - \phi - G(\phi, \e W_{\fK^\iota(\e)})\|^{p})_{\e \in (0, \e_0]}$ 
and hence its convergence as $\e \ra 0$ in $L^{p}$. 

The last statement of Theorem \ref{main result 2} follows from $\lim_{\e \ra 0} \PP(\e W_{\fK^\iota(\e)}(\e) \in U) 
= \frac{\mu(U \cap (D^\iota)^\mathsf{c})}{\mu((D^\iota)^\mathsf{c})}$ for all $U\in \bB(H)$ with $\mu(\pd U) = 0$. 
\begin{flushright}$\square$\end{flushright}

\bigskip

\section{Appendix} \label{sec: appendix}
\subsection{Proof of Lemma \ref{lem: candidate is exponential}: the law of the models.}\label{app: precise models} 

Since the family $(W_k)_{k\in \NN}$ is i.i.d. and $B^{\diamond}_k = \{\e W_k \in (D^\iota)\}$ 
we have that by construction $\fK^\iota(\e)$ is geometrically distributed 
with rate $\PP(B^\diamond) = \frac{\la^\iota_\e}{\beta_\e}$. 
Let $\theta>0$. We calculate the Laplace transform of~$\bar \fS^\iota(\e)$ 
\begin{align*}
\EE\left[e^{-\theta \bar \fS^\iota(\e)} \right] = & ~\EE\left[ e^{-\theta \sum_{k=1}^\infty T_k \prod_{j=1}^{k-1} (1-\ind(B_j^\diamond)) \ind(B_k^\diamond)} \right]
= ~\EE\left[~\prod_{k=1}^\infty e^{-\theta T_k \prod_{j=1}^{k-1} (1-\ind(B_j^\diamond)) \ind(B_k^\diamond)} \right]\\
= &~\sum_{k=1}^\infty \EE\left[ e^{-\theta T_k } \prod_{j=1}^{k-1} (1-\ind(B_j^\diamond)) \ind(B_k^\diamond)\right]
=  ~\sum_{k=1}^\infty \EE\left[ \prod_{j=1}^{k-1} e^{-\theta t_j}(1-\ind(B_j^\diamond)) e^{-\theta t_k} \ind(B_k^\diamond)\right].\\
\end{align*}
The independence of $(W_k)_{k\in \NN}$ and $(T_k)_{k\in\NN}$ as well as the stationarity of $(W_k)_{k\in \NN}$
yield that each summand takes the form
\begin{align*}
\EE\left[ \prod_{j=1}^{k-1} e^{-\theta t_j}(1-\ind(B_j^\diamond)) e^{-\theta t_k} \ind(B_k^\diamond)\right] 
&= \prod_{j=1}^{k-1} \EE\left[ e^{-\theta t_j}(1-\ind(B_j^\diamond))\right] \EE\left[e^{-\theta t_k} \ind(B_k^\diamond)\right]\\
&= \left(\EE\left[ e^{-\theta t_1}\right] (1-\PP(B_1^\diamond))\right)^{k-1} \EE\left[e^{-\theta t_1}\right] \PP(B_1^\diamond)\\[2mm]
&= \left(\frac{\beta_\e}{\theta + \beta_\e}  (1- \frac{\la^\iota_\e}{\beta_\e}) \right)^{k-1} \frac{\beta_\e}{\theta + \beta_\e} \frac{\la^\iota_\e}{\beta_\e}.
\end{align*}
Finally we conclude 
\begin{align*}
\EE\left[e^{-\theta \bar \fS^\iota(\e)} \right] = & ~\sum_{k=1}^\infty \left(\frac{\beta_\e}{\theta + \beta_\e}  (1- \frac{\la^\iota_\e}{\beta_\e}) \right)^{k-1} 
\frac{\beta_\e}{\theta + \beta_\e} \frac{\la^\iota_\e}{\beta_\e}\\
= & ~\frac{\beta_\e}{\theta + \beta_\e} \frac{\la^\iota_\e}{\beta_\e} ~\frac{1}{1-\frac{\beta_\e}{\theta + \beta_\e}  (1- \frac{\la^\iota_\e}{\beta_\e})}
= ~ \frac{\la^\iota_\e}{\beta_\e} ~\frac{1}{\frac{\theta + \beta_\e}{\beta_\e}-(1- \frac{\la^\iota_\e}{\beta_\e})} 
=  ~ \frac{\la^\iota_\e}{\theta + \la^\iota_\e} = \ha{\DEXP(\la_\e)}(\theta).
\end{align*}

\bigskip

\subsection{Proof of Lemma \ref{lem: Campbell}: a Campbell type estimate.} \label{app: Campbell}

Recall the notation from Step 2 of Proposition \ref{prop: convolution estimate}.  
For the $(\fF_t)_{t\gqq 0}$-predictable process $(H_t)_{t\gqq 0}$ 
given in (\ref{eq: H}) and $x\in D_2$ we define 
$h_x(s-, \e z):= 2 H_s \lgl\lgl \Phi^{\e, x}_{s-}, G(Y(s-;x), \e z) \rgl\rgl$. 
Consider the process 
\[
Z_t := Z_t^{\e,x} = \int_0^t \int_{\|z\|\lqq \rho^\e} h_x(s-, \e z) \ti N(dsdz).   
\]
We define the smooth function $\iI_c(r) := \sqrt{r^2 + c^2}$, $c\in (0, 1]$, with $\iI_0(r) = |r|$, 
which satisfies the following useful properties 
\begin{align}
&|r| \lqq \iI_c(r) \lqq |r|+c, \qquad r\in \RR, \label{eq: uniform Ic approx}\\
&\sup_{r\in \RR} |\frac{r}{\iI_c(r)}| = 1 \label{eq: fractional Ic approx}\\
&\iI_c(r+h) \lqq \iI_c(r) +\iI_c(h), \qquad r, h\in \RR,\label{eq: Lipschitz Ic approx}\\
&\iI_c(r)' = \frac{r}{\iI_c(r)} , \qquad r\in \RR,\nonumber\\
&\iI_c(r)'' = \frac{c^2}{\iI_c^3(r)}, \qquad r\in \RR.\nonumber
\end{align}
For $F(r) := \exp(\kappa \iI_c(r))$ for some parameter $\kappa>0$ we first obtain for all $r\in \RR$ 
\begin{align}
F'(r) &= F(r)  \frac{\kappa r}{\iI_c(r)}\label{eq: DF}\\ 
F''(r) &= F(r) \Big(\frac{\ka^2 r^2 \iI_c(r) + \kappa c^2}{\iI_c(r)^3}\Big).\label{eq: D2F}
\end{align}
Applying twice the mean value theorem, and (\ref{eq: fractional Ic approx}) - (\ref{eq: D2F}) 
we obtain for all $r, h\in \RR$ the estimate 
\begin{align}
|F(r+h) - F(r) - F'(r)h | 
&\lqq \int_0^1 \int_0^1 |F''(r+\theta' \theta h)|  d\theta' d\theta ~|h^2|\nonumber\\
&\lqq \int_0^1 \int_0^1 |F(r+\theta' \theta h) \Big(\frac{\ka^2 r^2 \iI_c(r+\theta' \theta h) + \kappa c^2}{\iI_c(r+\theta' \theta h)^3}\Big)|  d\theta' d\theta ~|h^2|\nonumber\\
&\lqq F(r) F(|h|)\big(\ka^2 + \frac{\ka}{c}\big) ~|h^2|.
\label{eq: approx inequality}
\end{align}
It\={o}'s formula for Poisson random measures then yields $\PP$-a.s. for all $t\gqq 0$ 
\begin{align*}
F(Z_t) 
& = 1 + \int_0^t \int_{\|z\| \lqq \rho^\e}  F(Z_{s-}+h(s-, \e z)) - F(Z_{s-}) \ti N(dsdz) \\
&\quad + \int_0^t \int_{\|z\| \lqq \rho^\e}   
 F(Z_{s-} +h(s-, \e z)) - F(Z_{s-}) - F(Z_{s-})  \frac{\kappa Z_{s-} h(s-, \e z)}{\iI_c(Z_{s-})} \nu(dz) ds.
\end{align*}
Let $\si$ be the $(\fF_t)_{t\gqq 0}$-stopping time defined in (\ref{def: exit time Y}) and  (\ref{def: exit time Psi}).  
Then for $\ka = \ka^\e = 8 \vartheta_\e^2 = 8 \gamma_\e^{-2q - 2}$ we have 
\begin{align*}
\sup_{s\in [0, \si \wedge T^\e]} \sup_{\|z\|\lqq \rho^\e} (\ka^\e)^2 |h(s, \e z)|^2
&\lqq  
2\frac{(\e \rho^\e)^2}{\gamma_\e^{4p+4}} \sup_{s\in [0, \si \wedge T^\e]} |H_{s-}| d(\rR)  g_1(\rR) \\
&\lqq 2d(\rR)  g_1(\rR) \frac{(\e \rho^\e)^2}{\gamma_\e^{4p+4}} 
\lqq 2 d(\rR)  g_1(\rR) \ti \Gamma(\e) \ra 0, \qquad \e \ra 0, 
\end{align*}
where $\ti \Gamma(\e) = \Gamma(\e) / T^\e$ in (\ref{eq: parameter limit conv}). 
Hence there is $\e_0\in (0, 1]$ with $\sup_{s\in [0, \si \wedge T^\e]} \sup_{\|z\| \rho^\e} |h(s, \e z)|\lqq 1$ for all $\e \in (0, \e_0]$. 
Then due to the optional stopping theorem the second term vanishes. Using $\rho^\e\gqq 1$, the constant 
$C_1 = \int_{\|z\|\lqq 1} \|z\|^2 \nu(dz) + \nu(B_1^\mathsf{c}(0))$ 
and the parametrization $c = c_\e= \gamma_\e$ we have 
\begin{align*}
\EE\Big[ F(Z_{t\wedge \si})\Big] 
&\lqq 1 + \EE\Big[\int_0^{t \wedge \si} \int_{\|z\| \lqq \rho^\e} 
\big(F(Z_{s-} +h(s-, \e z)) - F(Z_{s-}) - F(Z_{s-}) \frac{\ka Z_{s-} h(s-, \e z)}{\iI_c(Z_{s-})}\big) \nu(dz) ds\Big]\\
&\lqq 1 + \EE\Big[\int_0^{t \wedge \si} \int_{\|z\| \lqq \rho^\e} 
F(Z_{s-}) F(|h(s-, \e z|)\big((\ka^\e)^2 + \frac{\ka^\e}{c}\big) ~|h(s-, \e z|^2\nu(dz) ds\Big]\\
&\lqq 1 + \EE\Big[\int_0^{t} \int_{\|z\| \lqq \rho^\e} 
F(Z_{s- \wedge \si}) F(|h(s- \wedge \si, \e z|)\big((\ka^\e)^2 + \frac{\ka^\e}{c}\big) ~|h(s-\wedge \si, \e z)|^2\nu(dz) ds\Big]\\
&\lqq 1 + C_2\int_0^{t} \int_{\|z\| \lqq \rho^\e} 
\EE\Big[F(Z_{s- \wedge \si})\Big] F(\sqrt{2 d(\rR)  g_1(\rR) \Gamma(\e)}) \big((\ka^\e)^2 + \frac{\ka^\e}{\gamma_\e}\big) \frac{(\e \rho^\e)^2}{\gamma_\e^{4p+4}} \|z\|^2\nu(dz) ds\\
&\lqq 1 + C_3 \int_0^{t} \EE\Big[F(Z_{s- \wedge \si})\Big] ds, 
\end{align*}
where $C_2 = 2 d(\rR)  g_1(\rR)$ and $C_3 = C_1 C_2 \ti \Gamma(\e) F(1)$. Setting 
\[
\phi_\e(t) :=  \EE\Big[ F(Z_{ t\wedge \si}))\Big], \qquad t\gqq 0, 
\]
we have  
\begin{align*}
\phi_\e(t) \lqq 1 + C_3 \ti \Gamma(\e) \int_0^t \phi_\e(s)  ds, \qquad t\gqq 0.  
\end{align*}
The Gronwall-Bellman inequality yields 
$\phi_\e(t) \lqq \exp(C_3\ti \Gamma(\e) t)$ for all $t\gqq 0$, and in particular, 
$\phi_{\e}(T^\e) \lqq \exp(C_3\Gamma(\e))$. 
For $\e_0 \in (0,1]$ sufficiently small, $\e\in (0,\e_0]$ yields that 
the right-hand side is less than~$2$. We conclude by (\ref{eq: uniform Ic approx}) 
the existence of $\e_0\in (0,1]$ such that $\e\in (0, \e_0]$ implies 
\[
\EE\Big[\exp\Big(\ka^\e |Z_{\si \wedge T^\e}|\Big)\Big] \lqq \phi_{\e}(T^\e) \lqq 2. 
\]
Note that our estimates are uniformly for all $x\in D_2$. 
This finishes the proof of Lemma \ref{lem: Campbell}.

\bigskip

\subsection{Proof of Proposition \ref{prop exit locus}: uniform integrability. }\label{subsec: uniform integrability}

Fix $p \in (0, \al)$. We use the conventions in Step 0 of the proof of Proposition \ref{prop exit times}. 
Then for $x\in D_2$ 
\begin{align}
\EE\Big[\|X^\e(\tau; x)- (\phi + G(\phi, \e W_{\fK^\iota(\e)}))\|^{p} \Big]
&\lqq 3^{p} \bigg(\EE\Big[\|X^\e(\tau; x)\|^{p}\Big] + \|\phi\|^{p} + G_1(\phi) \EE\Big[ \|\e W_{\fK^\iota(\e)}\|^{p} \Big]\bigg).\label{eq: first estimate}
\end{align}
For the last term on the right-hand side of (\ref{eq: first estimate}) we obtain 
\begin{align*}
\EE\Big[ \|\e W_{\fK^\iota(\e)}\|^{p} \Big] = \sum_{k=1}^\infty \EE\Big[ \|\e W_{k}\|^{p} \Big] \PP(\fK^\iota(\e) = k) = \e^{p} \EE\Big[ \big(\e \|W_{1}\|^{p}\big) \Big],
\end{align*}
and hence for $\e_0 \in (0, 1]$ sufficiently small the regular variation of $\nu$ implies for $\e \in (0, \e_0]$ that 
\begin{align*}
\EE\Big[ \big(\e \|W_{1}\|\big)^{p} \Big] 
&= \int_{\rho^\e}^\infty r^{p-1} \PP(\|W_1\|>r) dr 
= \int_{\rho^\e}^\infty r^{p-1} \frac{\nu(r \frac{1}{\e} B_1^\mathsf{c}(0))}{\nu(\rho^\e B_1^\mathsf{c}(0))} dr 
\lqq 2 \int_{\rho^\e}^\infty r^{p-1} (\e \rho^\e r)^\al dr \\
&= 2(\e \rho^\e)^\al \int_{\rho^\e}^\infty r^{p-\al -1} dr 
= 2\frac{(\e \rho^\e)^\al}{\al - p}(\rho^\e)^{p-\al} \lqq  \frac{2 (\e_0 \rho^{\e_0})^\al}{\al - p}(\rho^{\e_0})^{p-\al}<\infty.
\end{align*}
We calculate the first term on the right side in (\ref{eq: first estimate})
\begin{align*}
\EE\Big[\|X^\e(\tau; x)\|^{p}\Big] 
&= \int_0^\infty r^{p-1} \PP(\|X^\e(\tau; x)\| >r) dr = \int_\rR^\infty r^{p-1} \PP(\|X^\e(\tau; x)\| >r) dr.
\end{align*}
Using (\ref{eq: exit event estimate 2}) 
and the same strong Markov argument as in Claim 1 we obtain for $x\in D_3$ 
\begin{align*}
&\PP(\|X^\e(\tau; x)\| >r) \\
&= \sum_{k=1}^\infty \PP(\{\|X^\e(\tau; x)\| >r\} \cap \{\tau = T_k\}) 
+ \sum_{k=1}^\infty \PP(\{\|X^\e(\tau; x)\| >r\} \cap \{\tau \in (T_{k-1}, T_k)\} )\\
&\lqq \sum_{k=1}^\infty \PP(\bigcap_{j=1}^{k-1} A_{x}^j \cap B_{x}^k \cap (\{\|X^\e(T_k; x)\| >r\}) 
+ \sum_{k=1}^\infty \PP(\bigcap_{j=1}^{k-1} A_{x}^j \cap C_{x}^k \cap \{\|X^\e(\tau_x; x)\| >r\} )\\
&\lqq \sum_{k=1}^\infty \sup_{y\in D_2} \PP(A_y)^{k-1} \sup_{y\in D_2} \PP(B_{y} \cap (\{\|X^\e(T_1; y)\| >r\}) \\
&\quad + \sum_{k=1}^\infty \sup_{y\in D_2} \PP(A_y)^{k-2} \sup_{y \in D_2} \PP(A_{y} \cap C^2_y \cap \{\|Y^\e(\tau_y; y)\| >r\}) \\
&\lqq \frac{\sup_{y\in D_2} \PP(B_{y} \cap (\{\|X^\e(T_1; y)\| >r\}) + 
\Big(\sup_{y\in D_2} \PP(\|Y^\e(\tau_y; y)\| >r)\Big)}{\sup_{y\in D_2} \PP(A_y) \big(1-\sup_{y \in D_2} \PP(A_y)\big)}.  
\end{align*}
For the first sum we have for $r> d(\rR) +2$ 
\begin{align*}
\sup_{y\in D_2} \PP(B_{y} \cap (\{\|X^\e(T_1; y)\| >r\}) 
&=\sup_{y\in D_2} \PP(B_{y} \cap (\{\|Y^\e(T_1; y) + G(Y^\e(T_1; y), \e W_1)\| >r\}) \\
&\lqq \sup_{y\in \uU^\rR} \PP(\|y + G(y, \e W_1)\| >r)\\
&\lqq \sup_{y\in \uU^\rR} \PP( G_1(y) \|\e W_1\| >r- d(\rR)- 1)\\
&\lqq \PP( g_1(\rR) \|\e W_1\| >r- d(\rR)- 1)\\
&\lqq \PP( \|W_1\| >\frac{1}{\e} \frac{r- d(\rR)- 1}{g_1(\rR)}).
\end{align*} 
Without loss of generality we fix $p'$
by $\al > p' >p> (1-\rho)\al$ and the estimate (\ref{eq: eTA}) of $A_x$ yields 
\begin{align*}
\frac{\sup_{y \in D_2} \PP(B_{y} \cap \{\|X^\e(T_1; y)\| >r\})}{1-\sup_{D_2} \PP(A_x)} 
&\lqq \frac{\EE\Big[ \|W_1\|^{p'}\Big] }{(1-c) \frac{\la_\e}{\beta_\e}} \frac{\e^{p'} G_1^{p'}(\phi)}{(r- \|\phi\|)^{p'}})\\
&\lqq \e_0^{p' - \al (1-\rho)} \frac{\EE\Big[ \|W_1\|^{p'}\Big] }{(1-c) 2 \mu(B_1^\mathsf{c}(0))}\frac{G_1^{p'}(\phi)}{(r- d(\rR) -1)^{p'}} < \infty
\end{align*}
for any $\e \in (0, \e_0]$ for $\e_0$ sufficiently small. 
For the last term we have for $r> d(\rR)+K_2 + 1$ 
\begin{align*}
\sup_{y\in D_2} \PP(\|Y^\e(\tau_y(\e, \rR); y)\| >r) = 0,
\end{align*}
since $\|Y^\e(\tau_x(\e; \rR); x)\| = d(\rR) + (K_2 + 1)\e \rho^\e\lqq d(\rR) + K_2 +1$ for $\e_0 \rho^{\e_0} \lqq 1$. 
Therefore for $\e_0\in (0,1]$ and $\e \in (0, \e_0]$ we have 
\begin{align*}
\EE\Big[\|X^\e(\tau_y(\e, \rR); y)\|^{p}\Big] 
&\lqq C \int_{d(\rR)+2}^\infty r^{p-1} \PP(\|X^\e(\tau_y(\e, \rR); y)\| >r) dr \\
&\lqq C \int_{d(\rR)+2}^\infty  \frac{1}{r^{1-p}(r- d(\rR)-1)^{p'}} dr < \infty. 
\end{align*}
This establishes the uniform integrability result (\ref{eq: uniform boundedness}). 

\bigskip
\noindent {\bf Acknowledgments:} 
This work was supported by the FAPA grant ``Stochastic dynamics of L\'evy driven systems'' 
of Universidad de los Andes, Bogot\'a, Colombia, which is greatly acknowledged. 
The author also thanks the Escuela Venezolana en Matem\'aticas 2017, namely Prof. Dr. Stella Brassesco, 
IVIC, C\'aracas, Venezuela, for the invitation to hold a virtual summer course on the subject.   
Furthermore, the author is grateful to CIMAT, Guanajuato, M\'exico, for the 
invitation to the conference Mexico-Poland, 1st Meeting in Probability, 
during which final parts of the work were completed.

\end{document}